\documentclass[12pt]{amsart}
\usepackage[utf8]{inputenc}
\usepackage[english]{babel}
\usepackage{subcaption}
\usepackage{enumerate}
\usepackage{amsfonts}
\usepackage{amsmath,amssymb}
\usepackage[makeroom]{cancel}
\usepackage{stmaryrd}
\usepackage{tikz-cd} 
\usepackage{enumerate}
\usepackage{graphicx}
\usepackage{amsthm}
\usepackage{wrapfig}
\newtheorem{theorem}{Theorem}
\newtheorem*{theorem*}{Theorem}
\newtheorem{prop}[theorem]{Proposition}
\newtheorem{conj}[theorem]{Conjecture}
\newtheorem{coro}{Corollary}[theorem]
\newtheorem{lemm}[theorem]{Lemma}
\theoremstyle{definition}
\newtheorem{defi}[theorem]{Definition}
\newtheorem*{defi*}{Definition}
\newtheorem{ques}[theorem]{Question}

\newtheorem{cor-definition}[theorem]{Corollary-Definition}
\newtheorem{rema}[theorem]{Remark}
\newtheorem{conv}[theorem]{Convention}

\numberwithin{theorem}{section}

\renewcommand{\int}[1]{\overset{\circ}{#1}}
\newcommand{\clos}[1]{\overline{#1}}

\newcommand{\quotient}[2]{{\raisebox{.2em}{$#1$}\left/\raisebox{-.2em}{$#2$}\right.}}

\numberwithin{equation}{section}
\begin{document}
\title{A new combinatorial invariant caracterizing Anosov flows on 3-manifolds} 
\author[I. Iakovoglou]{Ioannis Iakovoglou}
\address{Institut de Math\'{e}matiques de Bourgogne, Universit\'{e} de
Bourgogne, \\ \vspace{0.1cm} UMR 5584 du
CNRS, BP 47 870, 21078, Dijon Cedex, France.}
\email{e.mail: ioannis.iakovoglou@u-bourgogne.fr}

\begin{abstract}
In this paper, we describe a new approach to the problem of classification of transitive Anosov flows on $3$-manifolds up to orbital equivalence. More specifically, generalizing the notion of Markov partition, we introduce the notion of Markovian family of rectangles in the bifoliated plane of an Anosov flow. We show that any transitive Anosov flow admits infinitely many Markovian families, each one of which can be canonically associated to a finite collection of combinatorial objects, called geometric types. We prove that any such geometric type describes completely the flow up to Dehn-Goodman-Fried surgeries on a finite set of periodic orbits of the flow. As a corollary of the previous result, we show that any Markovian family can be canonically associated to a finite collection of combinatorial objects, called geometric types with cycles, each describing the flow up to orbital equivalence.
\end{abstract}
\maketitle
\vspace{-0.5cm}
\vspace{0.15cm}
\normalsize
\section{Introduction}
\subsection{General setting}
Thanks to the flourishing construction of Anosov flows in dimension 3 (for recent examples see \cite{Bowdenmann} and \cite{ClPi}) the problem of classification of transitive Anosov flows in dimension 3 is progressively gaining in interest and importance. 

To this day different approaches have been used in order to address the question of classification of Anosov flows in dimension 3: 
\begin{enumerate}
    \item By examining the action of the fundamental group on the bifoliated plane of the flow and by using the topology of the ambient manifold: 
        \begin{itemize}
            
            \item the authors of \cite{BarbotFenley} show among others that the only Anosov flow in a manifold with virtually solvable fundamental group is a suspension Anosov flow and that the only Anosov flow in a Seifert manifold is up to a finite cover a geodesic flow in the unit tangent bundle of a hyperbolic surface
            \item the authors of \cite{flowsingraphmanifolds} classify the totally periodic Anosov flows in graph manifolds
        \end{itemize}

    \item By cutting the manifold along transverse tori we can decompose the manifold and the flow into smaller pieces called plugs. When this is possible it can be done in an essentially unique way  (see \cite{plugsdecomp}). By understanding the simpler dynamics inside the plugs (see for instance \cite{plugsclassification}) and then classifying all possible plug glueings giving rise to Anosov flows (see \cite{plugs} and \cite{uniqueglueing} for more details) it would be possible to classify all Anosov flows in dimension 3. An application of this method can be found in \cite{plugsclassification}. 
\item By examining the action of the fundamental group on the line at infinity $\mathbb{R}_{\infty}$ of the bifoliated plane of an $\mathbb{R}$-covered Anosov flow, the authors of \cite{bartmann} prove that an $\mathbb{R}$-covered Anosov flow in $M^3$ is completely determined by the elements in the fundamental group of $M$ acting on  $\mathbb{R}_{\infty}$ with fixed points.  
\end{enumerate}

Despite the variety of approaches towards a classification of transitive Anosov flows in dimension 3, the question of classification remains still open. In fact, till this moment little is known about Anosov flows in toro\"{i}dal manifolds containing simultaneously Seifert and atoro\"idal pieces. It is for this reason that we would like in this paper to introduce a classification method that has not yet been thoroughly studied and  that is a priori independent from the ambient topology of the manifold, a classification \emph{by geometric types}. This classification method has already been successful in the case of structurally stable diffeomorphisms on surfaces (see \cite{Asterisque} and \cite{Beguin}), we therefore hope that it will shed some light and open new avenues towards the classification of Anosov flows.

Following the work of M.Ratner in \cite{Ratner}, we can associate to any (transitive) Anosov flow on a closed 3-manifold a Markov partition (see Definition \ref{d.markovpartition}). According to a folklore result, proven in this article, any such Markov partition describes its associated Anosov flow up to Dehn-Goodman-Fried surgeries on a finite number of periodic orbits. The previous set of periodic orbits, called the set of \emph{boundary periodic orbits} of the partition, can be explicitly described as the finite (and always non-empty) set of periodic orbits of the flow that never intersect the interior of some rectangle of the Markov partition. It is therefore possible to ``classify" Anosov flows in 3-manifolds up to specific  surgeries by their Markov partitions. 

Even better, to any Markov partition we can associate a finite number of canonical combinatorial objects, called geometric types (see Definition \ref{d.geometrictype}), each one encoding the action of the first return map on the set of rectangles of the partition.  Once again, any geometric type associated to a Markov partition describes the Anosov flow up to Dehn-Goodman-Fried surgeries on the boundary periodic orbits. It is therefore possible ``classify" Anosov flows in 3-manifolds up to specific surgeries by geometric types, which are finite combinatorial objects.

However, to every Anosov flow we can associate infinitely many Markov partitions (thus also geometric types) and Ratner's construction is unfortunately far from being canonical. Therefore, the two following questions, that constitute central parts for any classification, remain open:

\begin{ques}\label{q.combincanonical}
Can we associate any transitive Anosov flow on a closed 3-manifold to a canonical (finite) family of geometric types?
\end{ques}

\begin{ques}\label{q.combcompare}
Given two geometric types describing two families of Anosov flows. Can we decide algorithmically whether those two families contain two flows that are orbitally equivalent?
\end{ques}

The previous questions having already been answered in the case of stable diffeomorphisms on surfaces, our aim in this article will be to introduce a method for classifying Anosov flows, seeing an Anosov flow not as a diffeomorphism on a surface, but as a group action on a plane. More specifically, 
\begin{enumerate}
    \item we will define a notion of Markov partition, that we will call \emph{Markovian family} (see Definition \ref{d.markovfamily}), for a group action in the bifoliated plane of an Anosov flow 
    \item we will show that it is possible to associate to every Markovian family a finite number of geometric types 
    \item we will show that a geometric type associated to a Markovian family describes the Anosov flow up to specific surgeries 
    \item we will complete the geometric type into a combinatorial invariant, called the \emph{geometric type with cycles} describing the Anosov flow up to orbital equivalence
\end{enumerate}

\subsection{Markovian families and geometric types}
Fix $M$ a closed, orientable, connected manifold of dimension $3$, $\Phi$ a transitive Anosov flow on $M$ and $\mathcal{F}^s$, $\mathcal{F}^u$ its (weak) stable and unstable foliations.  

In \cite{Palmeira} Palmeira shows that that the universal cover of $M$ is homeomorphic to $\mathbb{R}^3$. Furthermore, in \cite{Fe1} and \cite{Ba1}, Fenley and Barbot show independently  that the lift  $\widetilde{\Phi}$ of $\Phi$ on the universal cover of $M$ is orbitally equivalent to the constant vector field $\frac{\partial}{\partial{x}}$ on $\mathbb{R}^3$. The space of orbits of $\widetilde{\Phi}$ is therefore a plane, endowed with the natural quotient of the lift of the weak stable and unstable manifolds of $\Phi$ on $\mathbb{R}^3$. In other words, the flow $\Phi$ is naturally associated to a plane $\mathcal{P}$ endowed with pair of transverse foliations $F^s$ and $F^u$. We call $(\mathcal{P},F^s,F^u)$ the \emph{bifoliated plane of $\Phi$}. 

The fundamental group of $M$ acts on $\widetilde{M}=\mathbb{R}^3$ by preserving the orbits of the lifted flow, therefore the action of $\pi_1(M)$ descends to an action of on $\mathcal{P}$. Barbot in \cite{Barbotthese} shows that the bifoliated plane $\mathcal{P}$ together with the action of $\pi_1(M)$ describes completely the flow $\Phi$ up to orbital equivalence. Therefore, when trying to understand an Anosov flow in dimension 3, we may choose to work with a group action on a 2 dimensional plane. 

It is for this reason that we were driven to define a notion of Markov partition (for a group action) on the bifoliated plane using as a starting point the projection on $\mathcal{P}$ of the lift on $\widetilde{M}$ of any Markov partition in $M$. By identifying the core properties of such a projection we defined a Markovian family as follows: 

\begin{defi*}
A \emph{Markovian family of rectangles} in $\mathcal{P}$ is a set of mutually distinct rectangles (see Definition \ref{d.rectanglesinplane}) $(R_i)_{i \in I}$ covering $\mathcal{P}$ such that 
\begin{enumerate}
\item $(R_i)_{i \in I}$ is the union of a finite number of orbits of rectangles of the action by $\pi_1(M)$ 
\item For every two rectangles $R_i,R_j$ in $(R_i)_{i \in I}$, if $\overset{\circ}{R_i } \cap \overset{\circ}{R_j} \neq \emptyset$, then $R_i \cap R_j$ is a non-trivial horizontal subrectangle of $R_i$ (or $R_j$ resp.) and a non-trivial vertical subrectangle of $R_j$ (or $R_i$ resp.) 
\item Take any point $x\in \mathcal{P}$ and any of its four quadrants defined by $\mathcal{F}^s(x)$ and $\mathcal{F}^u(x)$. For any sufficiently small neighbourhood $G$ of $x$ in this quadrant, there exists $R\in (R_i)_{i \in I}$ such that $G\subset R$
\end{enumerate} 
\end{defi*}
For a more detailed definition see Definition \ref{d.markovfamily}.

A Markovian family has many properties in common with a Markov partition -conjecturally we believe that all Markovian families can be obtained as projections of Markov partitions on the bifoliated plane-. In particular we show among others that it is possible to associate canonically any Markovian family to a family of combinatorial objects, called a geometric types:

\begin{defi*}
Take $R_1,...,R_n$ a finite number of copies of $[0,1]^2$, endowed with the horizontal and vertical foliations. For every $i\in \llbracket 1,n \rrbracket$ choose $h_i,v_i \in \mathbb{N}^*$ such that $$\sum_i h_i =\sum_i v_i$$ 

Consider now for every $i\in \llbracket 1,n \rrbracket$ a collection of $h_i$ (resp. $v_i$) mutually disjoint horizontal (resp. vertical) subrectangles of $R_i$: $H_i^1,...,H_i^{h_i}$ (resp. $V_i^1,...,V_i^{v_i}$). Finally, take any bijection $\phi$ of the two sets $$\mathcal{H}=\lbrace H^j_i| i\in\llbracket 1,n \rrbracket, j \in \llbracket 1,h_i \rrbracket  \rbrace  \text{ and } \mathcal{V}=\lbrace V^j_i| i\in\llbracket 1,n \rrbracket, j \in \llbracket 1,v_i \rrbracket  \rbrace$$ and $u$ a function from the set of rectangles $\mathcal{H}$ to $\lbrace -1,+1\rbrace$.

The data $(R_1,...,R_n,(h_i)_{i \in \llbracket 1,n \rrbracket}, (v_i)_{i\in \llbracket 1,n \rrbracket}, \mathcal{H}, \mathcal{V},\phi, u)$ will be called a \emph{geometric type}. Two geometric types will be called \emph{equivalent} if there exists a homeomorphism respecting $\phi$ and $u$ between the two (see Definitions \ref{d.geometrictype} and \ref{d.equivalentgeomtypes} for more details). 
\end{defi*}

\textbf{Theorem A}\textit{
Let $M$ be an orientable, closed and connected 3-manifold and $\Phi$ a transitive
Anosov flow on $M$. To any Markovian family $\mathcal{R}$ of $\Phi$ we can associate canonically a unique and finite class of equivalent geometric types.}  
\vspace{0.22cm}

The notion of Markovian family generalizes the notion of Markov partition. The advantage of working with Markovian families instead of Markov partitions is that Markov partitions are intimitely connected with dimension 3, whereas Markovian families with dimension 2. Furthermore, we believe that by adapting the methods of \cite{Asterisque} and \cite{Beguin} (applied here for a group action on the plane and not a diffeomorphism on a closed surface) it will be possible to produce and compare a finite number of canonical Markovian families and thus a finite number of geometric types, which will in turn provide a classification. 

Going in this direction, in \cite{Iabola} we have shown that we can canonically associate any generalized Bonatti-Langevin flow (see \cite{Barbotbola}) to a canonical Markovian family and thus to a finite number of canonical geometric types. Furthermore, in \cite{realisablegeomtypes} we give a necessary and sufficient condition for a geometric type to be associated to a transitive Anosov flow in dimension 3. 

The main goal of this article is to prove that a geometric type associated to a Markovian family $\mathcal{R}$ describes (as in the case of Markov partitions) completely the flow up to specific  surgeries. In order to do that, we first establish the existence of boundary periodic orbits (i.e. points in $\mathcal{P}$ corresponding to periodic orbits in $M$ and that never intersect the interior of some rectangle of $\mathcal{R}$) for Markovian families in Proposition \ref{p.boundaryperiodicpoints}. As in the case of  Markov partitions, we show that those correspond to a finite (and always non-zero) number of periodic orbits of the flow.  

Then we engage the proof of the main result of this paper: 

\vspace{0.22cm}
\textbf{Theorem B}
\textit{Let $\Phi_1,\Phi_2$ be two transitive Anosov flows in dimension 3, $\mathcal{R}_1,\mathcal{R}_2$ two Markovian families in their bifoliated planes $\mathcal{P}_1,\mathcal{P}_2$ and $\Gamma_1,\Gamma_2$ their boundary periodic orbits. If $\mathcal{R}_1$ and $\mathcal{R}_2$ are associated to the same class of geometric types, then $\Phi_1$ (up to orbital equivalence) can be obtained from $\Phi_2$ by performing a finite number of Dehn-Goodman-Fried surgeries on $\Gamma_2$. }
\vspace{0.22cm}

The proof of the above result is rather technical. During the proof of above theorem we introduce an interesting object generalizing the bifoliated plane: given a finite number of periodic orbits of $\Phi$, say $\Gamma$, and their lifts on $\mathcal{P}$, say $\widetilde{\Gamma}$, we define the \emph{bifoliated plane of $\Phi$ up to surgeries on $\Gamma$} as  
 the universal cover of $\mathcal{P}-\widetilde{\Gamma}$ together with some points at infinity. We then prove the following generalization of Barbot's theorem (Theorem 3.4 of \cite{Ba1}): 

\vspace{0.22cm}
\textbf{Theorem C}
\textit{The bifoliated plane $\clos{\mathcal{P}}$ of $\Phi$ up to surgeries on $\Gamma$ can be endowed with two transverse singular foliations and an action by $\pi_1(M-\Gamma)$. Together with this action and those foliations, $\clos{\mathcal{P}}$ describes completely the flow $\Phi$ up to orbital equivalence and up to Dehn-Goodman-Fried surgeries on $\Gamma$.}
\vspace{0.22cm}

See Theorem \ref{t.generalbarbot} for more details. 

Finally, in Theorem B we obtain a combinatorial object describing the flow up to orbital equivalence and up to some surgeries. At the end of this article, by adding (in a canonical way) some supplemental  combinatorial data to a geometric type, we define the notion of \emph{geometric type with cycles} and we show the following refined versions of Theorems A and B:

\textbf{Theorem D}\textit{
Let $M$ be an orientable, closed and connected 3-manifold and $\Phi$ a transitive
Anosov flow on $M$. To any Markovian family $\mathcal{R}$ of $\Phi$ we can associate canonically a unique and finite class of equivalent geometric types with cycles.}  
\vspace{0.22cm}

\textbf{Theorem E}\textit{
Let $\Phi_1,\Phi_2$ be two transitive  Anosov flows, $\mathcal{R}_1,\mathcal{R}_2$ two Markovian families in $\mathcal{P}_1,\mathcal{P}_2$ associated to the same equivalence class of geometric types with cycles. Then $\Phi_1$ is orbitally equivalent to $\Phi_2$.}

\subsection{Structure of the paper} 
In Section \ref{s.preliminaries}, we will remind some basic definitions and properties of Markov partitions, geometric types and Dehn-Goodman-Fried surgeries on Anosov flows.

In Section \ref{s.markovianfamilies}, we will  introduce the notion of Markovian families of rectangles in the bifoliated plane of an Anosov flow. We will then analyse and compare the properties of Markovian families with those of Markov partitions, which will lead us to a proof of Theorem A. 

In Section \ref{s.boundarypoints}, given a Markovian family $\mathcal{R}$, we show that there are finitely many points in the bifoliated plane up to the action of the fundamental group that are not contained in the interior of any rectangle in $\mathcal{R}$. The previous points will play a crucial role in the proof of Theorem B, since they constitute the only points of the bifoliated plane, whose small  neighbourhoods are not described by a unique, but multiple rectangles in $\mathcal{R}$. 

In Section \ref{s.rectpaths}, given a Markovian family $\mathcal{R}$ associated to a family of geometric types $G$, we describe a way for defining ``coordinates" in the bifoliated plane using rectangles in $\mathcal{R}$. Those coordinates will be called rectangle paths and will allow us to navigate  simultaneously in the bifoliated planes of all Anosov flows associated to the same family of geometric types $G$. 

In view of Theorem B, in Section \ref{s.constructionPbar} we will develop tools that will allow us to compare Anosov flows up to surgeries on a finite number of periodic points. We will thus define the bifoliated plane up to surgeries, compare its properties with those of the regular bifoliated plane and finally prove Theorem C. 

In Section \ref{s.homotopiesofpaths}, given two Anosov flows associated to the same family of geometric types, we prove the following technical result: rectangle paths in the bifoliated planes up to sugeries of the two flows define compatible systems of coordinates. This will allow us to extend those coordinate systems to an equivalence between the bifoliated planes up to surgeries in Section \ref{s.mainresult} and deduce by Theorem C a proof of Theorem B. 

Finally, in Section \ref{s.theoremsDE} by adding some combinatorial data in the geometric type, we will define a geometric type with cycles and prove Theorems D and E.

\vspace{0.25cm}
\textit{Acknowledgements}. We would like to address a special thanks to Christian Bonatti for his interest and useful intuition in this work, and to Fran\c{c}ois B\'eguin for organizing the study group on Anosov flows that allowed us to present early versions of our results. 

\section{Preliminaries}\label{s.preliminaries}
Throughout this paper we will assume that the reader has basic familiarity with Anosov flows and the bifoliated plane of an Anosov flow. For an introduction to Anosov flows see \cite{Anosov} and for an introduction to the bifoliated plane of an Anosov flow see \cite{Barbotthese},\cite{Fe} and  \cite{Fe1}. In this small section, we will remind some elementary definitions and results concerning Markov partitions, geometric types and Dehn-Goodman-Fried surgeries. 
\subsection{Markov partitions and geometric types} 
Let $M$ be an orientable, closed, connected 3-manifold carrying a transitive Anosov flow $\Phi$. Let $\mathcal{F}^s$ and $\mathcal{F}^u$ be the stable and unstable foliations of $\Phi$. 
\begin{defi}
Consider $[0,1]^2$ endowed with the vertical and horizontal foliations, the trivially bifoliated rectangle. A \emph{rectangle} $R$ in $M$ is the image of a trivially bifoliated rectangle by a continuous embedding $\phi: [0,1]^2 \rightarrow M$ sending every  vertical segment to an unstable segment (i.e. a segment contained in a weak unstable leaf) and every  horizontal segment to a stable segment.   

We will call $\partial^{u}R:= \phi(\lbrace0,1\rbrace \times [0,1])$ the \emph{unstable boundary} of $R$ and $\partial^{s}R:=\phi([0,1] \times \lbrace0,1\rbrace )$ its \emph{stable boundary}. 

Furthermore, any rectangle of the form $\phi([s,t]\times [0,1])$, where $s,t\in [0,1]$ and $s<t$, will be called a \emph{vertical subrectangle} of $R$. Similarly, rectangles of the form $\phi([0,1]\times [s,t])$ will be called \emph{horizontal subrectangles} of $R$. 
\end{defi}

\begin{defi}\label{d.markovpartition}
A \emph{Markov partition} of $\Phi$ is a finite family of rectangles $R_1,...,R_n$ in $M$ transverse to $\Phi$ such that: 
\begin{enumerate}
\item The rectangles are pairwise disjoint
\item Pushing positively by the flow, the first return on $\underset{i=1}{\overset{n}{\cup}}R_i$ of any point $x\in  \underset{i=1}{\overset{n}{\cup}}R_i$ is well defined and will be denoted by $f(x)$. Furthermore, there exists $T>0$ such that for all $x\in M$ there exists $t\in [0,T]$ for which $\Phi^t(x) \in \underset{i=1}{\overset{n}{\cup}}R_i$. 
\item For any two $i,j$ the closure of each connected component of  $f(\text{Int}(R_i))\cap \text{Int}(R_j)$  (the previous set can be empty) is a vertical subrectangle of $R_j$. 
\item For any two $i,j$ the closure of each connected component of  $f^{-1}(\text{Int}(R_i))\cap \text{Int}(R_j)$ (the previous set can be empty) is a horizontal subrectangle of $R_j$. 
\end{enumerate}
Furthermore, we will call the family $R_1,...,R_n$ a \emph{reduced} Markov partition if for every $i\neq j$, there doesn't exist a continuous function $\tau: R_i \rightarrow \mathbb{R}$ such that $\Phi^{\tau}(R_i)\subseteq R_j$ or $\Phi^{\tau}(R_i)\supseteq R_j$. 
\end{defi}
\begin{rema}
\begin{itemize}
    \item In the above definition, $(2)$ implies that the number of connected components of $f(\text{Int}(R_i))\cap \text{Int}(R_j)$  or $f^{-1}(\text{Int}(R_i))\cap \text{Int}(R_j)$ is finite for any $i,j$.
    \item The reason why we consider in (3) the set $f(\text{Int}(R_i))\cap \text{Int}(R_j)$ instead of $f(R_i)\cap R_j$ is because the latter one can in general be a segment.
\end{itemize} 
\end{rema}
The following two results ensure the existence of Markov partitions and reduced Markov partitions for all transitive Anosov flows.  
\begin{theorem}[Ratner, \cite{Ratner}] \label{t.existenceofmarkovpartitions}
For any periodic orbit $\gamma$ of $\Phi$, there exists a Markov partition of $\Phi$ formed by rectangles, whose stable and unstable boundaries are contained respectively in $\mathcal{F}^s(\gamma)$ and $\mathcal{F}^u(\gamma)$.
\end{theorem}

\begin{theorem}
For any periodic orbit $\gamma$ of $\Phi$, there exists a reduced Markov partition of $\Phi$ formed by rectangles, whose stable and unstable boundaries are contained respectively in $\mathcal{F}^s(\gamma)$ and $\mathcal{F}^u(\gamma)$.
\end{theorem}
\begin{proof}
Take $\gamma$ a periodic orbit of $\Phi$ and $R_1,...R_n$ a Markov partition in $M$, given by Ratner's theorem, formed by rectangles whose stable and unstable boundaries are contained respectively in $\mathcal{F}^s(\gamma)$ and $\mathcal{F}^u(\gamma)$. Suppose now, without any loss of generality, that there exists a continuous function $\tau: R_2 \rightarrow \mathbb{R}$ such that $R_1\subseteq \Phi^{\tau}(R_2)$. It suffices to prove that the family of rectangles $R_2,...,R_n$ is also a Markov partition of $\Phi$. 

Indeed, the rectangles $R_2,...,R_n$ are pairwise disjoint by construction. Let us now show that for any point $x\in  \underset{i=2}{\overset{n}{\cup}}R_i$ its first return time on $\underset{i=2}{\overset{n}{\cup}}R_i$ is well defined and uniformly bounded. 

Suppose that there exists $x\in  \underset{i=2}{\overset{n}{\cup}}R_i$, whose positive orbit doesn't intersect $\underset{i=2}{\overset{n}{\cup}}R_i$. Since $R_1,...R_n$ is a Markov partition of $\Phi$, this implies that the positive orbit of $x$ only intersects $R_1$. But since $\tau$ is continuous, there exists $m>0$ such that $|\tau|<m$. Take $t>m$ such that $\Phi^t(x)\in R_1$ (such a $t$ exists, since $R_1,...R_n$ is a Markov partition). By definition of $m$, there exists $s\in [t-m,t+m]\subset \mathbb{R}^+$ such that $\Phi^s(x)\in R_2$, which contradicts the fact that the positive orbit of $x$ doesn't intersect $\underset{i=2}{\overset{n}{\cup}}R_i$. Therefore, the first return time is defined for all points in $\underset{i=2}{\overset{n}{\cup}}R_i$.

Since $R_1,...R_n$ is a Markov partition, there exists $T>0$ such that for any $x\in M$ there exists $t\in [0,T]$ such that $\Phi^t(x)\in \underset{i=1}{\overset{n}{\cup}}R_i$. The same thing is true for $R_2,...R_n$. Indeed, if the positive orbit of $x\in M $ visits first any of the rectangles in $R_2,...R_n$, then it does so in a time bounded by $T$. Similarly, if the positive orbit of $x\in M $ visits first the rectangle $R_1$ $k$ times and then any rectangle of the rectangles in $R_2,...R_n$, then it does so in a time bounded by $(k+1)T$. 

On the other hand, since the rectangles $R_1,...R_n$ are pairwise disjoint, there exists $\epsilon>0$ such that for any $x\in M$ if $\Phi^{[0,t]}(x)$ intersects $N$ times $\underset{i=1}{\overset{n}{\cup}}R_i$, then $t>(N-1)\epsilon$. Take $N\in \mathbb{N}$ such that $N \epsilon>m$. If the positive orbit of a point $x\in M$ visits more than $N+1$ times $R_1$ before visiting any of the rectangles in $R_2,...R_n$, then it does so in a time bigger than $N\epsilon>m$. By our arguments in the previous paragraph this implies that there exists $s\in [N\epsilon-m,N\epsilon+m]$ such that $\Phi^s(x)\in R_2$, which leads to a contradiction. Therefore, the positive orbit of any point in $M$ intersects $R_2,...R_n$ in a time bounded by $\max(T,2T,...,(N+1)T, N\epsilon+m)=\max((N+1)T,N\epsilon+m)$. 

Finally, let us show that $R_2,...R_n$ satisfies the third axiom of the Definition \ref{d.markovpartition}. The fact that $R_2,...R_n$ satisfies the fourth axiom of the Definition \ref{d.markovpartition} can be proved in the exact same way. 

Let us denote by $f: \underset{i=1}{\overset{n}{\cup}}R_i \rightarrow \underset{i=1}{\overset{n}{\cup}}R_i$ and $f^{red}: \underset{i=2}{\overset{n}{\cup}}R_i \rightarrow \underset{i=2}{\overset{n}{\cup}}R_i$ the first return map on $\underset{i=1}{\overset{n}{\cup}}R_i$ and $\underset{i=2}{\overset{n}{\cup}}R_i$ respectfully. Fix $i\in \llbracket 2, n \rrbracket$. We would like to show that for any $j\in \llbracket 2, n \rrbracket$, the closure of every connected component of $f^{red}(\text{Int}(R_i))\cap \text{Int}(R_j)$ is a vertical subrectangle of $R_j$. In order to do so, we are going to break $R_i$ into smaller rectangles and we are going to show the desired result for each of those smaller rectangles.

Consider $H_1,...,H_l$ the closures of the connected components of $\underset{k=1}{\overset{n}{\cup}}\text{Int}(R_i)\cap f^{-1}(\text{Int}(R_k)) $ (see Figure \ref{f.reduced}). Notice that since $R_1,...,R_n$ is a Markov partition, all the above are components are horizontal subrectangles of $R_i$ and that the closure of the image by $f$ of any $\text{Int}(H_s)$ is a vertical subrectangle of some $R_p$, where $p\in \llbracket 1,n \rrbracket$. It suffices to show that for every $s$ and $j$ the closure of every connected component of $f^{red}(\text{Int}(H_s))\cap \text{Int}(R_j)$ is a vertical subrectangle of $R_j$. 
\begin{figure}[h!]
\includegraphics[scale=0.35]{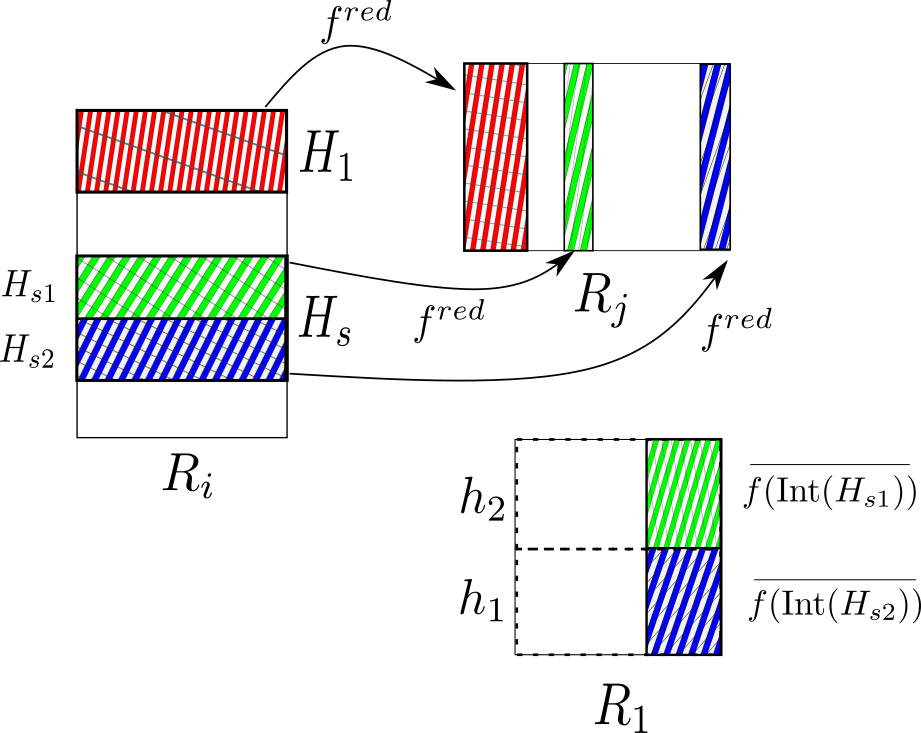}
\caption{}
\label{f.reduced}
\end{figure}

Fix an $s\in \llbracket 1,l \rrbracket$. If the closure of $f(\text{Int}(H_s))$ is a vertical subrectangle of $R_p$ with $p\in \llbracket 2,n \rrbracket$ (this is the case for the rectangle $H_1$ in Figure \ref{f.reduced}), then $f(\text{Int}(H_s))=f^{red}(\text{Int}(H_s))$, which gives us the desired result.

Suppose now that the closure of $f(\text{Int}(H_s))$ is a vertical subrectangle of $R_1$ (see Figure \ref{f.reduced})). In this case, $f^{red}(\text{Int}(H_s)) \neq f(\text{Int}(H_s))$. Recall that we would like to show that for every $j$ the closure of every connected component of $f^{red}(\text{Int}(H_s))\cap \text{Int}(R_j)$ is a vertical subrectangle of $R_j$. In order to do so, once again, we are going to divide $H_s$ into smaller rectangles and we are going to show the desired result for those rectangles.

Let $h_1,...,h_m$ be the closures of the connected components of $\underset{k=1}{\overset{n}{\cup}}\text{Int}(R_1)\cap f^{-1}(\text{Int}(R_k)) $. Notice that the $h_i$ have disjoint interiors, that they are all horizontal subrectangles of $R_1$ and that the closure of the image by $f$ of any $\text{Int}(h_i)$ is a vertical subrectangle of some $R_p$ with $p\in \llbracket 1,n \rrbracket$.

We will denote by $H_{s1},...,H_{sm}$ the rectangles $\overline{H_s\cap f^{-1}(\text{Int}(h_1))},...,\overline{H_s\cap f^{-1}(\text{Int}(h_m))}$. Each of the above rectangles is a horizontal subrectangle of $H_s$ and also for every $k$ the closure of $f^2(\text{Int}(H_{sk}))$ is a vertical subrectangle of some $R_p$ with $p\in \llbracket 1,n \rrbracket$. It therefore suffices to prove that for any $k\in \llbracket1, m \rrbracket$ and every $j\in \llbracket 2, n \rrbracket$, the closure every connected component of $f^{red}(\text{Int}(H_{sk}))\cap \int{R_j}$ is a vertical subrectangle of $R_j$.  

Fix $k$. As in the previous case, if $\overline{f^2(\text{Int}(H_{sk}))}$ is a vertical subrectangle of any of the rectangles $R_2,...,R_n$, then $\overline{f^{red}(\text{Int}(H_{sk}))}=\overline{f^2(\text{Int}(H_{sk}))}$ and we get the desired result. If not, following our above arguments we will need to partition once again $H_{sk}$ into a finite number of horizontal subrectangles, $\overline{H_{sk}\cap f^{-2}(\text{Int}(h_1))},...,\overline{H_{sk}\cap f^{-2}(\text{Int}(h_m))}$, and show the desired result for each of those subrectangles.

Since every positive orbit of any point in $\underset{i=2}{\overset{n}{\cup}}R_i$ meets in (uniformly) bounded time $\underset{i=2}{\overset{n}{\cup}}R_i$, by a finite number of repetitions of the previous argument, we obtain the desired result. 
\end{proof}
To every Markov partition $R_1,...R_n$ we can associate (this association will be explained in a more general and detailed way in the proof of Theorem \ref{t.associatemarkovfamiliestogeometrictype}) a family of combinatorial objects describing abstractly the way that the first return map acts on $R_1,...,R_n$.
\begin{defi}\label{d.geometrictype}
Take $n\in \mathbb{N}^*$ -the reader may consider that this integer  corresponds to  a finite number of copies of $[0,1]^2$, say $R_1,...,R_n$, endowed with the horizontal and vertical (oriented) foliations-. For every $i\in \llbracket 1,n \rrbracket$ choose $h_i,v_i \in \mathbb{N}^*$ such that $$\sum_i h_i =\sum_i v_i$$ 

Consider now for every $i\in \llbracket 1,n \rrbracket$ two finite sets of the form $\{H_i^j, j\in \llbracket 1, h_i\rrbracket\}$ and $\{V_i^j, j\in \llbracket 1, v_i\rrbracket\} $ -the reader may think of the first (resp. second) set as a collection of $h_i$ (resp. $v_i$) mutually disjoint horizontal (resp. vertical) subrectangles of $R_i$ ordered from bottom to top (resp. from left to right)-. Consider also a bijection $\phi$ between  $$\mathcal{H}=\lbrace H^j_i| i\in\llbracket 1,n \rrbracket, j \in \llbracket 1,h_i \rrbracket  \rbrace  \text{ and } \mathcal{V}=\lbrace V^j_i| i\in\llbracket 1,n \rrbracket, j \in \llbracket 1,v_i \rrbracket  \rbrace$$ and $u$ a function from the set of rectangles $\mathcal{H}$ to $\lbrace -1,+1\rbrace$.

The data $(n,(h_i)_{i \in \llbracket 1,n \rrbracket}, (v_i)_{i\in \llbracket 1,n \rrbracket}, \mathcal{H}, \mathcal{V},\phi, u)$ will be called a \emph{geometric type}. 
\end{defi}
In the following pages, even though a geometric type can be thought as an abstract combinatorial object, we will often think of it as a set of rectangles, each containing a finite number of horizontal and vertical subrectangles on which the map $\phi$ acts. A geometric type is a ``finite information" version of a Markov partition, where the function $\phi$  plays the role of the first return map and the function $u$ indicates whether the first return map preserves or changes the orientation of the vertical foliation (see Figure \ref{f.examplegeometrictype}). 

\begin{figure}[h!]
\includegraphics[scale=0.75]{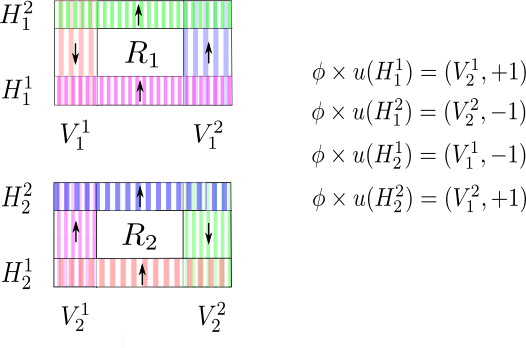}
\caption{In the above example the ``first return map" sends a horizontal rectangle to a vertical rectangle of the same color. The arrows represent how the ``first return map" acts on the orientation of the vertical foliation.}
\label{f.examplegeometrictype}
\end{figure}
 
\begin{defi}\label{d.equivalentgeomtypes}
Take $(n,(h_i)_{i \in \llbracket 1,n \rrbracket}, (v_i)_{i\in \llbracket 1,n \rrbracket}, \mathcal{H}, \mathcal{V},\phi,u)$ and $(n',(h'_i)_{i \in \llbracket 1,n' \rrbracket}, (v'_i)_{i\in \llbracket 1,n' \rrbracket}, \mathcal{H}', \mathcal{V}',\phi', u')$ two geometric types. Suppose that
\begin{align*}
    &\mathcal{H}=\{{H_i}^j, i\in \llbracket 1,n \rrbracket, j\in \llbracket 1, h_i \rrbracket\} \quad \text{and} \quad  \mathcal{V}=\{{V_i}^j, i\in \llbracket 1,n \rrbracket, j\in \llbracket 1, v_i \rrbracket\} \\ & \mathcal{H'}=\{{H'_i}^j, i\in \llbracket 1,n' \rrbracket, j\in \llbracket 1, h'_i \rrbracket\} \quad \text{and} \quad  \mathcal{V'}=\{{V'_i}^j, i\in \llbracket 1,n' \rrbracket, j\in \llbracket 1, v'_i \rrbracket\}
\end{align*} We will say that the two previous geometric types are \emph{equivalent} if 
\begin{enumerate}
    \item $n=n'$ 
    \item up to reindexing $h_i=h_i'$ and $v_i=v_i'$ 
    \item there exists a bijection $H: \mathcal{H}\cup \mathcal{V} \rightarrow \mathcal{H'}\cup \mathcal{V'}$ such that 
    \begin{itemize}
        \item for every $i\in \llbracket 1, n\rrbracket$, $H$ defines a bijection between $\{{H_i}^j, j\in \llbracket 1,h_i \rrbracket \}$ and $\{{H'_i}^{j'}, j'\in \llbracket 1,h'_i \rrbracket\}$ that is monotonous with respect to $j$. We define $\epsilon_i=+1$ if the previous map is increasing and $\epsilon_i=-1$ if not.
        \item for every $i\in \llbracket 1, n\rrbracket$, $H$ defines a bijection between $\{{V_i}^j, j\in \llbracket 1,v_i \rrbracket \}$ and $\{{V'_i}^{j'}, j'\in \llbracket 1,v'_i \rrbracket\}$ that is monotonous with respect to $j$. We define $\epsilon'_i=+1$ if the previous map is increasing and $\epsilon'_i=-1$ if not.
        \item either  $\epsilon_i\cdot \epsilon_i'=-1$ for all $i$ or  $\epsilon_i\cdot \epsilon_i'=+1$ for all $i$
        \item for every $h\in \mathcal{H}$ $\phi'(H(h))= H(\phi(h))$ 
    \end{itemize}
   \item for every $h\in \mathcal{H}$ we have $u'(H(h))=\epsilon_i \cdot \epsilon_j \cdot u(h)=\epsilon'_i \cdot \epsilon'_j \cdot u(h)$, where $i,j$ are such that $h$ and $\phi(h)$ are respectively of the form $H_i^{ \bullet}$ and $V_j^{ \bullet}$
    
\end{enumerate} 
If futhermore all the above $\epsilon_i$ and $\epsilon_i'$ are equal to $+1$, then we will say that the two geometric types are \emph{equal}.  
\end{defi}
Thinking of a geometric type as a collection of rectangles and subrectangles, an equivalence between two geometric types can be described as a homeomorphism between rectangles sending horizontal/vertical subrectangles to horizontal/vertical subrectangles and respecting the maps $\phi$ and $u$ -taking into account that $H$ might change the orientation of the verticals-. For instance, changing the orientation of the stable or unstable foliations in the example in Figure \ref{f.examplegeometrictype} yields an equivalent geometric type -recall that the $H_i^1,...,H_i^{h_i}$ and $V_i^1,...,V_i^{v_i}$ are respectively ordered from bottom to top and from left to right; therefore the previous operation reindexes the $H_i^j$ and leads to a different geometric type-. In fact, it is possible to show that given a geometric type $G=(n,(h_i)_{i \in \llbracket 1,n \rrbracket}, (v_i)_{i\in \llbracket 1,n \rrbracket}, \mathcal{H}, \mathcal{V},\phi, u)$, by reindexing the $(h_i,v_i)$ and changing the orientation of the foliations of the rectangles, we can obtain all geometric types equivalent to $G$. 

\begin{rema}\label{r.finitenumbergeometrictypes}
It is easy to see that the  relation defined in Definition \ref{d.equivalentgeomtypes} is an equivalence relation for which every equivalence class of geometric types is finite.
\end{rema}

\subsection{Dehn-Goodman-Fried surgery}\label{s.surgeries}

Let $X$ be a transitive Anosov flow on a oriented, closed, $3$-manifold $M$ and let $\gamma$ be a periodic orbit with positive eigenvalues. 

By blowing up $M$ along $\gamma$ we obtain a manifold $M_{\gamma}$ with one torus boundary endowed with a continuous map $\pi_\gamma\colon M_\gamma\to M$ of $M$. The map $\pi_{\gamma}$ induces a diffeomorphism from the interior of $M_{\gamma}$ to $M\setminus\gamma$. 

Furthermore, for every $x\in\gamma$ the fiber $\pi_\gamma^{-1}(x)$ is a circle which is canonically identified with the unit normal bundle $N^1(x)$ of $\gamma$ in $M$ at the point $x$. More specifically, consider two small segments $\sigma_1,\sigma_2$
in $M_\gamma$ transverse to the boundary $\partial M_\gamma$ at $\sigma_1(0), \sigma_2(0)\in \partial M_\gamma$.  Then we have  $\sigma_1(0)=\sigma_2(0)$ if and only if
\begin{itemize}
\item $\pi_{\gamma}(\sigma_1(0))=\pi_{\gamma}(\sigma_2(0))=c$ and 
\item there exists $\lambda\in \mathbb{R}$ such that  $$\frac{\partial (\pi_\gamma\circ \sigma_1)}{\partial t}(0)= \lambda X(c)+ \frac{\partial (\pi_\gamma\circ \sigma_2)}{\partial t}(0)$$ 
\end{itemize}

The vector field $\pi_\gamma^{-1}(X)$ is  well defined  on the interior of $M_\gamma$ and extends by continuity on its boundary $T_\gamma$ by the natural action of the derivative $DX^t$ on the normal bundle over $\gamma$. We denote by $X_\gamma$ this (smooth) vector field on $M_\gamma$.

The flow on $T_\gamma$ is a Morse-Smale flow with $4$ periodic orbits, which correspond to the normal vectors to $\gamma$ tangent to the stable and unstable manifolds of $\gamma$.  These $4$ periodic orbits are freely homotopic one to 	another and are non trivial in $\pi_1(T_\gamma)$.  The homotopy (or homology) class $p\in\mathbb{Z}^2=\pi_1(T_\gamma)$ of these periodic orbits is called \emph{the parallel}. 
 
On the other hand the fibers of $\pi_\gamma\colon T_\gamma\to \gamma$ inherit an orientation from the orientation of $M$ and the corresponding homotopy class $m\in\mathbb{Z}^2=\pi_1(T_\gamma)$ is called \emph{the meridian}. 

Given any integer $n\in\mathbb{Z}$, one easily checks the existence of foliations $\mathcal{G}_n$  on $T_\gamma$, transverse to the flow $X_\gamma$, whose leaves are simple closed curves of homotopy class $m+np$. By reparametrizing  the flow $X_\gamma$, one gets a new smooth vector field  $Y_\gamma$ on  $M_\gamma$  that leaves invariant the foliation $\mathcal{G}_n$.

Let $M_{\gamma,n}$ be the manifold obtained from $M_\gamma$ by collapsing the leaves of $\mathcal{G}_n$. The flow $Y_\gamma$ passes to the quotient and becomes a topological Anosov flow $X_{\gamma,n}$ on $M_{\gamma,n}$. 

Shannon in \cite{Mariothese} proves that $X_{\gamma,n}$ is orbitally equivalent (by a homeomorphism isotopic to identity) to an Anosov flow and that its orbit equivalence class depends only on $n$ and not our choice of foliation $\mathcal{G}_n$. We call $X_{\gamma,n}$ \emph{the Anosov flow obtained from $X$ by a (Dehn-Goodman-Fried) surgery along $\gamma$ with characteristic number $n$}. 

\subsubsection{Dehn-Goodman-Fried surgeries on orbits with negative eigenvalues} \label{s.vpnegative}
On an orientated $3$-manifold, it is also possible to perform Dehn-Goodman-Fried surgeries on  periodic orbits $\gamma$ with  negative eigenvalues. Once again, the intersection of the weak stable (or unstable) manifold of $\gamma$ with $T_\gamma$ defines a homotopy (or homology) class $P\in \mathbb{Z}^2= \pi_1(T_\gamma)$ and the fibers of $\pi_\gamma$ define a homotopy (or homology) class $m\in \mathbb{Z}^2= \pi_1(T_\gamma)$. In this case however, $P$ and $m$ intersect twice and therefore they do not form a basis of $\pi_1(T_\gamma)$. In order to complete $m$ into a basis of $\pi_1(T_\gamma)$ we consider the element $p:=\frac{1}{2}(P+m)\in \pi_1(T_\gamma)$. The homotopy classes $p$ and $m$ define respectively a canonical parallel and a canonical meridian on $T_{\gamma}$. 

Following our construction in the case of positive eigenvalues, we now need to define a new meridian $m_{new}$ along which we are going to collapse the torus $T_\gamma$. This new meridian must correspond to an indivisible element of $\pi_1(T_\gamma)$ and must verify $i(P,m)=i(P,m_{new})=\pm 2$, where $i$ is the intersection number of two elements in $\pi_1(T_\gamma)$. The last property assures that $\gamma$ (in the flow after surgery) remains a regular saddle and does not transform into a $k$-prong singularity. We can easily deduce from the two above properties that $m_{new}$ is of the form $$(2-2l)p+lm$$
where $l$ is an odd integer. By using the fact that $p:=\frac{1}{2}(P+m)$, we get that $$m_{new}= m+ (1-l)P$$
Finally, since $(1-l)$ is even we can write $m_{new}$ in the form $m+2nP$. As in the case of positive eigenvalues, by eventually reparametrizing $X_\gamma$, given any integer $n\in\mathbb{Z}$, there exists a foliation $\mathcal{G}_n$  on $T_\gamma$, transverse and invariant by the flow $X_\gamma$, whose leaves are simple closed curves of homotopy class $m+2nP$. 

By collapsing $T_\gamma$ along the leaves of $\mathcal{G}_n$, the flow $X_\gamma$ passes to the quotient and we thus obtain a (transitive) topological Anosov flow on a closed 3-manifold, $M_{\gamma,n}$. Once again by \cite{Mariothese} the previous flow is orbitally equivalent (by a homeomorphism isotopic to identity) to a (transitive) Anosov flow $X_{\gamma,n}$ called \emph{the Anosov flow obtained from $X$ by a (Dehn-Goodman-Fried) surgery along $\gamma$ with characteristic number $n$}.

\section{Markovian families of rectangles}\label{s.markovianfamilies}
In this section, we will define the notion of Markovian  families as a generalization of the notion of Markov partition in the bifoliated plane of an Anosov flow. In Section \ref{ss.existencemarkovfamilies}, we will prove that the projection on the bifoliated plane of the lift on $\mathbb{R}^3$ of a reduced Markov partition of the flow is a Markovian  family and therefore there exist infinitely many Markovian families associated to a transitive Anosov flow. In Section \ref{markovianfamilytogeometrictype} we will prove Theorem A, using the fact that Markovian families and Markov partitions share an abundance of properties.

Let $M$ be an orientable, closed,  3-manifold carrying a transitive Anosov flow $\Phi$. We will denote by $F^s,F^u$ the (weak) stable and unstable manifolds of $\Phi$. Let also $(\mathcal{P}, \mathcal{F}^s, \mathcal{F}^u)$ be the bifoliated plane of $\Phi$ endowed with an orientation.

\subsection{Reduced Markov partitions correspond to Markovian  families}\label{ss.existencemarkovfamilies}
\begin{defi}\label{d.rectanglesinplane}
Consider $[0,1]^2$ endowed with the vertical and horizontal foliations, the trivially bifoliated rectangle. A \emph{rectangle} $R$ in $\mathcal{P}$ is the image of a trivially bifoliated rectangle by a continuous embedding $\phi: [0,1]^2\rightarrow \mathcal{P}$ sending vertical segments to unstable segments and horizontal segments to stable segments. 

We will call $\partial^sR:= \phi([0,1]\times\lbrace0,1\rbrace  )$ the stable boundary of $R$ and $\partial^uR:=\phi(  \lbrace0,1\rbrace \times[0,1])$ its unstable boundary. 

Furthermore, any rectangle $R'$ of the form $\phi([s,t]\times [0,1])$ where $s,t\in (0,1]$ and $s<t$ will be called a vertical subrectangle of $R$. If furthermore $R' \neq R$, then we will call $R'$ a non-trivial vertical subretangle of $R$. We define similarly horizontal subrectangles and non-trivial horizontal subrectangles.

\end{defi}

\begin{defi}\label{d.markovfamily}
A \emph{Markovian family} (of rectangles) in $\mathcal{P}$ is a set of mutually distinct rectangles $(R_i)_{i \in I}$ covering $\mathcal{P}$ (i.e. $\cup_{i \in I} R_i = \mathcal{P}$) such that 
\begin{enumerate}
\item (Finiteness axiom) $(R_i)_{i \in I}$ is the union of a finite number of orbits of rectangles of the action by $\pi_1(M)$ (i.e. there exists $i_1,...,i_n \in I$ such that $\pi_1(M).(\overset{n}{\underset{k=1}{\cup}} R_{i_k}) = (R_i)_{i \in I}$) 
\item (Markovian intersection axiom) For every two rectangles $R_i,R_j$ in $(R_i)_{i \in I}$, if $\overset{\circ}{R_i } \cap \overset{\circ}{R_j} \neq \emptyset$, then $R_i \cap R_j$ is a non-trivial horizontal subrectangle of $R_i$ (or $R_j$ resp.) and a non-trivial vertical subrectangle of $R_j$ (or $R_i$ resp.) 
\item (Finite return time axiom) Take any point $x\in \mathcal{P}$ and any of its four quadrants defined by $\mathcal{F}^s(x)$ and $\mathcal{F}^u(x)$. For any sufficiently small neighbourhood $G$ of $x$ inside this quadrant, there exists $R\in (R_i)_{i \in I}$ such that $G\subset R$
\end{enumerate} 
\end{defi}

\begin{prop}\label{p.projectionmarkovpartition}
Take $R_1,R_2,...,R_n$ a reduced Markov partition of $\Phi$. Consider $(\widetilde{R_i})_{i\in I}$ the lift of the $R_i$ on $\widetilde{M}=\mathbb{R}^3$ and $(\overline{R_i})$ the projection of the $\widetilde{R_i}$ on $\mathcal{P}$ for all $i \in I$. The set $(\overline{R_i})_{i\in I}$ is a Markovian family of rectangles. 
\end{prop}
\begin{proof}
We will check the axioms of the previous definition one by one.

\vspace{0.2cm}
\textit{The $\overline{R_i}$ are rectangles in $\mathcal{P}$}

Suppose that an orbit of $\widetilde{\Phi}$ intersects twice $\widetilde{R_j}$ along the points $x,y$ (see Figure \ref{f.orbitrect}a). Without any loss of generality, we will assume that $x$ and $y$ do not belong both to the same unstable segment of $\widetilde{R_j}$. Let $\gamma$ be an arc in $\widetilde{R_j}$ transverse to $\widetilde{F^u}$ and going from $y$ to $x$ (see Figure \ref{f.orbitrect}a).  The points $x,y$ belong to the same weak unstable leaf $\mathcal{L}\in \widetilde{F^u}$. Let us denote by $l$ the loop obtained by concatenating the orbit segment (by $\widetilde{\Phi}$) going from $x$ to $y$ and the arc $\gamma$. By a small perturbation (see Figure \ref{f.orbitrect}b) we can assume that $l$ in transverse to $\widetilde{F^u}$, which is impossible because a foliation by planes in a simply connected manifold can not admit a closed trasnversal (see for instance Corollary 1 of \cite{Palmeira}). Therefore, 
since all the rectangles $\widetilde{R_j}$ are transverse to the lifted flow $\widetilde{\Phi}$ and compact, their projection on the bifoliated plane defines a homeomorphism. We deduce that the $\overline{R_i}$ are rectangles in $\mathcal{P}$. 
\begin{figure}[h]
 
  \begin{minipage}[ht]{0.4\textwidth}
  \centering
    \includegraphics[width=0.8\textwidth]{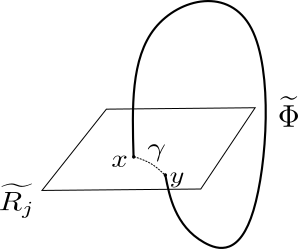}
    \caption*{(a)}
    
  \end{minipage}
  \begin{minipage}[ht]{0.4\textwidth}
  \centering
    \includegraphics[width=0.8\textwidth]{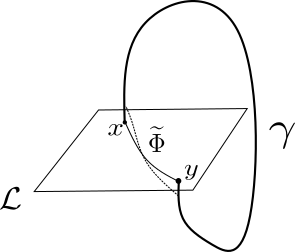}
    \caption*{(b)}
    
  \end{minipage}
  \caption{An orbit of $\widetilde{\Phi}$ can not intersect twice a rectangle}
  \label{f.orbitrect}
\end{figure}

\vspace{0.2cm}
\textit{Finiteness axiom} 

The  rectangles  $(\overline{R_i})_{i\in I}$ are mutually distinct and verify the finiteness axiom, since they correspond to projections on $\mathcal{P}$ of the lifts on $\widetilde{M}$ of a finite number of rectangles in $M$ that are not equivalent by $\Phi$ (i.e. for every $i,j\in \llbracket 1, n \rrbracket$, if $i\neq j $, then there doesn't exist a continuous function $\tau: R_i \rightarrow \mathbb{R}$ such that $\Phi^{\tau}(R_i)=R_j$) . 

\textit{Markovian intersection axiom} 

Take two distinct $i,j \in I$ such that $\overset{\circ}{\overline{R_i} } \cap \overset{\circ}{\overline{R_j}} \neq \emptyset$. Consider the lifts $\widetilde{R_i}$ and $\widetilde{R_j}$ of  $\overset{\circ}{\overline{R_i} }$ and $ \overset{\circ}{\overline{R_j}}$ on $\widetilde{M}$. Denote by $A$ (resp. $B$) the set of points of $\widetilde{R_i}$, whose positive (resp. negative) orbits by $\widetilde{\Phi}$ intersect $\widetilde{R_j}$. Without any loss of generality let us assume that $A\neq \emptyset$. It suffices to show that $A$ is a non-trivial horizontal subrectangle of $\widetilde{R_i}$ and that $B=\emptyset$. 

Consider the map $f: A \rightarrow \widetilde{R_j}$ associating every point $x\in A$ to the first point of intersection of the positive orbit of $x$ by $\widetilde{\Phi}$ with $\widetilde{R_j}$. Since $A$ and $\widetilde{R_j}$ are compact, the positive orbits of all points in $A$ intersect $\widetilde{R_j}$ in a uniformly bounded time. Hence, for a sufficiently big $N>0$, every orbit of $A$ intersects at most $N$ rectangles in $(\widetilde{R_i})_{i\in I}$ before intersecting $\widetilde{R_j}$. By induction on $N$ we can show that $A$ can be written as the union of finitely many horizontal subrectangles $h_1,...,h_s$ of $\widetilde{R_i}$ such that :
\begin{itemize}
    \item the sets  $\text{Int}_{\widetilde{R_i}}(h_i)$ and $f(\text{Int}_{\widetilde{R_i}}(h_i))$ are two by two disjoint
    \item the closure of every $f(\text{Int}_{\widetilde{R_i}}(h_i))$ is a vertical subrectangle of $\widetilde{R_j}$
    \item $f$ acts continuously on every  $\text{Int}_{\widetilde{R_i}}(h_i)$
\end{itemize}
Now, suppose that $s>1$. By pushing by $f$ any unstable segment in $\widetilde{R_i}$ going from one stable boundary to the other, we obtain $s$ unstable segments $u_1,...,u_s$ of $\widetilde{R_j}$ going from one stable boundary to the other. The segments $u_1,...,u_s$ belong to the same weak unstable leaf $\mathcal{L}$ in $\widetilde{F^u}$. Therefore, for any $x\in \widetilde{R_j}$ the leaf $\widetilde{F^s}(x)$ intersects $\mathcal{L}$ inside $\widetilde{R_j}$ more than once. Since any two weak stable and unstable manifolds on $\widetilde{M}$ intersect along one orbit at most (see Proposition 1.4.10 in \cite{Barbotthese}) this implies that there exists one orbit of $\widetilde{\Phi}$ intersecting multiple times $\widetilde{R_j}$. We have already established that this is impossible in the part \textit{The $\overline{R_i}$ are rectangles in $\mathcal{P}$}. Therefore, $s=1$ and $A$ consists of a unique horizontal subrectangle of $\widetilde{R_i}$ on the interior of which $f$ acts continuously. Notice that since $R_1,...,R_n$ is reduced $A\neq \widetilde{R_i}$, therefore the previous subrectangle is not trivial and $\overline{R_i} \cap \overline{R_j}$ contains a non-trivial horizontal subrectangle of $\overline{R_i}$. 

 If $B\neq \emptyset$, by the same argument, $\overline{R_i} \cap \overline{R_j}$ would also contain a non-trivial vertical subrectangle of $\overline{R_i}$. This is impossible since $\overline{R_i}$, $\overline{R_j}$ are embedded in $\mathcal{P}$. We thus showed that $\overline{R_i} \cap \overline{R_j}$ is a non-trivial horizontal subrectangle of $\overline{R_i}$ and by a symmetric argument we also obtain that it is also a non-trivial vertical subrectangle of $\overline{R_j}$. 

\vspace{0.2cm}
\textit{Finite return time axiom} 

Suppose that there exists a point $x\in \mathcal{P}$ and a quadrant $Q$ of $x$ such that for any small neighbourhood $U$ of $x$ in $Q$ there is no rectangle in $(\overline{R_i})_{i\in I}$ containing $U$. Since the orbit by $\Phi$ of any point in $M$ intersects a rectangle in $R_1,...,R_n$ in a uniformly bounded time, the orbit by $\widetilde{\Phi}$ of any point in $U$ intersects a rectangle in $(\widetilde{R_i})_{i\in I}$ in a uniformly bounded time. Therefore, we can find a sequence $x_n \in U$ accumulating to $x$ such that the orbit of $x_n$ by $\widetilde{\Phi}$ intersects $\widetilde{R_n}$ in a uniformly bounded time. Furthermore, by our initial hypothesis, we can assume that the $\widetilde{R_n}$ are mutually distinct. But, such a family of rectangles cannot be contained in a compact set in $\mathbb{R}^3$, which contradicts the fact that the orbits of the $x_n$ intersect them in a uniformly bounded time.

\end{proof}

A useful fact proven above (see \textit{the $\overline{R_i}$ are rectangles}) and that will later be used in other occasions is that 
\begin{rema}\label{r.orbitsrectangles}
Let $\Phi$ be a transitive Anosov flow on $M$. Every orbit of the lifted flow $\widetilde{\Phi}$ on $\widetilde{M}=\mathbb{R}^3$ intersects at most once any rectangle in $\mathbb{R}^3$ transverse to $\widetilde{\Phi}$. More generally, by the same exact proof, we can show that any leaf in $\widetilde{F^s}$ or $\widetilde{F^u}$ intersects any rectangle in $\mathbb{R}^3$ transverse to $\widetilde{\Phi}$ at most along one segment. 
\end{rema}

Using Proposition \ref{p.projectionmarkovpartition} and Theorem \ref{t.existenceofmarkovpartitions}, we conclude that 

\begin{coro}
There exist infinitely many Markovian families in the bifoliated plane of every transitive Anosov flow. 
\end{coro}
 
\subsection{Associating a Markovian  family to a geometric type} \label{markovianfamilytogeometrictype}
Recall that we denote the bifoliated plane of $\Phi$ by $\mathcal{P}$ and its oriented stable and unstable line foliations by $\mathcal{F}^s$ and $\mathcal{F}^u$. We will fix from now on $\mathcal{R}$ a Markovian family of $\mathcal{P}$.

Our goal in this section is to prove Theorem A. In order to do that we will first need to prove a series of lemmas showing among others that a  general Markovian family has a lot of similarities with the Markovian families constructed in Section \ref{ss.existencemarkovfamilies}: the boundaries of all the rectangles belong to stable/unstable leaves with non-trivial stabilizer in $\pi_1(M)$, every point in $\mathcal{P}$ belongs to infinitely many rectangles of the family, etc. Because of the abundance of such similarities we conjecture that:  
\begin{conj}
Every Markovian family corresponds to the projection on $\mathcal{P}$ of the lift on $\mathbb{R}^3$ of a Markov partition of $\Phi$.
\end{conj}
All the results that we prove in this section are analogous to well known results on Markov partitions (see Remark \ref{r.comparison}, this is why the reader may choose to think of a Markovian family as a Markov partition inside the bifoliated plane of an Anosov flow.

In the following lines, we will call 
\begin{itemize}
    \item a point in $\mathcal{P}$ \emph{periodic} if its associated orbit in $M$ is periodic
    \item a stable or unstable leaf in $\mathcal{F}^{s,u}$ \emph{periodic} if it contains a periodic point
    \item the closure of a connected component of $\mathcal{F}^s(x)-\{x\}$ (resp. $\mathcal{F}^u(x)-\{x\}$) a stable (resp. unstable) \emph{separatrix} of $x\in \mathcal{P}$
    \item the closure of any connected component of $\mathcal{P}-(\mathcal{F}^s(x) \cup \mathcal{F}^u(x))$ a quadrant of $x\in \mathcal{P}$
\end{itemize}

Using the orientations on $\mathcal{F}^{s,u}$, we will denote by $\mathcal{F}^{s,u}_{+}(x)$ (resp. $\mathcal{F}^{s,u}_{-}(x)$) the positive (resp. negative) stable/unstable separatrix of $x\in \mathcal{P}$. Also, the quadrant of $x$ delimited by $\mathcal{F}^{s}_{\epsilon}(x)$ and $\mathcal{F}^{u}_{\epsilon'}(x)$, where $\epsilon,\epsilon'\in \{+,-\}$, will be refered as the $(\epsilon,\epsilon')$ quadrant of $x$. 

For the sake of simplicity, we will assume from now on and until explicitely said otherwise, that the action of $\pi_1(M)$ on $\mathcal{P}$ preserves the orientations of the foliations. In other words, we will assume that \underline{$\Phi$ has transversally oriented foliations}. In Remark \ref{ss.nontransorientcase} we will explain how to adapt the following proofs in the non-transversally orientable case.

\begin{lemm}\label{l.vertsubrectangleexists}
Take $R\in\mathcal{R}$, $x\in R$ and  $Q$ a quadrant of $x$. If $R$ intersects a germ of $x$ in $Q$ (i.e. there exists $\mathcal{G}$ a neighbourhood of $x$ inside $Q$ such that $\mathcal{G}\subset R$), then there exists $R_v\in \mathcal{R}$ such that $R_v$ also  intersects a germ of $x$ in $Q$ 
and $R_v\cap R$ is a non-trivial vertical subrectangle of $R$.

\end{lemm}
\begin{proof}
Without any loss of generality, let us assume that $Q$ is the $(+,+)$ quadrant of $x$.

Recall that the action of $\pi_1(M)$ on $\mathcal{P}$ and the action of $\pi_1(M)$ on $\mathbb{R}^3$ are equivariant for the projection $\pi: \mathbb{R}^3 \rightarrow \mathcal{P}$. In order to avoid any confusion, we will use the following notations for the action of $\pi_1(M)$ on $\mathcal{P}$ and $\mathbb{R}^3$ : \newline{}
\begin{minipage}[b]{0.5\linewidth}
\begin{align*}
&(\pi_1(M),\mathcal{P})\rightarrow \mathcal{P}\\
& \quad (g,x)\rightarrow g(x)
\end{align*}
\end{minipage}
\begin{minipage}[b]{0.5\linewidth}
\begin{align*}
&(\pi_1(M),\mathbb{R}^3)\rightarrow \mathbb{R}^3\\
& \quad (g,x)\rightarrow g.x
\end{align*}
\end{minipage}
Let us first prove the lemma assuming that $x$ is periodic. Take $g\in \text{Stab}(x)$ that acts as an expansion on $\mathcal{F}^u(x)$ and as a contraction on $\mathcal{F}^s(x)$. The rectangle $g(R)$ is in $\mathcal{R}$, contains $x$ and since $g$ preserves the quadrants of $x$, $g(R)$ also contains a germ of the $(+,+)$ quadrant of $x$. Hence, $\int{g(R)}\cap \int{R}\neq \emptyset$. By the Markovian  intersection axiom and the fact that the unstable boundaries of $g(R)$ are ``longer" than the unstable boundaries of $R$, we have that $g(R)\cap R$ is a non-trivial vertical subrectangle of $R$. We thus obtain the desired result. 

Assume now that $x$ is not periodic. By the finiteness axiom, $\mathcal{R}$ is the union of a finite number of orbits of rectangles by the action of $\pi_1(M)$. Take a representative of each orbit. Let us name those rectangles $R_1,...,R_n$. Lift each of those rectangles of $\mathcal{P}$ to a smooth rectangle $\widetilde{R_i}$ in $\mathbb{R}^3$ transverse to the lifted flow $\widetilde{\Phi}$. Using the equivariance for the projection of the action of $\pi_1(M)$ on $\mathcal{P}$ and the action of $\pi_1(M)$ on $\mathbb{R}^3$, we can define in a unique way, using the $\widetilde{R_i}$, a lift $\tilde{r}$ for every rectangle $r$ in $\mathcal{R}$ and thus a lift $\widetilde{\mathcal{R}}$ of $\mathcal{R}$. 

Take $\tilde{x}$ to be the lift of $x$ in $\tilde{R}$. It suffices to show that there exist infinitely many rectangles $\tilde{r}$ in $\widetilde{\mathcal{R}}$ such that the negative orbit of $\tilde{x}$ by $\widetilde{\Phi}$, intersects $\tilde{r}$ at $\tilde{r}(x)$ and also  the separatrices $ \widetilde{\mathcal{F}^u_+}(x):=\pi^{-1}(\mathcal{F}^u_+(x))$ and $ \widetilde{\mathcal{F}^s_+}(x):=\pi^{-1}(\mathcal{F}^s_+(x))$ intersect $\tilde{r}$ along non-trivial segments containing $\tilde{r}(x)$. We will name the existence of infinitely many such rectangles, property  $(\star)$.

\vspace{0.2cm}
\textit{Proof that $(\star)$ suffices} 

Assume that $(\star)$ holds. Since the action of $\pi_1(M)$ is properly discontinuous on $\mathbb{R}^3$, this implies that  for every $T>0$ there exists $t<-T$ such that $\widetilde{\Phi}^t(\tilde{x})$ intersects a rectangle $\widetilde{R(T)}$ in $\widetilde{\mathcal{R}}$ that has a non-trivial intersection with $\widetilde{\mathcal{F}^s_+}(x)$ and $\widetilde{\mathcal{F}^u_+}(x)$. 

Take $U$ a compact neighbourhood of $\widetilde{R}$ in $\mathbb{R}^3$ and let  $\widetilde{A(T)}$ be the set of points of $\widetilde{R}$, whose orbits by $\widetilde{\Phi}$ intersect $\widetilde{R(T)}$. By taking $T$ sufficiently big, since the action of $\pi_1(M)$ on $\mathbb{R}^3$ is properly discontinuous, we can assume that $\widetilde{R(T)}\cap U =\emptyset$. 

Furthermore, $\widetilde{A(T)}$ intersects the interior of $\widetilde{R}$, since the rectangles $\widetilde{R}$ and $\widetilde{R(T)}$ intersect $ \widetilde{\mathcal{F}^s_+}(x)$ and $ \widetilde{\mathcal{F}^u_+}(x)$ along non-trivial segments containing $\tilde{x}$ and $\widetilde{R(T)}(x)$ respectively. By the Markovian intersection axiom applied for the projections on $\mathcal{P}$ of  $\widetilde{R}$ and $\widetilde{R(T)}$, the set $\widetilde{A(T)}$ is a non-trivial horizontal or vertical subrectangle of $\widetilde{R}$.    Therefore, using Remark \ref{r.orbitsrectangles} and the connectedness of $\widetilde{A(T)}$, the negative orbit of every point of $\widetilde{A(T)}$ intersects $\widetilde{R(T)}$ and the positive orbit of $\widetilde{A(T)}$ does not intersect $\widetilde{R(T)}$. Hence, there exist $\phi_T: \widetilde{A(T)}\rightarrow \mathbb{R}^{-}$ and $M(T)>0$ such that:
\begin{itemize}
    \item for every $z\in \widetilde{A(T)}$, $\widetilde{\Phi}^{\phi_T(z)}(z)\in \widetilde{R(T)}$ 
    \item $\phi_T<-M(T)$
\end{itemize}

Take any riemannian metric of $M$ and lift it to a metric on $\mathbb{R}^3$. $\pi_1(M)$ acts on $\mathbb{R}^3$ by isometries for the lifted metric. Also, since $\widetilde{\mathcal{R}}$ is the union of a finite number of orbits of smooth rectangles by the action of $\pi_1(M)$, the lengths of the boundaries of all the rectangles in $\widetilde{\mathcal{R}}$ are uniformly bounded from above and below. By eventually taking a bigger neighbourhood $U$ around $\widetilde{R}$ and increasing $T$, since the action of $\pi_1(M)$ is properly discontinuous, we can assume that $M(T)$ is big and therefore that $\widetilde{\Phi}^{\phi_T}(\widetilde{A(T)})$ has a very long stable boundary (with respect to the stable boundary of $\widetilde{A(T)}$) and a very small unstable boundary (with respect to the unstable boundary of $\widetilde{A(T)}$). By the same argument as in the previous paragraph, $\widetilde{\Phi}^{\phi_T}(\widetilde{A(T)})$ is a non-trivial vertical or horizontal subrectangle of $\widetilde{R(T)}$. But since the lengths of the boundaries of all the rectangles in $\widetilde{\mathcal{R}}$ are uniformly bounded from above and below, for $T$ sufficiently big,  $\widetilde{\Phi}^{\phi_T}(\widetilde{A(T)})$ cannot be a vertical  subrectangle of $\widetilde{R(T)}$. 

Take $R(T)$ and $A(T)$ to be respectfully the projections of $\widetilde{R(T)}$ and $\widetilde{A(T)}$ on $\mathcal{P}$. We deduce from the  previous paragraphs that:
\begin{itemize}
    \item both $R$ and $R(T)$ contain $x$ and a neighbourhood of $x$ inside its $(+,+)$ quadrant
    \item $R\cap R(T)=A(T)$ is a non-trivial horizontal subrectangle of $R(T)$ ; thus a non-trivial vertical subrectangle of $R$
\end{itemize}  
which gives us the desired result.

\vspace{0.2cm}
\textit{Proof of $(\star)$}

By Remark \ref{r.orbitsrectangles}, if a rectangle $\tilde{r}$ in $\widetilde{\mathcal{R}}$ intersects the orbit of $\tilde{x}$ along $\tilde{r}(x)$ and also intersects $ \widetilde{\mathcal{F}^s_+}(x)$ non-trivially (i.e. $ \widetilde{\mathcal{F}^s_+}(x)\cap \tilde{r}\neq \{\tilde{r}(x)\}$), then $ \widetilde{\mathcal{F}^s_+}(x)\cap \tilde{r}$ is a stable segment of the rectangle $\tilde{r}$ containing $\tilde{r}(x)$. Let us prove $(\star)$ by contradiction.

Suppose that the negative orbit of $\tilde{x}$ by $\widetilde{\Phi}$ intersects finitely many rectangles in $\widetilde{\mathcal{R}}$ that intersect non-trivially $ \widetilde{\mathcal{F}^s_+}(x)$ and $ \widetilde{\mathcal{F}^u_+}(x)$.  This implies that there exists $T>0$ such that for all $t<-T$ $\widetilde{\Phi}^t(\tilde{x})$ does not intersect any rectangle $\tilde{r}$ in $\widetilde{\mathcal{R}}$ intersecting $ \widetilde{\mathcal{F}^s_+}(x)$ and $ \widetilde{\mathcal{F}^u_+}(x)$ non-trivially. 

Take $x_M$ to be the projection of $\tilde{x}$ in $M$, $y\in M$ a point in the $\alpha$-limit of $x_M$ and $S$ a small rectangle in $M$ transverse to $\Phi$  containing $y$ in its interior. There exists $(t_n)_{n\in \mathbb{N}}$, an increasing sequence in $\mathbb{R}^+$ going to infinity such that $\Phi^{-t_n}(x_M)\in S$ and $\Phi^{-t_n}(x_M) \underset{n\rightarrow +\infty}{\rightarrow}y$. Since the orbit of $x_M$ is by hypothesis non closed, we can assume that the $\Phi^{-t_n}(x_M)$ are two by two distinct. 

Let us lift everything on $\mathbb{R}^3$. Take $\tilde{y}$ to be a lift of $y$ on $\mathbb{R}^3$ and $\tilde{S}$ the lift of $S$ containing $\tilde{y}$. By Remark \ref{r.orbitsrectangles}, the lifts of $\Phi^{-t_n}(x_M)$ contained in $\tilde{S}$ belong to different orbits of $\widetilde{\Phi}$. Therefore, there exists a sequence of $g_n\in \pi_1(M)$ such that $\widetilde{\Phi}^{-t_n}(g_n.\tilde{x})\in \tilde{S}$ and $\widetilde{\Phi}^{-t_n}(g_n.\tilde{x}) \underset{n\rightarrow +\infty}{\rightarrow}\tilde{y}$. 

By projecting everything on $\mathcal{P}$, we have that there exists a sequence of $g_n\in \pi_1(M)$, such that $g_n(x)\underset{n\rightarrow +\infty}{\rightarrow}Y$, where $Y$ is the projection of $\tilde{y}$ on $\mathcal{P}$. By eventually considering a subsequence, we can assume that all the $g_n(x)$ are contained in the same quadrant of $Y$. Therefore, by the finite return time axiom and again by eventually considering a subsequence, there exists $r\in \mathcal{R}$ containing the $g_n(x)$ and intersecting non-trivially all $\mathcal{F}^s_+(g_n(x))$ and $\mathcal{F}^u_+(g_n(x))$. 

This implies that the orbits of the $\widetilde{\Phi}^{-t_n}(g_n.\tilde{x})\in \tilde{S}$ cross a rectangle $\tilde{r}\in \widetilde{\mathcal{R}}$ that intersects  $\widetilde{\mathcal{F}^s_+}(g_n.\tilde{x})$ and $\widetilde{\mathcal{F}^u_+}(g_n.\tilde{x})$ non-trivially. Furthermore, since $\tilde{r}$ and $\tilde{S}$ are bounded in $\mathbb{R}^3$, the orbits of the $\widetilde{\Phi}^{-t_n}(g_n.\tilde{x})\in \tilde{S}$ intersect $\tilde{r}$ in a uniformly bounded time. But since the $g_n$ preserve the flow $\widetilde{\Phi}$, the set of rectangles $\widetilde{R}$ and the orientation of the stable and unstable foliations, this would imply that for every $n$ in a uniformly bounded time  the point $\widetilde{\Phi}^{-t_n}(\tilde{x})$ will cross a rectangle in $\widetilde{\mathcal{R}}$ intersecting $\widetilde{\mathcal{F}^s_+}(\tilde{x})$ and $\widetilde{\mathcal{F}^u_+}(\tilde{x})$ non-trivially. By hypothesis, the $t_n$ go to infinity, which contradicts the fact that the negative orbit of $\tilde{x}$ intersects finitely many such rectangles.  
\end{proof}
Naturally, by symmetry we can also prove that:  
\begin{lemm}\label{l.horizsubrectangleexists}
Take $R\in\mathcal{R}$, $x\in R$ and  $Q$ a quadrant of $x$. If $R$ intersects a germ of $x$ in $Q$, then there exists $R_h\in \mathcal{R}$ such that $R_h$ also intersects a germ of $x$ in $Q$ 
and $R_h\cap R$ is a non-trivial horizontal subrectangle of $R$.
\end{lemm}
Let us point out that during the proof of the Lemma \ref{l.vertsubrectangleexists}, we showed that thanks to the third axiom in the definition of a Markovian family (see Definition \ref{d.markovfamily}), the negative orbit by $\widetilde{\Phi}$ of any point $\tilde{x}\in \mathbb{R}^3$ will intersect in finite time a rectangle in $\widetilde{\mathcal{R}}$. This is  the reason why we called this axiom the finite return time axiom. 
\begin{lemm}\label{l.periodicboundary}
The boundary of any $R\in \mathcal{R}$ consists of stable/unstable segments belonging to stable/unstable periodic leaves in $\mathcal{P}$.  
\end{lemm}
\begin{proof}
Indeed, take $R\in \mathcal{R}$ and consider one of its stable boundary segments, say $s$. Let us denote $S$ the stable leaf in $\mathcal{F}^s$ containing $s$.

Take $x\in s$. By Lemma \ref{l.horizsubrectangleexists}, there exists $R'\in \mathcal{R}$ containing $x$ such that $R' \cap R$ is a non-trivial horizontal subrectangle of $R$. Therefore, $R'$ contains $s$. By repeatedly applying  this argument, we can construct $R=r_0,r_1,...,r_n,...$ a sequence of rectangles containing $s$ and such that $r_n \cap r_{n+1}$ is a non-trivial horizontal subrectangle of $r_n$. 

By the finiteness axiom, there exist $i,j$ two distinct integers and $g\in \pi_1(M)$ such that $g(r_i)=r_j$. Since $g$ preserves the orientation of the stable/unstable foliations, this implies that $g(S)=S$. We deduce that $S$ is a periodic stable leaf and we get the desired result.  
\end{proof}
\begin{lemm}\label{l.infiniteintersectionverticalrectangles}
Consider a sequence of rectangles $(r_n)_{n\in \mathbb{N}}$ in $\mathcal{R}$ such that for every $k\in \mathbb{N}$, $r_{k+1}\cap r_k$ is a non-trivial vertical subrectangle of $r_k$. We have that $\overset{+\infty}{\underset{k=0}{\cap}} r_{k}$ is an unstable segment of $r_0$.
\end{lemm}
\begin{proof}
It suffices to show the lemma for any subsequence of $(r_n)_{n\in \mathbb{N}}$. Note that thanks to the Markovian intersection property, for any such subsequence $(r_{k(n)})_{n \in \mathbb{N}}$, $r_{k(n+1)}\cap r_{k(n)}$ remains a vertical subrectangle of $r_{k(n)}$. 

By the finiteness axiom and by eventually considering a subsequence, we can assume that all the rectangles in our sequence belong to the same orbit of rectangles in $\mathcal{P}$. In other words, for any $n,m \in \mathbb{N}$ there exists $g_{n,m} \in \pi_1(M)$ such that $g_{n,m}(r_n)=r_m$. It is not difficult to prove that $g_{n,m}$ is unique. 

Consider $\widetilde{r_0}$ a smooth disk in $\mathbb{R}^3$, transverse to the lifted flow $\widetilde{\Phi}$, whose projection on $\mathcal{P}$ is $r_0$. Once again, we will use the following notations for the action of $\pi_1(M)$ on $\mathcal{P}$ and $\mathbb{R}^3$ : \newline{}
\begin{minipage}[b]{0.5\linewidth}
\begin{align*}
&(\pi_1(M),\mathcal{P})\rightarrow \mathcal{P}\\
& \quad (g,x)\rightarrow g(x)
\end{align*}
\end{minipage}
\begin{minipage}[b]{0.5\linewidth}
\begin{align*}
&(\pi_1(M),\mathbb{R}^3)\rightarrow \mathbb{R}^3\\
& \quad (g,x)\rightarrow g.x
\end{align*}
\end{minipage}

By hypothesis, for every $n \in \mathbb{N}$ there are orbits of $\widetilde{\Phi}$ intersecting both $\widetilde{r_0}$ and $g_{0,n}.\widetilde{r_0}$. Let's name $A_n$ the intersection of all these orbits with $g_{0,n}.\widetilde{r_0}$. $A_n$ projects to $r_0 \cap r_n \subset \mathcal{P}$ and by Remark \ref{r.orbitsrectangles} this projection is a homeomorphism, therefore $A_n$ is a horizontal subrectangle of $g_{0,n}.\widetilde{r_0}$. 

Moreoever, there exists $t_n: A_n \rightarrow \mathbb{R}$ a function such that $\widetilde{\Phi^{t_n(x)}}(x) \in \widetilde{r_0}$ for every $x \in A_n$. Using Remark \ref{r.orbitsrectangles} we can show that $t_n$ is continuous. Furthermore, if $g_{0,n}.\widetilde{r_0} \cap \widetilde{r_0} = \emptyset$ then either $t_n(x)>0$ for all $x \in A_n$ or $t_n(x)<0$ for all $x \in A_n$. Since the action of $\pi_1(M)$ on $\mathbb{R}^3$ is properly discontinuous, by eventually removing some rectangles from the sequence $(r_n)_{n\in \mathbb{N}}$, we can assume that $g_{0,n}.\widetilde{r_0} \cap \widetilde{r_0} = \emptyset$ for all $n$

Let us show, that except maybe a finite number of $n$, we have $t_n>0$. Suppose the contrary. By choosing any Riemannian metric on $M$ and lifting it on  $\mathbb{R}^3$, we can assume that $\pi_1(M)$ acts by isometries on $\mathbb{R}^3$; thus all the rectangles $g_{0,n}.\widetilde{r_0}$ are isometric. Take $c\in \mathbb{R}^-$ and $K$ any compact neighbourhood of $\widetilde{r_0}$. We can find $n$ sufficiently big such that $g_{0,n}.\widetilde{r_0} \cap K = \emptyset$ and $t_n<0$. If $K$ is taken big enough, we also have $t_n<c$. Therefore, for any $c\in \mathbb{R}^-$, we can find $n$ such that $t_n<c$. If $t_n<0$ and $|c|$ is sufficiently big, since the size of the $A_n$ is uniformly bounded, $\widetilde{\Phi^{t_n}}(A_n) $ will be a bifoliated compact disk, thin along the unstable direction and large along the stable one. This implies that for $n$ sufficiently big $\widetilde{\Phi^{t_n}}(A_n) $ can not be a vertical subrectangle of $\widetilde{r_0}$; hence, by the Markovian intersection axiom $r_{n} \cap r_0$ is a horizontal subrectangle of $r_0$, which contradicts the initial hypothesis. 

Therefore, there is a finite number of $n$ such that $t_n<0$ and by removing a finite number of $r_n$ from our sequence, we can assume that $t_n>0$ for all $n$.

By a similar argument, there exists a sequence $n_i \in \mathbb{N}$ such that $\min_{x\in A_{n_i}} t_{n_i}(x)$ is increasing and goes to infinity. Also, by uniform hyperbolicity and the fact that the $g_{0,n}.\widetilde{r_0}$ are isometric, there exists $c_{n_i}>0$ a decreasing sequence going to $0$ such that any stable segment of the disk $\widetilde{\Phi^{t_{n_i}}}(A_{n_i})$ has at most length $c_{n_i}$. We conclude that $\widetilde{r_0} \cap \overset{+\infty}{\underset{i=0}{\cap}} \widetilde{\Phi^{t_{n_i}}}(A_{n_i})$ is an unstable segment of $\widetilde{r_0}$ and we get the result we wanted. 
\end{proof}

Similarly: 
\begin{lemm}\label{l.infiniteintersectionhorizontalrectangles}
Consider a sequence of rectangles $(r_n)_{n\in \mathbb{N}}$ in $\mathcal{R}$ such that for every $k\in \mathbb{N}$, $r_{k+1}\cap r_k$ is a non-trivial horizontal subrectangle of $r_k$. We have that $\overset{+\infty}{\underset{k=0}{\cap}} r_{k}$ is a stable segment of $r_0$.
\end{lemm}

\begin{rema}\label{r.comparison}
In the previous lemmas, we establish a bridge between Markovian families and Markov partitions. Indeed, analogous statements are known to be true for Markov partitions. Fix $\mathcal{M}$ a Markov partition of $\Phi$ in $M$: 

\vspace{0.2cm}
\hspace{-0.8cm}
\begin{tabular}{ | p{4cm} | p{12cm}| } 
  \hline
  Markovian families & \hspace{4cm}   Markov partitions \vspace{0.1cm} \\
 
  \hline
  Lemmas \ref{l.vertsubrectangleexists} and \ref{l.horizsubrectangleexists}& Every positive or negative orbit (by $\Phi$) in $M$ intersects in bounded time a rectangle of $\mathcal{M}$ \\ 
  \hline
  Lemma \ref{l.periodicboundary} & The boundaries every rectangle in $\mathcal{M}$ belong to periodic stable and unstable leaves in $M$ \\ 
  \hline
 Lemmas \ref{l.infiniteintersectionverticalrectangles} and \ref{l.infiniteintersectionhorizontalrectangles}& Consider a rectangle $R$ of $\mathcal{M}$. A set of points in $R$ whose positive (resp.negative) orbits intersect the same infinite sequence of rectangles in $\mathcal{M}$ forms a stable (resp. unstable) segment of $R$ \\ 
  \hline
\end{tabular}

\end{rema}
The following lemma shows that as for any Markov partition, we can define a notion of first return map for any Markovian family. 
\begin{lemm}\label{l.existenceofpredecessors}
For any rectangle $R \in \mathcal{R}$ there exists a finite number of rectangles $R_1,...,R_n \in \mathcal{R}$ intersecting $R$ along non-trivial vertical subrectangles such that:
\begin{enumerate}
\item $R_1,...,R_n$ are maximal for this property: any $R' \in \mathcal{R}$ intersecting $R$ along a non-trivial vertical subrectangle verifies $R' \cap R \subseteq R_i \cap R$ for some $i \in \llbracket 1, n \rrbracket$ 
\item $R_1,...,R_n$ have disjoint interiors 
\item The $R_1,...,R_n$ cover $R$: $\overset{n}{\underset{i=1}{\cup}} R_i \cap R = R$ 
\end{enumerate}

\end{lemm}
The analogue of the previous lemma for horizontal subrectangles is also true. 
\begin{proof} 
Fix $R \in \mathcal{R}$. Let's call the property of intersecting $R$ along a non-trivial vertical subrectangle, property $(\star)$.  Lemma \ref{l.vertsubrectangleexists} assures the existence of rectangles satisfying $(\star)$. Let us begin by showing that for any point in $R$ there exists at least one rectangle maximal for $(\star)$ containing it. 

Indeed, let us fix $x\in R$. By Lemma \ref{l.vertsubrectangleexists}, there exists $r$ satisfying $(\star)$ containing $x$. Suppose there is no maximal rectangle for $(\star)$ containing $x$. Therefore, for any rectangle $r'$ intersecting $R$ along a vertical subrectangle,  there exists $r_1$ satisfying $(\star)$ such that $R \cap r' \subsetneq R\cap r_1$. We can thus construct by induction an infinite sequence $r=r_0,r_1,r_2,....,r_n,...$ such that $r_k\cap R \subsetneq r_{k+1}\cap R$ for all $k\in \mathbb{N}$. By the Markovian intersection axiom, $r_k \cap r_{k+1}$ is a horizontal subrectangle of $r_k$ and therefore by Lemma \ref{l.infiniteintersectionhorizontalrectangles}, we get that $\overset{+\infty}{\underset{k=0}{\cap}} r_{k}$ is a stable segment of $r_0$. But $r_k\cap R \subsetneq r_{k+1}\cap R$ for every $k$, so $r_0\cap R \subset \overset{+\infty}{\underset{k=0}{\cap}} r_{k}\cap R$, which is impossible, since $\overset{+\infty}{\underset{k=0}{\cap}} r_{k}$ is a segment. We deduce the existence of a maximal rectangle for $(\star)$ containing $x$.

Next, let us prove that maximal rectangles for $(\star)$ are either identical or they have disjoint interiors. Indeed, by the Markovian intersection axiom, two rectangles $R_i,R_j$ in $\mathcal{R}$ satisfying $(\star)$ have either disjoint interiors or they satisfy one of the following: $R \cap R_i \subseteq R\cap R_j$ or $R \cap R_i \subseteq R\cap R_j$. 

Finally, let us prove that the set of maximal rectangles for $(\star)$ is finite. Suppose the contrary.  Under this hypothesis, there exists a sequence $(r_i)_{i\in \mathbb{N}}$ of maximal rectangles for $(\star)$ such that the $\overset{\circ}{r_i}$ are pairwise disjoint. By compactness, we can assume that the rectangles $r_i \cap R$ converge to an unstable segment of $R$, say $s$. Without any loss of generality, assume that $r_i \cap R$ accumulate to $s$ from the right. Take $x\in s$ and suppose that $\mathcal{F}^s_-(x)$ is the stable separatrix on the right of $x$. By Lemma \ref{l.vertsubrectangleexists}, there exists $R' \neq R$ containing $x$ and a germ of the negative stable separatrix of $x$ such that $R'\cap R$ is a non-trivial vertical subrectangle of $R$ (hence $R'$ contains $s$). Since $R'$ intersects $\mathcal{F}^s_-(x)$ and contains $s$, it also contains all points of $R$ on the right of $s$ and sufficiently close to $s$. Therefore, for $i$ sufficiently big, $r_i \cap R \subset R'\cap R$, which contradicts the fact that the $r_i$ are maximal. Therefore, the set of maximal rectangles for $(\star)$ is finite.

\end{proof}

\begin{defi}\label{d.successor}
We will say that $R'$ is a \emph{predecessor} (resp. \emph{successor}) of $R$ if $R'\cap R$ is a non-trivial vertical (resp. horizontal) subrectangle of $R$ and $R'$ is maximal for this property in the sense of the previous lemma. 

We will say that $R'$ is a predecessor of \emph{$2$-nd generation} of $R$, if $R'$ is a predecessor of a predecessor of $R$. We define similarly a predecessor (resp. successor) of \emph{$n$-th generation} for any $n \in \mathbb{N}^*$.

\end{defi}

\begin{rema}\label{r.precedentsuivant}
If $R\in \mathcal{R}$ is a predecessor of $R'\in \mathcal{R}$ and $g\in \pi_1(M)$, then 

\begin{itemize}
    \item $g(R)$ is a predecessor of $g(R')$ and 
    \item $R'$ is a successor of $R$ 
\end{itemize}
\end{rema}
The first statement is an easy consequence of Definition \ref{d.successor} and the fact that $\mathcal{R}$ is preserved by the action of $\pi_1(M)$. Concerning the second statement, by the Markovian intersection property, since $R$ is a predecessor of $R'$,  $R\cap R'$ is a non-trivial horizontal subrectangle of $R$ (see Figure \ref{f.precedentsuivant}). If there exists $R''\in \mathcal{R}$ such that $R\cap R''$ is a non-trivial horizontal subrectangle of $R$ and $R'\cap R \subsetneq R''\cap R$, then  by the Markovian intersection axiom $R\cap R''$ is a non-trivial vertical subrectangle of $R''$ and  $R'\cap R''$ is a non-trivial vertical subrectangle of $R'$. Therefore, $R\cap R' \subsetneq R''\cap R'$, which contradicts the fact that $R$ is a predecessor of $R'$.
\begin{figure}[h!]
\includegraphics[scale=0.4]{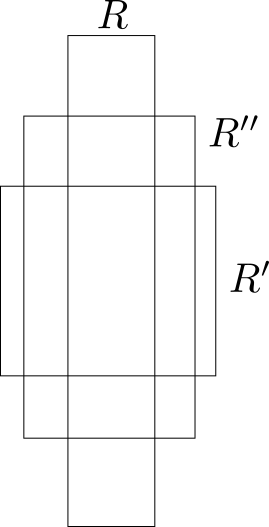}
\caption{}
\label{f.precedentsuivant}
\end{figure}
\begin{lemm}\label{l.npredecessor}
Take any $R \in \mathcal{R}$ and $R'$ intersecting $R$ along a non-trivial vertical subrectangle. Then $R'$ is a predecessor of $n$-generation of some $n\in \mathbb{N}^*$. 

\end{lemm}
\begin{proof}
By maximality, if $R'$ is not one of the predecessors of $R$ then it is contained in one of them, say $R_1$. Again, if it is not one of the predecessors of $R_1$, it is contained in one of them, say $R_2$. We construct in this way a sequence $R_0=R,R_1,...,R_n,...$ such that for every $n$, the rectangle  $R_{n+1}$ is the predecessor of $R_n$ containing $R'$. If there exists no $n$ such that $R_n=R'$, the previous sequence is infinite and by Lemma \ref{l.infiniteintersectionverticalrectangles}, $\overset{+\infty}{\underset{k=0}{\cap}} R_{k}$
is an unstable segment of $R$ containing $R'$, which is impossible. 
\end{proof}
\subsubsection*{The case of non-transversally orientable foliations} 

\begin{rema}\label{ss.nontransorientcase}
The Lemmas \ref{l.vertsubrectangleexists},  \ref{l.horizsubrectangleexists}, \ref{l.periodicboundary},  \ref{l.infiniteintersectionverticalrectangles},   \ref{l.infiniteintersectionhorizontalrectangles},   \ref{l.existenceofpredecessors} and \ref{l.npredecessor} remain true for transitive Anosov flows with non-transversally orientable foliations. 
\end{rema}

Indeed, let $\Phi$ be a transitive Anosov flow with non-transversally orientable foliations $F^s,F^u$. Let $\mathcal{P}$ be the bifoliated plane of $\Phi$. Consider now $M'$ the $2$-fold cover of $F^s$ and $F^u$. By lifiting $\Phi$ on $M'$, we obtain $\Phi'$ an Anosov flow with transversally orientable foliations. Notice that the bifoliated plane of $\Phi'$ coincides with $\mathcal{P}$ and that the action of $\pi_1(M)$ on $\mathcal{P}$ extends the action of $\pi_1(M')\leq \pi_1(M)$ on $\mathcal{P}$. 

Let $\mathcal{R}$ be a Markovian family of $\Phi$. It is easy to see that since $\pi_1(M')$ is subgroup of index $2$ of $\pi_1(M)$, $\mathcal{R}$ is a Markovian family of $\Phi'$ if and only if it is a Markovian  family for $\Phi$.

Lemma \ref{l.periodicboundary} remains true for $\Phi$, since the periodic points of $\mathcal{P}$ coincide for $\Phi$ and $\Phi'$. All the other lemmas concern only the properties of the intersections of rectangles of $\mathcal{R}$. By our previous arguments, $\mathcal{R}$ is a Markovian family for $\Phi$ and $\Phi'$. Therefore, all of the previously proven Lemmas remain true in the non-transversally orientable case.

We are now ready to proceed to the main theorem of this section. We will no longer assume that $\Phi$ has transversally orientable foliations. 
\begin{theorem}[Theorem A] \label{t.associatemarkovfamiliestogeometrictype}
Let $M$ be an orientable and closed 3-manifold, $\Phi$  a transitive Anosov flow on $M$ and $\mathcal{P}$ its bifoliated plane.  To any Markovian family $\mathcal{R}$ of $\Phi$ we can associate canonically a unique class of equivalent  geometric types. 
\end{theorem}

\begin{proof}
\textit{Construction of a geometric type}

 Let us begin by constructing a (non-canonical) geometric type associated to $\mathcal{R}$. We will later show that all geometric types constructed in this way are equivalent. 
 
We first need to define an integer $n\in \mathbb{N}^*$ associated to $\mathcal{R}$. The integer $n$ will correspond to the number of rectangles forming our geometric type (see Definition \ref{d.geometrictype}). The set $\mathcal{R}$ is the union of the (distinct) $\pi_1(M)$-orbits of a finite number of rectangles: $r_1,...,r_K \subset \mathcal{P}$. We take $n=K$. To each $r_i$ we associate one copy of $[0,1]^2$ trivially bifoliated, that we will denote by $R_i$. 

We define $h_i$ (resp. $v_i$) as the number of successors (resp. predecessors) of the rectangle $r_i$. It is easy to see that replacing $r_i$ by another rectangle in its orbit by $\pi_1(M)$ does not change change either $h_i$ or $v_i$. An easy consequence of Remark \ref{r.precedentsuivant} is that $$\overset{n}{\sum_{i=1}} h_i =\overset{n}{\sum_{i=1}} v_i$$
The previous $h_i$ and $v_i$ will thus play the role of the number of horizontal and vertical subrectangles in our geometric type. 

Choose now an orientation of the stable and unstable foliations $\mathcal{F}^s,\mathcal{F}^u$ in $\mathcal{P}$ (this is always possible for a line foliation in the plane).

Consider for every $R_i$, $h_i$ horizontal subrectangles and $v_i$ vertical subrectangles, denoted respectively by $H_i^1,...,H_i^{h_i}$ and $V_i^1,...,V_i^{v_i}$. We will assume that all the $H_i^j$ (resp. $V_i^j$) are mutually disjoint and ordered from bottom to top (resp. from left to right). Using the orientation of the unstable (resp. stable) foliation on $\mathcal{P}$ we can order the successors (resp. predecessors) of $r_i$ from bottom to top (resp. from left to right). The previous orders allow us to identify (bijectively) the successors (resp. predecessors) of $r_i$ with the rectangles $H_i^1,...,H_i^{h_i}$ (resp.$V_i^1,...,V_i^{v_i}$).

We define $\mathcal{H}$ and $\mathcal{V}$ as $\mathcal{H}:=\lbrace H^j_i| i\in\llbracket 1,n \rrbracket, j \in \llbracket 1,h_i \rrbracket  \rbrace$ and $\mathcal{V}:=\lbrace V^j_i| i\in\llbracket 1,n \rrbracket, j \in \llbracket 1,v_i \rrbracket  \rbrace$. In order to define a geometric type associated to $\mathcal{R}$, we still need to define the functions $\phi, u$ (we follow here the notations of Definition \ref{d.geometrictype}). 

Take any $i \in\llbracket 1,n \rrbracket$ and any successor of $r_i$, say $R$. We know that $R$ belongs to the $\pi_1(M)$-orbit of some $r_j$, where $j \in\llbracket 1,n \rrbracket$. Consider $g\in \pi_1(M)$ such that $g(R)=r_j$. The previous $g$ is unique. Indeed, suppose that there exists another $g'\in \pi_1(M)$ such that $g'(R)=r_j$. In this case, we have that $g'^{-1}(g(R))=R$, therefore $g'^{-1}\circ g$ has a fixed point in $R$. Since $g'\neq g$, by eventually considering its inverse, we can assume that $g'^{-1}\circ g$ acts as an expansion along the stable leaf of its fixed point in $R$, which contradicts the fact that $g'^{-1}(g(R))=R$. Therefore, the element $g$ is unique.

Suppose now that $R$ is the $k$-th successor of $r_i$ (for the bottom to top order) and that $g(r_i)$ is the $l$-th predecessor of $r_j$ (for the left to right order, see also Remark \ref{r.precedentsuivant}). We define $\phi(H_i^k)=V_j^l$. Furthermore, we define $u(H_i^k)=+1$ if $g$ preserves the orientation of the foliations and $u(H_i^k)=-1$ if not.

It is not difficult to see that  $G=(n,(h_i)_{i \in \llbracket 1,n \rrbracket}, (v_i)_{i\in \llbracket 1,n \rrbracket}, \mathcal{H}, \mathcal{V},\phi, u)$ is a geometric type, whose construction depended  solely on our choice of orientation of $\mathcal{F}^s$ and $\mathcal{F}^u$ and also on our initial choice of rectangles $r_i$. We would now like to show that modifying the orientation of $\mathcal{F}^s$ or $\mathcal{F}^u$ or changing our initial choice of rectangles $r_i$ leads to the construction of geometric types that are equivalent to $G$.

\textit{A unique equivalence class of geometric types}

Suppose that we change our initial choice of rectangles, while keeping fixed our initial  orientations of $\mathcal{F}^{s}$ and $\mathcal{F}^{u}$. Say that the $r_i$ have been reordered and then replaced by the rectangles $g_i(r_i)$ (with $g_i\in \pi_1(M)$). By our previous arguments, we can canonically associate to $\mathcal{R}$ together with this choice of rectangles and orientations  a geometric type $G':=(n, (h'_i)_{i \in \llbracket 1,n \rrbracket}, (v'_i)_{i\in \llbracket 1,n \rrbracket}, \mathcal{H'}, \mathcal{V'},\phi', u')$. Using the analogy between a geometric type and a set of rectangles, once again:
\begin{itemize}
    \item the integer $n$ corresponds to $n$ rectangles in $G'$, say $R_1',...,R_n'$
    \item every $R'_i$ contains $h'_i$ (resp. $v'_i$) horizontal (resp. vertical) subrectangles two by two disjoint and ordrered from bottom to top (resp. left to right), say ${H'_i}^1,...,{H'_i}^{h_i}$ (resp. ${V'_i}^1,...,{V'_i}^{v_i}$) 
\end{itemize}

We will show that $G$ and $G'$ are equivalent. Indeed, up to reindexing the $R_i'$, for every $i$, the rectangles $R_i$ and $R'_i$ correspond to the same orbit of rectangles in $\mathcal{R}$. Therefore, by our previous construction, we have that $h_i=h'_i$ and $v_i=v'_i$.

We would now like to define a bijection from $\mathcal{H}\cup \mathcal{V}$ to $ \mathcal{H'}\cup \mathcal{V'}$. There exists a natural way to associate any rectangle in $\mathcal{H}$ to a rectangle in $\mathcal{H}'$. Indeed, recall that for every $i \in\llbracket 1,n \rrbracket$ and every $j \in \llbracket 1,h_i \rrbracket$, the rectangle ${H_i}^j$ (resp.${H'_i}^j$) corresponds to a successor of $r_i$ (resp. $g_i(r_i)$), say $R_{i,j}\in \mathcal{R}$ (resp. $R'_{i,j}$). Since the successors of $r_i$ and $g_i(r_i)$ are ordered from bottom to top, if $g_i$ preserves the orientation of the foliations, $g_i(R_{i,j})=R'_{i,j}$. If not, $g_i(R_{i,j})=R'_{i,(v_i-j)}$. We can therefore associate in the first case ${H_i}^j$ to ${H'_i}^j$ and in the second ${H_i}^j$ to ${H'_i}^{(h_i-j)}$. By symmetry the previous map from $\mathcal{H}$ to $\mathcal{H}'$ admits an inverse, it thus defines a bijection $H_h:\mathcal{H}\rightarrow \mathcal{H'}$. Moreover, notice that for every $i$ we have that $H_h(\{{H_i}^j, j\in \llbracket 1, h_i\rrbracket \})= \{{H'_i}^{j'}, j'\in \llbracket 1, h_i\rrbracket\}$ and that $H_h$ is increasing with respect to $j$ when $g_i$ preserves the orientations of the foliations and decreasing when $g_i$ reverses the orientations of the foliations. In the exact same way, we can construct a bijection $H_v$ from $\mathcal{V}$ to $\mathcal{V'}$ with similar properties. By combining $H_v$ and $H_h$, we can define a bijection $H:\mathcal{H}\cup \mathcal{V}\rightarrow  \mathcal{H'}\cup \mathcal{V'}$. Let us now show that $H$ is an equivalence between $G$ and $G'$. 
 


Define $\epsilon(R_i)=\epsilon'(R_i)=+1$ if $g_i$ respects the orientation of the foliations and $\epsilon(R_i)=\epsilon'(R_i)=-1$ otherwise. Notice that for every $i$ we have that $\epsilon(R_i)\cdot \epsilon'(R_i)=+1$. It suffices to show that for any rectangle $h\in \mathcal{H}$ such that $h\subset R_i$ and $\phi(h)\subset R_j$, we have  $H\circ \phi(h)= \phi'\circ H(h)$ and $u'(H(h))=\epsilon(R_i)\epsilon(R_j)u(h)$. Let us start by proving the first equality. By construction, the subrectangles in $\mathcal{V}'$ are bijectively identified with the predecessors of the rectangles $g_i(r_i)$; it thus suffices to show that both $H\circ \phi(h) \in \mathcal{V}'$ and $\phi'\circ H(h) \in \mathcal{V}'$ correspond to the same predecessor of $g_j(r_j)$.
 
 Recall that $h$ corresponds to a successor $R$ of $r_i$, that there exists a unique $G\in \pi_1(M)$ such that $G(R)=r_j$ and that the rectangle $\phi(h)$ corresponds to the predecessor $G(r_i)$ of $r_j$. Therefore, by our definition of $H$, the rectangle $H(\phi(h))$ corresponds to the predecessor $g_j(G(r_i))$ of $g_j(r_j)$. 
 
Similarly, the rectangle $H(h)$ corresponds to the successor $g_i(R)$ of $g_i(r_i)$. Also, by our construction in  \textit{Construction of a geometric type}, since $(g_j\circ G \circ g_i^{-1})(g_i(R))=g_j(r_j)$, the rectangle  $\phi'(H(h))$ corresponds to the predecessor $(g_j\circ G \circ g_i^{-1})(g_i(r_i))=g_j(G(r_i))$ of $g_j(r_j)$. This proves the first equality. Recall now that by construction of $u$ and $u'$, $u'(H(h))=+1$ (resp. $u(h)=+1$) if and only if $g_j\circ G \circ g_i^{-1}$  (resp. $G$) preserves the orientation of the stable/unstable foliations in $\mathcal{P}$. The previous statement is equivalent to  $u'(H(h))=\epsilon(R_i)\epsilon(R_j)u(h)$.

 Our previous arguments show that $G$ and $G'$ are equivalent geometric types. By a similar argument, we can show that changing the orientation of $\mathcal{F}^u$ or $\mathcal{F}^s$ also leads to geometric types equivalent to $G$, which concludes the proof of the theorem. 

\end{proof}
 In the above proof we have shown that
\begin{rema}\label{r.canonicalassociationgeometrictype}
Let $\mathcal{R}$ be a Markovian family in $\mathcal{P}$. Choose $r_1,..,r_n$ representatives of every rectangle orbit (by $\pi_1(M)$) in $\mathcal{R}$ and an orientation of $\mathcal{F}^s$ and $\mathcal{F}^u$. For each one of the previous choices, we can canonically associate a geometric type to $\mathcal{R}$. 
\end{rema}
\begin{lemm}\label{l.geomtypeinclass}
Let $G=(n,(h_i)_{i \in \llbracket 1,n \rrbracket}, (v_i)_{i\in \llbracket 1,n \rrbracket}, \mathcal{H}, \mathcal{V},\phi, u)$ be a geometric type in the class of geometric types canonically associated to $\mathcal{R}$. When the flow $\Phi$ does not have transversally orientable foliations, by an appropriate choice of representatives of every rectangle orbit in  $\mathcal{R}$ and an appropriate choice of orientation of $\mathcal{F}^s$ and $\mathcal{F}^u$, we can canonically associate $\mathcal{R}$ to $G$. If we also assume that $u_{|\mathcal{H}}=1$, the same result is true when the flow $\Phi$ has transversally orientable foliations 
\end{lemm}
\begin{proof}
Let us first chose arbitrarily representatives $r_1,...,r_n$ for every rectangle orbit in $\mathcal{R}$ and orientations for $\mathcal{F}^s$ and $\mathcal{F}^u$. Thanks to the previous choices,  Theorem \ref{t.associatemarkovfamiliestogeometrictype} and Remark \ref{r.canonicalassociationgeometrictype}, we can canonically associate $\mathcal{R}$ to a geometric type $G'$ equivalent to $G$. Up to permuting the $(h_i,v_i)$ we can assume that $G'=(n,(h_i)_{i \in \llbracket 1,n \rrbracket}, (v_i)_{i\in \llbracket 1,n \rrbracket}, \mathcal{H'}, \mathcal{V'},\phi', u')$. Let us first assume that the flow $\Phi$ does not have transversally orientable foliations. We will now adapt our choices of representatives and orientations in order to canonically obtain $G$. 

By definition, there exists an equivalence $H$ between $G$ and $G'$. Following the notations of Definition \ref{d.equivalentgeomtypes}, consider for every $i$ the integers $\epsilon_i,\epsilon_i'\in \{-1,+1\}$. Recall that either $\epsilon_i \cdot \epsilon_i'=-1$ for all $i$ or $\epsilon_i \cdot \epsilon_i'=+1$ for all $i$. Changing the orientation of $\mathcal{F}^s$, changes the order of all vertical subrectangles in $G'$ (i.e. the elements in $\mathcal{V}'$) and  replaces every  $\epsilon_i'$ by $-\epsilon_i'$, while keeping all the $\epsilon_i$ fixed. Therefore, by eventually changing the orientation of $\mathcal{F}^s$, we can assume that for every $i$ we have $\epsilon_i \cdot \epsilon_i'=+1$. 

Suppose that there exists $j$ such that $\epsilon_j= \epsilon_j'=-1$. Take $g\in\pi_1(M)$ any element reversing the orientation of the foliations. Replacing $r_j$ by $g(r_j)$, changes the order of the horizontal and vertical subrectangles in $r_j$, thus simultaneously changes the signs of both $\epsilon_j$ and  $\epsilon_j'$. We can therefore assume by an eventual change of representatives that all the $\epsilon_i$ and $ \epsilon_i'$ are equal to $+1$. By definition, this implies that $G'=G$ which gives us the desired result. 

If the flow has transversally orientable foliations, by the same argument, we can assume that for every $i$ we have $\epsilon_i \cdot \epsilon_i'=+1$. Let us now suppose that for all $i$ we have that $\epsilon_i= \epsilon_i'=-1$. By eventually changing the orientations of $\mathcal{F}^s$ and $\mathcal{F}^u$, we can once again change simultaneously the signs of all the $\epsilon_i$ and $\epsilon_i'$, which gives us the desired result. Assume now that there exists $j_0,j_1$ such that $\epsilon_{j_0}= \epsilon_{j_0}'= -\epsilon_{j_1}= -\epsilon_{j_1}'=-1$. It is not difficult to see that in this case, there exists $h\in \mathcal{H}$ or $h'\in \mathcal{H}'$ such that $u(h)=-1$ or $u'(h')=-1$ respectively. By hypothesis $u$ cannot be negative and by our construction in Theorem \ref{t.associatemarkovfamiliestogeometrictype}, $u'$ cannot be negative too since all elements in $\pi_1(M)$ preserve the orientations of the foliations. We thus arrive to an absurd and we get the desired result. 
\end{proof}
\section{Boundary periodic points and arc points}\label{s.boundarypoints}
A Markovian family does not cover all the points of the bifoliated plane in the same way: for every Markovian family there exist points that do not belong in the interior of any rectangle of the family. Among those points, we can distinguish non periodic points that we will call \textit{boundary arc points} and periodic points that we will call  \textit{boundary periodic points}. 

The sets of boundary arc and boundary periodic points will be of great importance to us throughout this paper. In particular, in view of Theorem B, we will prove that a class of geometric types associated to a Markovian family contains all the information concerning the flow, except from the behaviour of the flow around the boundary periodic orbits. In this section, our goal is to show that for any Markovian family the set of boundary arc points and also the set of boundary periodic orbits are non-empty and are the union of the $\pi_1(M)$-orbits of a finite number of points.

As in the previous section, we fix $M$ an orientable closed 3-manifold carrying a transitive Anosov flow $\Phi$ endowed with its weak stable and unstable foliations $F^s,F^u$,  $(\mathcal{P}, \mathcal{F}^s, \mathcal{F}^u)$ the bifoliated plane of $\Phi$ endowed with an orientation and $\mathcal{R}$ a Markovian family of $\mathcal{P}$.
\begin{defi}
Consider $x\in\mathcal{P}$ that is not contained in the interior of any rectangle of $\mathcal{R}$. If $x$ is a periodic point, we will call $x$ a \emph{boundary periodic point}. If not, we will call $x$ a \emph{boundary arc point}. 

\end{defi} 

\begin{prop}\label{p.boundaryperiodicpoints}
The set of boundary periodic orbits of $\mathcal{R}$ is non-empty and is the union of a finite number of orbits by the action of $\pi_1(M)$. 
\end{prop}
\begin{proof}
By Lemma \ref{l.periodicboundary}, the stable and unstable boundary segments of any rectangle in $\mathcal{R}$ belong to periodic leaves. By the finiteness axiom, we have that the set of periodic leaves containing a stable/unstable boundary component of some rectangle in $\mathcal{R}$ -we will denote this set by $A$- corresponds to the union of the $\pi_1(M)$-orbits of a finite number of leaves. Therefore, the set of periodic points belonging to a leaf in $A$ - we will denote this set by $B$- is a also the union of the $\pi_1(M)$-orbits of a finite number of points. The set of boundary periodic points is a subset of $B$. We have therefore shown the second part of the proposition. 

In order to show that there exists at least one boundary periodic point, it suffices to show that every point of $B$ is a boundary periodic orbit. Take $p\in B$. 

Suppose $p$ is contained in the interior of some rectangle $R\in \mathcal{R}$. Since $p\in B$, we can assume without any loss of generality that $\mathcal{F}^s(x)$ contains a stable boundary component, say $s$, of some rectangle $R' \in \mathcal{R}$. Let $g \in \text{Stab}(p) \subset \pi_1(M)$ such that $g(s)\subset \int{R}$. The intersection of $g(R')$ and $R$ does not satisfy the Markovian intersection axiom. Therefore $p$ is not contained in the interior of any rectangle in $\mathcal{R}$ and is a boundary periodic point. 
\end{proof}
\begin{prop}\label{p.arcpointsexist}
The set of boundary arc points of $\mathcal{R}$ is non-empty and is the union of a finite number of orbits by the action of $\pi_1(M)$. 
\end{prop}

In order to prove the above proposition, we will use the following lemma that resembles closely Lemma \ref{l.existenceofpredecessors}. 
\begin{lemm}\label{l.crossingrectanglesnoperiodicpoints}
Take a rectangle $R \in \mathcal{R}$ and $s$ one of its stable boundaries. Suppose that $s$ contains no periodic point, then there exist $R_1,...,R_n \in \mathcal{R}$ ($n\in \mathbb{N}$) intersecting $R$ along non-trivial vertical subrectangles and whose stable boundaries don't intersect $s$ such that: 
\begin{enumerate}
\item $R_1,...,R_n$ are maximal for this property: any $R' \in \mathcal{R}$ intersecting $R$ along a non-trivial vertical subrectangle and whose stable boundary doesn't intersect $s$ verifies $R' \cap R \subseteq R_i \cap R$ for some $i \in \llbracket 1, n \rrbracket$ 
\item $R_1,...,R_n$ have disjoint interiors 
\item $R_1,...,R_n$ cover $R$: $\overset{n}{\underset{i=1}{\cup}} R_i \cap R = R$ 
\end{enumerate}
\end{lemm}
\begin{proof}
The proof of this lemma is an adaptation of the proof of Lemma \ref{l.existenceofpredecessors}. Let us call the property of ``intersecting $R$ along a non-trivial vertical subrectangle and having stable boundaries not intersecting $s$", property $(\star)$. Suppose without any loss of generality that $s$ is the lower stable boundary of $R$ (i.e. for every $x\in s$,  $(\mathcal{F}^s_{\epsilon}(x)-x) \cap R \neq \emptyset$). Let us prove that for any point $x\in s$, there exists $r \in \mathcal{R}$ verifying $(\star)$, containing $x$ and such that $(\mathcal{F}^s_{\epsilon}(x)-x) \cap r \neq \emptyset$. Let us fix a point $x\in s$ and suppose that there is no rectangle in $\mathcal{R}$ with this property. 

By Lemma \ref{l.existenceofpredecessors}, $x$ belongs to exactly one predecessor of $R$ intersecting non-trivially $\mathcal{F}^s_+(x)$, say $R_1$. By our assumption, $R_1$ doesn't verify $(\star)$. The point $x$ also belongs to exactly one predecessor of $R_1$ intersecting non-trivially $\mathcal{F}^s_{+}(x)$, say $R_2$. Again, $R_2$ doesn't verify $(\star)$. By induction, we construct $R_0=R,R_1,...,R_n...$ such that $R_{n+1}$ is a predecessor of $R_n$, none of the $R_n$ satisfy $(\star)$ and every $R_n$ intersects non-trivially $\mathcal{F}^s_+(x)$. The fact that the $R_n$ don't satisfy $(\star)$ implies that all the $R_n$ contain $x$ in their stable boundaries. 

By the finiteness axiom, there exists $(R_{i(n)})_{n\in \mathbb{N}}$ a subsequence of $(R_n)_{n\in \mathbb{N}}$ containing rectangles in the same $\pi_1(M)$-orbit. Consider $g_n \in \pi_1(M)$ such that $g_n.R_{i(n)}=R_{i(n+1)}$. By eventually extracting a subsequence, we can assume that the $g_n$ preserve the orientation of the foliations. Hence, every $g_n$ preserves $\mathcal{F}^s(x)$ and  $g_n\in \text{Stab}(p)$, where $p$ is the unique periodic point in $\mathcal{F}^s(x)$. Let us denote $s'$ the stable boundary of $R_{i(0)}$ containing $x$. We have that $g_n \circ g_{n-1} \circ ... g_0 (s') \subset g_{n-1} \circ ... g_0 (s') \subset ... \subset s' \subset s$. We deduce from this that the periodic point on $\mathcal{F}^s(x)$ is in $g_n \circ g_{n-1} \circ ... g_0 (s')$ for every $n$. This implies together with Lemma \ref{l.infiniteintersectionverticalrectangles}, that the segment $g_n \circ g_{n-1} \circ ... g_0 (s')$ tends to become a point; more specifically the point $p$, which is impossible since $s$ does not contain periodic points. 
\begin{figure}[h!]
\includegraphics[scale=0.5]{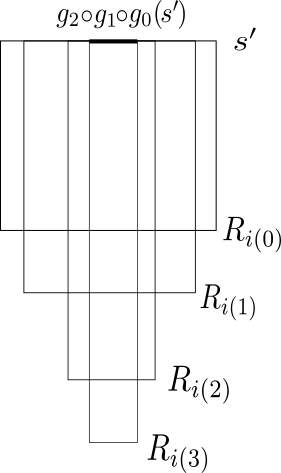}
\caption{}
\label{f.proof34}
\end{figure}

The existence of a maximal rectangle satisfying $(\star)$ for any point $x\in s$, the fact that the maximal rectangles for $(\star)$ are finite and cover $R$ follow by the same arguments that were used in the proof of Lemma  \ref{l.existenceofpredecessors}. 
\end{proof}

Of course the analogue of the previous lemma for horizontal rectangles is also true. By a similar argument, we can also prove the following: 
\begin{lemm}\label{l.crossingrectangleswithperiodicpoints}
Take a rectangle $R \in \mathcal{R}$ and $s$ one of its stable boundary components. Suppose that $s$ contains a periodic point $p$. There exist $R_1,...,R_n \in \mathcal{R}$ ($n\in \mathbb{N}$) intersecting $R$ along non-trivial vertical subrectangles and whose stable boundaries don't intersect $s$ such that 
\begin{enumerate}
\item the $R_1,...,R_n$ are ``maximal" for this property: any $R' \in \mathcal{R}$ intersecting $R$ along a non-trivial vertical subrectangle and whose stable boundary doesn't intersect $s$ verifies $R' \cap R \subseteq g.R_i \cap R$ for some $i \in \llbracket 1, n \rrbracket$ and some $g\in \text{Stab}(p)$
\item $R_1,...,R_n$ have disjoint interiors 
\item $R_1\cap s,...,R_n \cap s$ cover a fundamental domain of $\mathcal{F}^s(p)$ for the action of $\text{Stab}^+(p)$, the group of elements in $\text{Stab}(p)$ preserving the orientation of the foliations. 
\end{enumerate}
\end{lemm}

\begin{defi}
Take a rectangle $R \in \mathcal{R}$ and $s$ one of its stable boundary components. We will call $R'$ an \emph{$s$-crossing predecessor} of $R$ if it is maximal for the following property (in the sense of the previous lemmas): it intersects $R$ along a non-trivial vertical subrectangle and $s\cap \partial^s R' = \emptyset$. We define similarly an \emph{$s$-crossing successor} when $s$ in an unstable boundary component of $R$. 
\end{defi}

\begin{proof}[Proof of Proposition \ref{p.arcpointsexist}]
It suffices to show that for every $R\in \mathcal{R}$, if $s$ is a stable (resp. unstable) boundary component of $R$, then the intersection of $s$ with the unstable (resp. stable)  boundary of one  $s$-crossing predecessor (resp.  successor ) of $R$ consists of two boundary arc points. Indeed if this is the case, by Lemmas \ref{l.crossingrectanglesnoperiodicpoints} and \ref{l.crossingrectangleswithperiodicpoints}, the set of boundary arc points is non-empty. Furthermore, by Lemma \ref{l.crossingrectanglesnoperiodicpoints} if $s$ is either a stable or unstable boundary component of $R$ that contains no periodic points the number of boundary arc points in $s$ is finite. If $s$ contains a periodic point, by Lemma \ref{l.crossingrectangleswithperiodicpoints} the number of boundary arc points in $s$ is finite up to the action of the stabilizer of the periodic point. Therefore, since the number of rectangles of any Markovian family is finite up to the action of $\pi_1(M)$, the number of boundary arc points up to the action of $\pi_1(M)$ is also finite. 
 
Let us now prove our original claim. Suppose $x$ is a boundary arc point. Take any rectangle $R$ in $\mathcal{R}$ containing $x$. Such a rectangle exists because a Markovian family always covers $\mathcal{P}$. Since $x$ is a boundary arc point, it belongs to the boundary of $R$. Without any loss of generality, let us assume that $x$ belongs to a stable boundary component of $R$, say $s$. By the  Lemmas \ref{l.crossingrectanglesnoperiodicpoints} and \ref{l.crossingrectangleswithperiodicpoints}, we have that $x$ belongs to some $s$-crossing predecessor of $R$, say $R'$. But since $x$ is a boundary arc point and $R'$ is a $s$-crossing predecessor of $R$, $x$ belongs to the unstable boundary of $R'$.

Conversely, take a rectangle $R\in \mathcal{R} $, $s$ a stable boundary component of $R$, $R'$ a $s$-crossing predecessor of $R$ and $x \in \partial^u R' \cap s$. Let us show that $x$ is a boundary arc point. Let's start by proving that $x$ is not a periodic point. Suppose the contrary. Since $R'$ is $s$-crossing, $x$ is not a corner point of the rectangle $R'$. Without any loss of generality, let's assume that $x$ is contained in the right unstable boundary component of $R'$ (see Figure \ref{f.Proofprop33}(a)). Let us denote $s'$ the connected component of $s-\lbrace x \rbrace$ on the left of $x$. Notice that $s'\neq \emptyset$. Consider now $g \in \text{Stab}(x)$ such that $g(s')\subset \int{R'}$. Since $x$ is not a corner point of $R'$, the intersection of $g(R)\in \mathcal{R}$ and $R'$ does not verify the Markovian intersection axiom. Therefore, $x$ in not periodic. 

Finally, let us show that $x$ is in the interior of no rectangle in $\mathcal{R}$. Suppose that $x$ is contained in the interior of $R'' \in \mathcal{R}$ (see Figure \ref{f.Proofprop33}(b)). In this case, $\overset{\circ}{R'} \cap   \overset{\circ}{R''} \neq \emptyset$ and also $\overset{\circ}R \cap   \overset{\circ}{R''} \neq \emptyset$. It is not difficult to see that  the Markovian intersection axiom implies that $R'\cap R''$ is a vertical subrectangle of $R''$ and $R\cap R''$ is also a vertical subrectangle of $R$. This contradicts the maximality of the $s$-crossing predecessor $R'$ and finishes the proof of our initial claim. 

\begin{figure}[h]
  \begin{minipage}[ht]{0.4\textwidth}
    \centering    
    \includegraphics[width=0.5\textwidth]{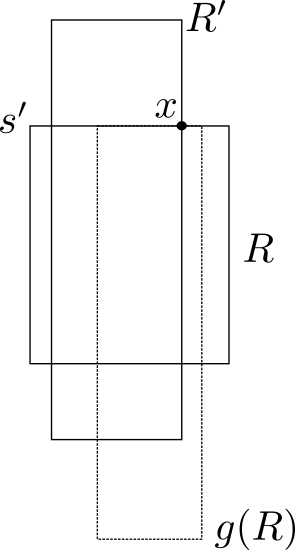}
    \caption*{\quad (a)}
    
  \end{minipage}
 \begin{minipage}[ht]{0.4\textwidth}
 \centering
    \includegraphics[width=0.5\textwidth]{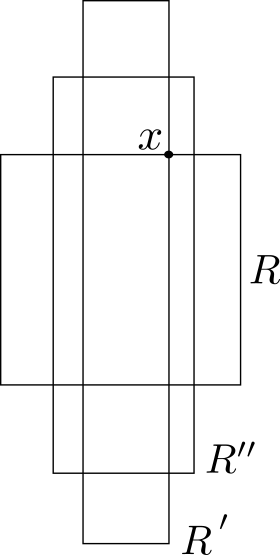}
    \caption*{\quad (b)}
    
  \end{minipage}
  \caption{(a) This configuration does not satisfy the Markovian intersection property (b) This is the only possible configuration for which $x\in \int{R''}$}
  \label{f.Proofprop33}
\end{figure}
  
\end{proof}




\begin{rema}\label{r.caracterisationarcpoints}In the proofs of Propositions \ref{p.boundaryperiodicpoints} and \ref{p.arcpointsexist} we established that:
\begin{itemize}
\item A periodic point that is contained in the boundary of one rectangle in $\mathcal{R}$ is a boundary periodic points and vice versa. 

\item Consider $R\in \mathcal{R}$ and $s$ a stable (resp. unstable) boundary component of $R$. A point in the intersection of $s$ with the unstable (resp. stable)  boundary of one $s$-crossing predecessor (resp.  successor ) of $R$ is a boundary arc point. The converse is also true. 
\end{itemize}
\end{rema}
\section{Rectangle paths}\label{s.rectpaths}
\subsection{Rectangle paths in $\mathcal{P}$}\label{s.rectanglepathsinP}
 A Markovian family in $\mathcal{P}$ can function as a coordinate system of the bifoliated plane. We can travel inside $\mathcal{P}$ using \textit{rectangle paths} and describe points of  $\mathcal{P}$ as an infinite intersection of rectangles in $\mathcal{R}$. The goal of this section is to define the notion of rectangle paths for any Markovian family, to associate to a curve a rectangle path and finally to show that all rectangles in $\mathcal{P}$ can be reached via a rectangle path. In view of Theorem B, this will later allow us to locally identify bifoliated planes endowed with Markovian families associated to the same class of geometric types. We fix again $\mathcal{R}$ a Markovian family of $\mathcal{P}$. 

\begin{defi}\label{d.rectanglepath}
A finite sequence of rectangles in $\mathcal{R}$ of the form $R_0$,$R_1$,...,$R_n$ will be called a \emph{rectangle path} going from $R_0$ to $R_n$ if for every $i \in \llbracket 0, n-1 \rrbracket$ the rectangle $R_{i+1}$ is either a successor or predecessor of $R_i$. The length of the sequence defining the rectangle path will be called length of the rectangle path. We will say that the rectangle path $R_0$,$R_1$,...,$R_n$ is 
\begin{itemize}
    \item closed if it also verifies $R_n=R_0$
    \item increasing (resp. decreasing) if it also verifies that $R_{i+1}$ is a successor (resp. predecessor) of $R_i$ for every $i$ 
    \item monotonous if it is increasing or decreasing
\end{itemize}
\end{defi}
\begin{defi}\label{d.polygonalcurve}
A continuous curve $\gamma: [0,1]\rightarrow \mathcal{P}$ that is a finite juxtaposition of alternating stable and unstable segments will be called a \emph{polygonal curve} in $\mathcal{P}$. Furthermore, if the curve $\gamma$ is closed, we will call it a \emph{closed} polygonal curve.
\end{defi}
 \begin{rema}\label{r.gamma0}
 Notice that by our above definition if $\gamma$ is a good polygonal curve, then $\gamma(0)$ is necessarily a corner point of $\gamma$ (i.e. not in the interior of any segment of $\gamma$). 
 \end{rema}
\begin{defi}\label{d.goodcurve}
By definition, for any polygonal curve $\gamma$ there exist $c_0=0<c_1<...<c_n=1$ dividing $\gamma$ in alternating stable and unstable segments. The number $n$ will be called the \emph{length} of the polygonal curve. Also, a polygonal curve $\gamma$ will be called \emph{good} if 
\begin{enumerate}
\item none of its stable or unstable segments belongs to the stable or unstable leaf of a boundary periodic point 
\item there don't exist $i\neq j\in \llbracket 1,n-1\rrbracket$ such that $\gamma([c_i,c_{i+1}])$ and $\gamma([c_j,c_{j+1}])$ belong to the same stable or unstable leaf of $\mathcal{P}$
\end{enumerate}

\end{defi}

\begin{lemm}\label{l.createpolygonalcurve}
Any smooth curve $\gamma: [0,1]\rightarrow \mathcal{P}$ for which $\gamma(0)$ and $\gamma(1)$ do not belong to a stable or unstable leaf of a boundary periodic point is homotopic relatively to its boundary to a good polygonal curve in $\mathcal{P}$

\end{lemm}
\begin{proof}
Indeed, we can locally deform $\gamma$ to a finite juxtaposition of alternating stable and unstable segments. By compactness, $\gamma$ is homotopic relatively to its boundary to a polygonal curve $\gamma'$ in $\mathcal{P}$. Without any loss of generality, we can assume that $\gamma'$ is the juxtaposition of a stable segment $s_1$, followed by an unstable segment $u_1$,..., ending with an unstable segment $u_n$. 

Suppose that one of the stable segments in $\gamma'$ belongs to a stable leaf of a boundary periodic point. Those points are finite up to the action of $\pi_1(M)$, therefore the stable leaves of those points are only countable in $\mathcal{P}$. This stable segment cannot be $s_1$ because of our initial hypothesis. Let us denote this segment by $s_k$. By changing a little bit the ``lengths" of the segments $u_{k-1}$ and $u_k$, we can replace $s_k$ by another stable segment in $s_k$'s neighbourhood that does not belong to the countable set of stable leaves of boundary periodic points. By a repeated application of the previous argument, $\gamma'$ is homotopic relatively to its boundary to a polygonal curve satisfying the axiom 1 of the above definition.

By the exact same procedure, we can show that $\gamma'$ is homotopic relatively to its boundary to a polygonal curve satisfying also the second axiom of the above definition; hence that $\gamma'$ is homotopic relatively to its boundary to a good polygonal curve.  

\end{proof}

\begin{prop}\label{p.curvetorectanglepath}
Any good polygonal curve $\gamma$ in $\mathcal{P}$ together with a rectangle $r_0\in \mathcal{R}$ such that $\gamma(0)\in \overset{\circ}{r_0}$ is canonically associated to a rectangle path starting from $r_0$. 

\end{prop}
\begin{proof}
Take a good polygonal curve $\gamma:[0,1] \rightarrow \mathcal{P}$. Assume without any loss of generality, that $\gamma$ is the juxtaposition of the stable segment $s_1$, followed by the unstable segment $u_1$,..., ending with the unstable segment $u_n$. We will now construct a sequence of rectangles that will correspond to the desired rectangle path. Since $\gamma(0)\in \overset{\circ}{r_0}$, we will begin our sequence by $r_0$. 

Suppose that $s_1$ crosses the unstable boundary component $S$ of $r_0$. Because of Definition \ref{d.goodcurve} and Remark \ref{r.caracterisationarcpoints}, $s_1$ doesn't belong to the stable leaf of a boundary arc or boundary periodic point. Therefore, by Lemmas \ref{l.crossingrectanglesnoperiodicpoints}, \ref{l.crossingrectangleswithperiodicpoints},  there exists a unique $S$-crossing successor of $r_0$, say $R^{(0)}$ containing $S\cap s_1$. By the Lemma \ref{l.npredecessor} and the fact that the successors of a rectangle have disjoint interiors, there exists a unique sequence of rectangles $R_0=r_0,R_1,...,R_n=R^{(0)}$ such that $R_{i+1}$ is a successor of $R_i$ for every $i\in \llbracket 0, n-1 \rrbracket$. This sequence of rectangles will constitute the first part of our rectangle path. If $s_1$ exits $R^{(0)}$, we apply the same algorithm as before to extend our sequence of rectangles and by our previous arguments $s_1$ will visit a crossing successor of $R^{(0)}$, say $R^{(1)}$. 

Suppose now that this algorithm never stops for $s_1$ (therefore the sequence $(R_n)_{n\in \mathbb{N}}$ becomes infinite thus not corresponding to a rectangle path). Under this hypothesis, we would  construct an infinite sequence $(R^{(n)})_{n\in \mathbb{N}}$ such that $R^{(i+1)}$ is a crossing successor of $R^{(i)}$ and $s_1$ exits every $R^{(i)}$. Let us denote by $F$ the stable separatrix of $\gamma(0)$  containing $s_1$. The sequence of closed segments $R^{(i)} \cap F$ is strictly increasing and is contained in $s_1$, therefore it is bounded in $F$ (see Figure \ref{f.Proof44}). This would imply that there  exists $s$ a first point of $F$ that is not contained in any of the $R^{(i)}$. Suppose that the $R^{(i)} \cap F$ approach $s$ from the left and consider $[t,s]$ a small stable segment in $F$ on the left of $s$. By the finite return time axiom, there exists $R\in \mathcal{R}$ containing $[t,s]$. Notice that since $F$ does not belong to the stable leaf of a boundary periodic orbit, $[t,s]-\lbrace s \rbrace \subset \int{R}$.

On the other hand, for every $i$, $R^{(i+1)}\cap R^{(i)}$ is a non-trivial horizontal subrectangle of $R^{(i)}$, therefore by Lemma \ref{l.infiniteintersectionhorizontalrectangles}, for $n$ sufficiently big, $R^{(n)}$ will be a very thin along the vertical direction and a very long along the horizontal direction rectangle approaching $s$ from the left. Hence, for $n$ big $\overset{\circ}{R^{(n)}}\cap \overset{\circ}{R} \neq \emptyset$, which contradicts the Markovian intersection axiom and leads to an absurd. 

\begin{figure}[h!]
\includegraphics[scale=0.5]{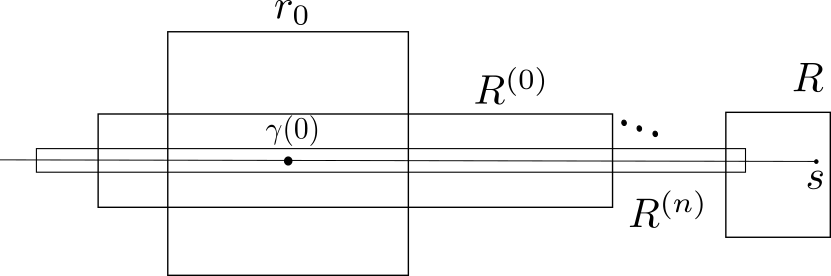}
\caption{}
\label{f.Proof44}
\end{figure}

Therefore, the algorithm for $s_1$ that was described previously eventually stops. In other words, there exists $n$ in our previous sequence such that $s_1$ doesn't exit $R^{(n)}$. At this point, we move on to $u_1$ and we apply the same algorithm to extend our sequence of rectangles, until we reach $u_n$. At the end of this procedure, we obtain a rectangle path starting from $r_0$ and  canonically associated to the good polygonal curve $\gamma$. 


\end{proof}
\begin{coro}\label{c.rectpathsstartingending}
For every two rectangles $R_0, R\in \mathcal{R}$ there exists a rectangle path starting from $R_0$ and ending at $R$. 
\end{coro}
\begin{proof}
We consider a smooth curve $\gamma: [0,1] \rightarrow \mathcal{P}$ such that $\gamma(0)\in \overset{\circ}{R_0}$, $\gamma(1)\in \overset{\circ}{R}$ and $\gamma(0),\gamma(1)$ don't belong to the stable or unstable leaf of a boundary periodic point. We can deform $\gamma$ to a good polygonal curve $\gamma'$ by Lemma \ref{l.createpolygonalcurve}. By Proposition \ref{p.curvetorectanglepath}, $\gamma'$ is associated to a canonical rectangle path of the form $R_0,R_1,...R_n$, with  $R_n$ not necessarily equal to $R$. However, since $\gamma(1)$ does not belong to the stable or unstable leaf of a boundary periodic point, $\gamma(1)\in \overset{\circ}{R}\cap \int{R_n}$ and by Lemma \ref{l.npredecessor} $R_n$ is a $k$-th successor or predecessor of $R$ for some $k \in \mathbb{N}$. Assume without any loss of generality that $R_n$ is a $k$-th predecessor of $R$. In this case, since the successors of a rectangle have disjoint interiors,  there exists a unique sequence of rectangles $R_n,...,R_{n+k}=R$ such that  $R_{l+1}$ is a successor of $R_l$ for every $l\in \llbracket n, n+k-1 \rrbracket$. The rectangle path $R_0,...,R_{n+k}=R$ is the desired rectangle path.
\end{proof}

Of course the rectangle path of the previous corollary is not unique; it depends on the choice of the original curve $\gamma$ and how we deformed it into a good polygonal curve $\gamma'$. In fact, for any rectangle $R\in \mathcal{R}$ there exist infinitely many distinct rectangle paths going from $R_0$ to $R$: if $R_0,R_1,...,R$ is one such rectangle path, then $R_0,R_1,R_0,R_1,...,R$ is also a rectangle path from $R_0$ to $R$. 

According to Proposition \ref{p.curvetorectanglepath}, it is possible to associate a rectangle path to any good polygonal curve in $\mathcal{P}$. The inverse is also possible:  
\begin{prop}\label{p.rectanglepathtocurve}
For every rectangle path $R_0,...,R_n$ there exists a good polygonal curve $\gamma$ associated to $R_0,...,R_n$ by Proposition \ref{p.curvetorectanglepath}. Furthermore, if $R_0,...,R_n$ is closed we can choose $\gamma$ to be good and closed. 
\end{prop}
\begin{proof}
Consider $R_0,...,R_n$ a rectangle path and $x\in \int{R_0}$ that doesn't belong to the stable or unstable leaf of a boundary periodic point. Consider now a good polygonal path starting from $x\in R_0$, exiting $R_0$ in order to enter $R_1$, then exiting $R_1$ in order to enter $R_2$ and so on until we reach $R_n$. If $R_n=R_0$ we can furthermore ask that the path ends at $x$. The above curve corresponds to the desired good polygonal curve. 
\end{proof}

\begin{rema}\label{r.polygonalcurverectassociation}
\begin{itemize}
    \item The construction of $\gamma$ in Proposition  \ref{p.rectanglepathtocurve} is not unique. 
    \item If $\gamma:[0,1]\rightarrow \mathcal{P}$ is associated to the rectangle path $r_0,...,r_n$, then there exist $0=c_0<c_1<...<c_{n+1}=1$ such that $\gamma([c_i,c_{i+1}])\subset r_i$. The $c_i$ are not unique. Using the above $c_i$, we can define a function $\textit{Rect}_{\gamma, r_0}:[0,1]\rightarrow \left\{ r_0,...,r_n \right\}$ sending every interval of the form $[c_i,c_{i+1})$ to $r_i$ and with $\textit{Rect}_{\gamma, r_0}(1)=r_n$. The function $\textit{Rect}_{\gamma, r_0}$ associates every point of $\gamma$ to a rectangle and every segment in $[0,1]$ to a rectangle path.
 \end{itemize}
 \end{rema}
 \begin{rema}\label{r.choiceci}
We can easily deduce from our proof of Proposition \ref{p.curvetorectanglepath} that if the good polygonal curve $\gamma:[0,1]\rightarrow \mathcal{P}$ is associated to the rectangle path $r_0,...,r_n$, we can always choose $0=c_0<c_1<...<c_{n+1}=1$ and define $\textit{Rect}_{\gamma, r_0}$ so that for any stable (resp. unstable) segment $S$ of $\gamma$ the rectangle path $\textit{Rect}_{\gamma,r_0}(\gamma^{-1}(S))$ is trivial (i.e. of length one) or of the form $r_k,...,r_{k+l}$, where for every $i\in \llbracket k,k+l-1\rrbracket $ $r_{i+1}$ is a successor (resp. predecessor) of $r_{i}$. Furthermore, for this choice of $c_i$, notice that the rectangle path starting from $r_0$ associated to $\gamma_{|[c_0,c_k]}$ is precisely $\textit{Rect}_{\gamma,r_0}([c_0,c_k])$.

\end{rema}
\subsection{Singularities of polygonal curves}\label{s.singularities}
Rectangle paths and good polygonal curves, as it was proved in the previous section, are two objects that are very closely related. Understanding the behaviour of a polygonal curve will give us information about its associated rectangle path and vice versa. In this section, we would like to define the notion of tangency for a good polygonal curve and prove a combinatorial result linking the different types of tangencies of a good polygonal curve. We fix once again $(M,\Phi)$ a transitive Anosov flow, $\mathcal{P}$ its bifoliated plane endowed with an orientation, $\mathcal{F}^{s,u}$ the stable/unstable foliations in $\mathcal{P}$, $\mathcal{R}$ a Markovian family on $\mathcal{P}$ and $\Gamma$ the boundary periodic orbits of $\mathcal{R}$. 

\begin{defi}
A closed poygonal curve $\gamma:[0,1]\rightarrow \mathcal{P}$ will be called \emph{simple} if the function $\gamma$ restricted on $[0,1)$ is injective. 

By Jordan's theorem, a simple closed polygonal curve $\gamma$ defines in the plane $\mathcal{P}$ two complementary regions: a bounded region that we will name the \emph{interior} of $\gamma$ and an unbounded region that we will name the \emph{exterior} of $\gamma$

\end{defi}

Take $\gamma$ a simple closed polygonal curve in $\mathcal{P}$, $D$ the interior of $\gamma$ and $U$ a neighbourhood of $D\cup \gamma$. Since $\mathcal{P}$ is simply connected, the foliations $\mathcal{F}^{s,u}$ are orientable and transversally orientable (even if the flow has not transversally orientable folations in $M$). Therefore,  we can endow the foliations $\mathcal{F}^{s,u}$ with an orientation and also there exists a continuous non-singular vector field $X$ in $U$ such that for every point $x\in U$, $X(x)$ is tangent to $\mathcal{F}^s(x)$. 

\begin{defi}
We can define 4 types of tangencies to the vector field $X$ for the polygonal curve $\gamma$, that we will call respectively \emph{tangencies of type (a), (b), (c) and (d)} (see Figure \ref{f.singularities}). More specifically, if we consider  $\gamma$ as a function from $\mathbb{S}^1$ (endowed with an orientation) to $\mathcal{P}$, then to every stable segment $s$ in $\gamma$ corresponds an  unstable segment $u^{-}$ coming before it and an unstable segment $u^+$ coming after it. By using the orientation on $\mathbb{S}^1$, we can naturally orient $u^{-}$ and $u^+$. We will say that $s$ is a (stable) tangency of $\gamma$ if one of the two previous segments is negatively oriented and the other is positively oriented with respect to the orientation of $\mathcal{F}^{u}$. 

Furthermore, we will say that a tangency $s$ of $\gamma$ is a tangency of type (a) (resp. (d)) if for any small positive unstable segment $u$ starting from $s$ we have that $u\subset D$ and for  any stable leaf of $\mathcal{F}^{s}$ intersecting $u$ intersects also (resp. does not intersect) $u^{-}$ and $u^+$. 

By changing the positive segment $u$ to a negative segment, we can similarly define a tangency of type (b) or (c).
\begin{figure}[h]
  \begin{minipage}[ht]{0.2\textwidth}
    \centering    
    \includegraphics[width=0.7\textwidth]{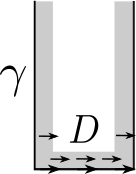}
    \caption*{\quad (a)}
    
  \end{minipage}
 \begin{minipage}[ht]{0.2\textwidth}
 \centering
    \includegraphics[width=0.7\textwidth]{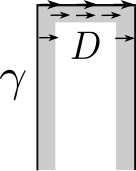}
    \caption*{\quad (b)}
    
  \end{minipage}
  \begin{minipage}[ht]{0.2\textwidth}
  \centering
    \includegraphics[width=0.7\textwidth]{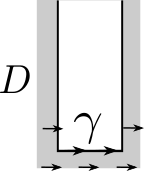}
    \caption*{\quad (c)}
    
  \end{minipage}
  \begin{minipage}[ht]{0.2\textwidth}
  \centering
    \includegraphics[width=0.7\textwidth]{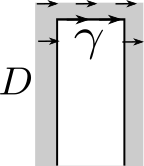}
    \caption*{\quad (d)}
    
  \end{minipage}
  \caption{Types of singularities for $\gamma$}
  \label{f.singularities}
\end{figure}

\end{defi}

Perturb $\clos{D}$ a little bit in $U$ so that its boundary becomes smooth and is in general position with respect to $X$. Let us name $D'$ this disc. $\partial D'$ has a finite number of tangencies with $X$ (i.e. points of $\partial D'$ whose tangent direction is collinear with $X$).   

Take $x$ a point of tangency between $\partial D'$ and $X$. We will say that $x$ is an \textit{interior} (resp. \textit{exterior}) \textit{tangency } if $X(x)$ points towards the interior (resp. exterior) of $D'$. By the Poicar\'e Index theorem for the disc $D'$, since $X$ has no singularities, we have that $$(\text{number of exterior tangencies}) - (\text{number of interior  tangencies})=2 $$ 
Therefore, there are always more exterior tangencies than interior ones. 

\begin{rema}
The definition of tangency of type (a) (resp. (b), (c), (d)) or the definition of exterior and interior tangencies for $X$ does not depend on the choice of the non-singular vector field $X$ in $U$. From now on, we will therefore speak of tangencies without necessarily specifying our choice of vector field $X$.  
\end{rema}

It is possible to perturb $\clos{D}$ so that every stable tangency of type (a) or (b) (resp. (c) or (d)) for $\gamma$ becomes an exterior (resp. interior) tangency for $\partial D'$ and all the tangencies of $\partial D'$ are obtained in this way; in other words every exterior or interior tangency of $\partial D'$ corresponds to a stable tangency of type (a), (b), (c) or (d) of  $\gamma$. Indeed, any stable segment of $\gamma$ that is not a tangency can be perturbed into a small segment transverse to  $\mathcal{F}^s$. By the previous arguments, we can assume that the  correspondence between the tangencies of $\partial D'$ and the tangencies of $\gamma$ is a bijection. Therefore we obtain the following result:
\begin{lemm}\label{l.numbertangencies} Consider the stable tangencies of any simple closed polygonal curve $\gamma$ in $\mathcal{P}$. We have that:
$$(\text{number of tangencies of type (a) or (b)}) - (\text{number of tangencies of type (c) or (d)})=2 $$ 
\end{lemm}
\begin{rema}
For any stable tangency of $\gamma$, its type ((a),(b),(c) or (d)) depends on our choice of orientation of $\mathcal{F}^{u}$. However, the previous lemma is true independently from our choice of orientation. 
\end{rema}
Of course by considering a vector field tangent to the unstable foliation, we can similarly define 4 types of unstable tangencies and prove the above lemma for unstable tangencies. Let us finish this section with the following lemma: 
\begin{lemm}\label{l.canonicalneighbourhoodtangency}
Let $\gamma$ be a simple closed polygonal curve, $s$ a stable tangency of type (a) or (b) and  $D$ the interior of $\gamma$. There exists a canonical neighbourhood for $s$ in $D\cup \gamma$ trivially foliated by the stable foliation
\end{lemm}
\begin{proof}
By considering $\gamma$ as a function from $\mathbb{S}^1$ to $\mathcal{P}$, let $u^{-}$ be the unstable segment coming before $s$ and $u^+$ the unstable segment  coming after it. Suppose without any loss of generality that $s$ is a tangency of type (a). For every point $x$ of $u^-$ close to $s$, there exists a stable segment $s_x$ from $x$ to $u^+$ such that $\int{s_x}\subset D$. If such a segment exists for all $x\in u^-$ then the desired neighbourhood will be $\underset{x\in u^-}{\cup}s_x$. If not, repeat the same procedure for $u^+$. 

If for every $x\in u^-$ there exists a stable segment $s_x$ from $x$ to $u^+$ such that $\int{s_x}\subset D$  and for every $x'\in u^+$ there exists a stable segment $s_{x'}$ from $x'$ to $u^-$ such that $\int{s_{x'}}\subset D$, then it is easy to see that $\underset{x\in u^-}{\cup}s_x=\underset{x'\in u^+}{\cup}s_{x'}$. Therefore, the previous construction does not depend on the choice of $u^-$ or $u^+$.

Let us now suppose that there exist
\begin{itemize}
    \item $X^- \in u^-$ for which there is no stable segment $s_X$ going from $X^-$ to $u^+$ and such that $\int{s_X}\subset D$
    \item $X^+ \in u^+$ with the analogous property
\end{itemize}  The set of points with the previous properties is closed, we can therefore suppose that $X^-$ is the first point in $u^-$ with this property starting from the unique point in $s\cap u^-$. Same for $X^+$. This implies that for every point $x_-$ in $u^-$ below $X^-$, there exists a stable segment $s_{x_-}$ from $x_-$ to  $P(x_-)\in u^+$ such that $\int{s_{x_-}}\subset D$. Similarly, for every point $x_+$ in $u^+$ below $X^+$, there exists a stable segment $s_{x_+}$ from $x_+$ to  $P(x_+)\in u^-$ such that $\int{s_{x_+}}\subset D$. 

Furthermore, any sequence $x_n$ in $u^-$ (resp. $u^+$) converging to $X^-$ (resp. $X^+$) from below verifies $P(x_n)\rightarrow X^+$   (resp. $P(x_n)\rightarrow X^-$). Indeed, it suffices to prove this for monotonous sequences $x_n$. If $x_n$ is monotonous, then $P(x_n)\in u^+$ is also monotonous and since $u^+$ is compact, it converges to a point $y\in u^+$. $y$ cannot be strictly above  $X^+$, since arbitrarily close to $y$ we can find a point  $x\in u^+$ and a segment $s_x$ going from $x$ to $u^-$ and such that $\int{s_x}\subset D$. This would imply the existence of a stable segment with analogous properties from $X^+$ to $u^{-}$, which is impossible. Also, $y$ cannot be strictly below $X^+$, because that would imply that there exists a neighbourhood of $X^-$ containing points $x$ and segments $s_x$ going from $x$ to $u^+$ and such $\int{s_x}\subset D$. We deduce that $y=X^+$ and that either $X^-$ and $X^+$ belong to the same stable leaf or to two stable leaves that are not separated in $\mathcal{P}$.

$X^-$ and $X^+$ cannot belong to two non separated stable leaves. Indeed, if this is the case, by Corollary 4.4 of  \cite{Fe}, there exists an unstable leaf $L$, such that for every point $x\in u^-$ below $X^-$ we have that $L\cap s_x \neq \emptyset$. Furthermore, by the same result of \cite{Fe}, if $I^-$ is the set of points in $u^-$ below $X^-$, then $L \cap \underset{x\in I^-}{\cup}s_x$ is an unbounded segment of $L$ contained in $D\cup \gamma$, which is a contradiction. Therefore, $X^-$ and $X^+$ belong to the same stable leaf. 

Finally, since $X^-$ and $X^+$ belong to the same stable leaf, we can consider $s'$ the stable segment from $X^-$ to $X^+$. The desired canonical neighbourhood in this case will be $\underset{x\in I^-}{\cup}s_x \cup s' \subset D\cup \gamma$. 

\end{proof} 
The canonical neighbourhood constructed in the previous lemma for the tangency $s$ has two stable and two unstable boundary components and will be called the \textit{domain} of $s$. Furthermore, if the domain of $s$ covers completely either $u^-$ or $u^+$ (we follow here the notations of the proof of Lemma \ref{l.canonicalneighbourhoodtangency}), we will call it a \textit{complete domain}. In any other case, we will call it an \emph{incomplete domain}.

By the proof of Lemma \ref{l.canonicalneighbourhoodtangency} we can deduce that: 
\begin{rema}\label{r.incompletedomain}
 If the domain $D$ of $s$ is  incomplete, there exists a finite number of stable segments of $\gamma$ in the interior of the stable boundary component of $D$ that is not $s$.  
\end{rema}
\section{Comparing bifoliated planes up to surgeries}\label{s.constructionPbar}
Let $(M_1,\Phi_1)$,  $(M_2,\Phi_2)$ be two transitive Anosov flows,  $\mathcal{P}_1$,  $\mathcal{P}_2$ their bifoliated planes endowed with an orientation, $\mathcal{R}_1$, $\mathcal{R}_2$ two Markovian families in $\mathcal{P}_1$ and $\mathcal{P}_2$ respectively and $\Gamma_1$,  $\Gamma_2$ their boundary periodic orbits. Assume that  $\mathcal{R}_1$ and $\mathcal{R}_2$ are associated by Theorem \ref{t.associatemarkovfamiliestogeometrictype} to the same class of geometric types. We would like to prove that the two previous Anosov flows are orbitally equivalent up to surgeries along boundary periodic orbits. In this section we will develop tools that will allow us to compare flows up to surgeries. 

More specifically, instead of comparing the bifoliated planes $\mathcal{P}_1$ and $\mathcal{P}_2$ that depend greatly on the choice of surgery on the boundary periodic orbits (see \cite{BonattiIakovoglou}), we will compare the universal covers of $\mathcal{P}_1-\Gamma_1$ and $\mathcal{P}_2-\Gamma_2$, denoted by $\widetilde{\mathcal{P}_1}$ and $\widetilde{\mathcal{P}_2}$ that have the advantage of being independent of the choice of surgery on the boundary periodic orbits. The lifts of the rectangles in $\mathcal{R}_1$, $\mathcal{R}_2$ on  $\widetilde{\mathcal{P}_{1}}$, $\widetilde{\mathcal{P}_{2}}$ do not always correspond to  rectangles in the sense of Definition \ref{d.rectanglesinplane}; some rectangles lift to ``rectangles" with corners at infinity and thus cease being compact. In order to avoid this difficulty, by adding some points at infinity, in Sections \ref{s.constructionPbar12} and \ref{s.relationbetweenptildepbar} we will complete  $\widetilde{\mathcal{P}_{1}}$, $\widetilde{\mathcal{P}_{2}}$ to ``branched" cover spaces of $\mathcal{P}_{1}$, $\mathcal{P}_{2}$ that we will denote by $\clos{\mathcal{P}_{1}}$, $\clos{\mathcal{P}_{2}}$. The ``ramification points" of these covers will correspond to $\Gamma_{1}$, $\Gamma_{2}$ and their ``ramification indexes" will be infinite. The spaces $\clos{\mathcal{P}_{1}}$, $\clos{\mathcal{P}_{2}}$ will be called the \emph{bifoliated planes of $\Phi_1,\Phi_2$ up to surgeries along $\Gamma_1,\Gamma_2$}. In Sections  \ref{s.ptildebifoliated} and \ref{s.pbarbifoliated} we will establish that the bifoliated plane up to surgeries has many points in common with the bifoliated plane of an Anosov flow: it is endowed with two (singular) transverse foliations $\clos{\mathcal{F}^{s,u}}$ and one transitive group action that preserves $\clos{\mathcal{F}^{s,u}}$. Our main goal in this setion consists in proving Theorem C according to which $\clos{\mathcal{P}_{1}}$, $\clos{\mathcal{P}_{2}}$ describe the flows $\Phi_1,\Phi_2$ up to surgeries along $\Gamma_1$ and $\Gamma_2$.

\subsection{The construction of $\clos{\mathcal{P}_{1}}, \clos{\mathcal{P}_{2}}$} \label{s.constructionPbar12} 
For this first part of the section, we will fix $(M,\Phi)$ a transitive Anosov flow, $\mathcal{P}$ its bifoliated plane endowed with an orientation, $\mathcal{R}$ a Markovian family on $\mathcal{P}$ and $\Gamma \subset \mathcal{P}$ the set of boundary points of $\mathcal{R}$. We will denote by $\widetilde{\Gamma}$ the lift of $\Gamma$ on $\mathbb{R}^3$ and by $\Gamma_M$ the projection of $\widetilde{\Gamma}$ on $M$. . 

The construction of  $\clos{\mathcal{P}}$ is very similar to the construction of the universal cover $\widetilde{\mathcal{P}}$ of $\mathcal{P}-\Gamma$. Let us first recall the latter one. 
Fix $x_0\in \mathcal{P}-\Gamma$ and $\text{curv}_1= \lbrace \gamma:[0,1]\overset{C^0}{\rightarrow} \mathcal{P} | \gamma(0)=x_0, \gamma[0,1]\cap \Gamma= \emptyset \rbrace$.  For any  $\gamma_1,\gamma_2\in \text{curv}_1$, we will say that $\gamma_1 \sim_1 \gamma_2$ if there exists $H:[0,1]^2\overset{C^0}{\rightarrow} \mathcal{P}$ such that $H(0,\cdot)=\gamma_1$, $H(1,\cdot)=\gamma_2$, $H(\cdot,1)$ is constant and for every $t\in [0,1]$ we have $H(t,\cdot)\in  \text{curv}_1$. We have that $$ \widetilde{\mathcal{P}} = \text{curv}_1 \big{/}\sim_1 $$


Any point $\tilde{x}\in \widetilde{\mathcal{P}}$ corresponds to a union of arcs of $\text{curv}_1$ starting from $x_0$ and ending at the same point of $\mathcal{P}-\Gamma$, say $x$. We define the projection $\tilde{\pi}$ from $\widetilde{\mathcal{P}}$ on $\mathcal{P}-\Gamma$ as the function associating $\tilde{x}$ to $x$. It is well known that $\tilde{\pi}$ is continuous.

Now let us define the space $\clos{\mathcal{P}}$. Take $\text{curv}_2= \lbrace \gamma:[0,1]\overset{C^0}{\rightarrow} \mathcal{P} | \gamma(0)=x_0, \gamma[0,1)\cap \Gamma= \emptyset \rbrace $. For any $\gamma_1,\gamma_2\in \text{curv}_2$, we will say that $\gamma_1 \sim_2 \gamma_2$  if there exists $H:[0,1]^2\overset{C^0}{\rightarrow} \mathcal{P}_1$ such that $H(0,\cdot)=\gamma_1$, $H(1,\cdot)=\gamma_2$, $H(\cdot,1)$ is constant and for every $t\in [0,1]$ we have $H(t,\cdot)\in  \text{curv}_2$. We define
$$ \clos{\mathcal{P}} := \text{curv}_2 \big{/}\sim_2 $$

Any point $\clos{x}\in \clos{\mathcal{P}}$ corresponds to a union of arcs of $\text{curv}_2$ starting from $x_0$ and ending at the same point of $\mathcal{P}$, say $x$. We define the projection $\clos{\pi}$ from $\clos{\mathcal{P}}$ on $\mathcal{P}$ as the function associating $\clos{x}$ to $x$. The function $\clos{\pi}$ is clearly a surjection.

In order to define the topology of $\clos{\mathcal{P}}$, let us construct $\clos{\mathcal{P}}$ in a different way. Blow-up every point in $\Gamma\subset \mathcal{P}$ to a circle and denote this new space by $\mathcal{P}_{\Gamma, blowup}$. The space $\mathcal{P}_{\Gamma, blowup}$ is homeomorphic to a plane minus countably many open disks, forming a discrete set. Its universal cover $\widetilde{\mathcal{P}_{\Gamma, blowup}}$ is a plane with countably many line boundaries. By identifying each of those lines to a point, we obtain the space $\clos{\mathcal{P}}$. We endow $\clos{\mathcal{P}}$ with the quotient topology for this projection.

\subsection{The relation between $\clos{\mathcal{P}}$ and $\widetilde{\mathcal{P}}$}
\label{s.relationbetweenptildepbar}
As an immediate result of our second construction of  $\clos{\mathcal{P}}$ and the fact that $\widetilde{\mathcal{P}}$ embeds continuously in $\widetilde{\mathcal{P}_{\Gamma, blowup}}$ we get the following proposition: 

\begin{prop}\label{p.densityoftildeinsidebar}
There exists a continuous embedding $\phi$ of $ \widetilde{\mathcal{P}}$ in $\clos{\mathcal{P}}$. Furthermore, $\clos{\mathcal{P}}- \phi(\widetilde{\mathcal{P}})$ is a discrete (countable) set of points 
\end{prop}
Using the embedding $\phi$, by a small abuse of language, we will assume from now on that $\widetilde{\mathcal{P}}\subset \clos{\mathcal{P}}$. 

\begin{figure}[h]

  \begin{minipage}[ht]{0.4\textwidth}
    \centering 
     \vspace{0.8cm}
    \includegraphics[width=0.4\textwidth]{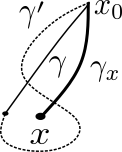}
  \hspace{-1cm}
    \caption*{(a)}
    
  \end{minipage}
 \begin{minipage}[ht]{0.4\textwidth}
 \centering
    \includegraphics[width=0.8\textwidth]{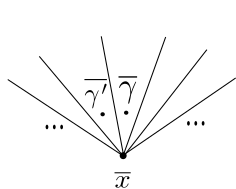}
    \hspace{-1cm}
    \caption*{\quad \quad (b)}
    
  \end{minipage}
  
  \caption{(a) $\clos{\gamma},\clos{\gamma'}$ belong in the neighbourhood of $\overline{\gamma_x}$ ($\clos{\gamma}\neq\clos{\gamma'}$ when $x\in \Gamma$) (b) We can picture $\overline{\mathcal{P}}$ as a branched cover over $\mathcal{P}$, where the points of $\Gamma$ correspond to points with infinite branching index }
  \label{f.conicpoint}
\end{figure}

Notice that by our first construction of $\clos{\mathcal{P}}$, we have $\clos{\mathcal{P}}-\widetilde{\mathcal{P}}=\clos{\pi}^{-1}(\Gamma)$. Furthermore, going around a point in $\Gamma$, changes  the class in $\sim_2$ of a curve whose endpoints are not in $\Gamma$ (see Figure \ref{f.conicpoint}a). However, doing so for a curve going from $x_0$ to a point in $\Gamma$ does not change the class of the curve in $\sim_2$. This  justifies the fact that we can visualize $\clos{\mathcal{P}}$ as a branched cover of $\mathcal{P}$, whose branching points have infinite index and are exactly $\clos{\Gamma}: \clos{\pi}^{-1}(\Gamma)$ (see Figure \ref{f.conicpoint}b). We would also like to point out at this point that:

\begin{rema} 
The projection $\clos{\pi}: \clos{\mathcal{P}}\rightarrow \mathcal{P}$ is continuous. 
\end{rema}
Indeed, by our first definition of $\clos{\mathcal{P}}$, we have that $\clos{\pi}_{|\widetilde{\mathcal{P}}}\equiv \tilde{\pi}$. Therefore, $\clos{\pi}$ is continuous on $\widetilde{\mathcal{P}}$. Consider now, any open set of $\mathcal{P}$ around a unique point in $\Gamma$. The previous open can be naturally associated with an open set in $\mathcal{P}_{\Gamma, blowup}$ and lifts to an open neighbourhood of some boundary line in $\widetilde{\mathcal{P}_{\Gamma, blowup}}$. By definition of the quotient topology, the previous open set projects to an open  neighbourhood of some point in $\clos{\Gamma}$, which gives us the desired result.

\subsection{$\widetilde{\mathcal{P}}$ as the bifoliated plane of a flow}\label{s.ptildebifoliated}
The space $\widetilde{\mathcal{P}}$ is the universal cover of a plane minus a countable number of points forming a discrete set; it is therefore a plane. Furthermore, by lifting on $\widetilde{\mathcal{P}}$ the stable and unstable foliations of $\mathcal{P}-\Gamma$, we obtain two transverse foliations $\widetilde{\mathcal{F}^s}$ and $\widetilde{\mathcal{F}^u}$ that we will respectively call stable and unstable foliations in $\widetilde{\mathcal{P}}$. Notice that for any $x\in \Gamma$ any of the connected components of $\mathcal{F}^s(x)-\lbrace x\rbrace$ (resp. $\mathcal{F}^u(x)-\lbrace x\rbrace$) lifts to a leaf of $\widetilde{\mathcal{F}^s}$ (resp. $\widetilde{\mathcal{F}^u}$). Notice also that, as for $\mathcal{F}^s$ and $\mathcal{F}^u$, any two leaves $s\in \widetilde{\mathcal{F}^s}$ and $u \in \widetilde{\mathcal{F}^s}$, intersect at most at one point.

Not only is the space $\widetilde{\mathcal{P}}$ endowed with two transverse foliations, but also a natural group action: 

\begin{prop}\label{p.orbitspacetilde}
The space $\widetilde{\mathcal{P}}$ is the orbit space of the lift on $\mathbb{R}^3$ of the flow $(\Phi,M)$ minus the boundary periodic orbits $\Gamma_M$. Consequently, $\widetilde{\mathcal{P}}$ is naturally endowed with a faithful action of $\pi_1(M-\Gamma_M)$. 
\end{prop}
\begin{proof}
Let us denote by $\widetilde{\Phi}$ the lift of $\Phi$ on $\widetilde{M}=\mathbb{R}^3$. Since $\widetilde{M}-\widetilde{\Gamma} $ is a covering space of $M-\Gamma_M$, the universal cover of $M-\Gamma_M$, denoted by $\widetilde{M-\Gamma_M}$, coincides with  the universal cover of $\widetilde{M}-\widetilde{\Gamma} $. Moreover, the universal cover of $\mathbb{R}^3$ minus a countable and transversally discrete set of lines is homeomorphic to $\mathbb{R}^3$, therefore  $\widetilde{M-\Gamma_M}$ is homeomorphic to $\mathbb{R}^3$. Similarly, the flow $\Phi$ minus $\Gamma_M$ lifted on $\widetilde{M-\Gamma_M}$ can be identified with the lift of $\widetilde{\Phi}$ minus $\widetilde{\Gamma}$ on the universal cover of  $\widetilde{M}-\widetilde{\Gamma}$. We will denote this flow by $\widetilde{\Phi_{\Gamma_M}}$. 

Recall that $\widetilde{\Phi}$ is orbitally equivalent to the constant vertical flow on $\mathbb{R}^3$. Take $P$ a topological plane intersecting once every orbit of $\widetilde{\Phi}$. $P$ can be identified with the bifoliated plane $\mathcal{P}$ of $\Phi$. Consider $\widetilde{P}$ the lift of $P-\widetilde{\Gamma} \simeq \mathcal{P}-\Gamma$ on the universal cover of  $\widetilde{M}-\widetilde{\Gamma} $. $\widetilde{P}$ is a topological plane, where the punctures of $P$ have now become ``points at infinity".

From one side, the plane $\widetilde{P}$ intersects once every orbit of $\widetilde{\Phi_{\Gamma_M}}$; it can therefore be identified with the bifoliated plane of $\Phi-\Gamma_M$. From the other side, $\widetilde{P}$ corresponds to the universal cover of $P-\widetilde{\Gamma} \simeq \mathcal{P}-\Gamma$, which gives us the desired result.

\end{proof}

Using the above proposition, it is possible to show, in the same exact way as for the bifoliated plane of an Anosov flow (see for instance \cite{Barbotthese}) that 

\begin{itemize}
\item the bifoliated plane of $\Phi$ minus $\Gamma_M$ can be given a structure of $C^1$ manifold.
\item $\pi_1(M-\Gamma_M)$ acts on $\widetilde{\mathcal{P}}$ by preserving the foliations $\widetilde{\mathcal{F}^s}$ and $\widetilde{\mathcal{F}^u}$. The previous action is closely related to the action of $\pi_1(M)$ on $\mathcal{P}$. Denote by $\tilde{\phi}$ the morphism $\pi_1(M-\Gamma_M) \rightarrow \text{Hom}(\widetilde{\mathcal{P}})$, by $\phi$ the morphism $\pi_1(M) \rightarrow \text{Hom}(\mathcal{P})$ and by $\rho$ the natural morphism $\pi_1(M-\Gamma_M) \rightarrow \pi_1(M)$. By  quotienting $\widetilde{\mathcal{P}}$ by the action of $\text{ker}(\rho)$, we obtain $\mathcal{P}-\Gamma$. Furthermore, by the constructions of the actions $\phi, \tilde{\phi}$, for every element $g\in \pi_1(M-\Gamma_M)$ we have that $$\tilde{\pi}\circ \tilde{\phi}(g) = \phi(\rho(g)) \circ \tilde{\pi}$$
\item the orbit by $\pi_1(M-\Gamma_M)$ of any stable or unstable leaf in $\widetilde{\mathcal{F}^{s,u}}$ is dense in $\widetilde{\mathcal{P}}$
\item if $g\in \pi_1(M-\Gamma_M)$ acts trivially on a stable or unstable leaf then $g=id$
\item for any point $x \in \widetilde{\mathcal{P}}$ we have that $\text{Stab}(x) = \mathbb{Z}$ if and only if $x$ is periodic and $\text{Stab}(x) =\lbrace \text{id} \rbrace$ in all the other cases
\item for any stable/unstable leaf $f$ in $\widetilde{\mathcal{F}^{s,u}}$ we have that $\text{Stab}(f) = \mathbb{Z}$ or $\text{Stab}(f) =\lbrace \text{id} \rbrace$. More specifically,  $\text{Stab}(f) = \mathbb{Z}$ if and only if $f$ contains a periodic point or $f$ projects to a stable/unstable separatrix of a point in $\Gamma$
\end{itemize}
A major difference between the bifoliated plane of an Anosov flow and $\widetilde{\mathcal{P}}$ resides in the following remark:
\begin{rema}
If $g\in \pi_1(M-\Gamma_M)$ preserves a leaf $f\in \widetilde{\mathcal{F}^s}$, then $g$ fixes a point in $f$ if and only if $f$ does not project to a stable separatrix of some point in $\Gamma$.
\end{rema}

\subsection{Extending the structure of bifoliated plane to $\clos{\mathcal{P}}$}\label{s.pbarbifoliated}
Even though $\clos{\mathcal{P}}$ is not the bifoliated plane of some Anosov flow, it shares many properties with both $\mathcal{P}$ and $\widetilde{\mathcal{P}}$. 
\begin{prop}\label{c.extensionhomeomorphisms}
The action of $\pi_1(M-\Gamma_M)$ on $\widetilde{\mathcal{P}}$ can be extended continuously to a faithful action on $\clos{\mathcal{P}}$.
\end{prop}
\begin{proof}
As we remarked in Section \ref{s.relationbetweenptildepbar}, $\widetilde{\mathcal{P}}$ embeds continuously in $\widetilde{\mathcal{P}_{\Gamma,blowup}}$. More specifically, the image of this embedding is exactly the interior of $\widetilde{\mathcal{P}_{\Gamma,blowup}}$. 

Recall that $\mathcal{P}$ can be given a structure of $C^1$ manifold, where the action by $\pi_1(M)$ can be seen as an action by diffeomorphisms (see for instance \cite{Barbotthese}). Consider now the diffeomorphism $g\in \pi_1(M)$ acting on $\mathcal{P}$. Using the action of the differential of $g$ on the tangent bundle of $\mathcal{P}$, $g$ corresponds to a unique homeomoprhism on $\mathcal{P}_{\Gamma,blowup}$. Take $\widetilde{g}$ a lift of this homeomorphism on $\widetilde{\mathcal{P}_{\Gamma,blowup}}$. Using the relation between the actions  of $\pi_1(M)$ on $\mathcal{P}$ and of $\pi_1(M-\Gamma_M)$ on $\widetilde{\mathcal{P}}$ (see Section \ref{s.relationbetweenptildepbar}), we get that $\widetilde{g}_{|\widetilde{\mathcal{P}}}$ corresponds to a unique element of $\pi_1(M-\Gamma_M)$. By a density argument, we can extend the action of $\pi_1(M-\Gamma_M)$ on $\widetilde{\mathcal{P}}$ to an action of  $\pi_1(M-\Gamma_M)$ on $\widetilde{\mathcal{P}_{\Gamma,blowup}}$. The previous action projects to a $C^0$ action of $\pi_1(M-\Gamma_M)$ on  $\clos{\mathcal{P}}$, which gives us the desired result.

\end{proof}

Furthermore, 
\begin{rema} \label{r.projectequiv} Any two points in the same $\pi_1(M-\Gamma_M)$-orbit in $\clos{\mathcal{P}}$ project to two points in the same $\pi_1(M)$-orbit in $\mathcal{P}$ and vice-versa. 
\end{rema}
Indeed, the previous remark is true for any two points in $\widetilde{\mathcal{P}}$ that are in the same $\pi_1(M-\Gamma_M)$-orbit. We obtain the desired result by using the density of $\widetilde{\mathcal{P}}$ in $\clos{\mathcal{P}}$.

 Moreover, it is possible to endow $\clos{\mathcal{P}}$ with two (singular) foliations $\clos{\mathcal{F}^{s,u}}$ by extending the foliations $\widetilde{\mathcal{F}^{s,u}}$. More precisely, $F$ will be a leaf of $\clos{\mathcal{F}^{s,u}}$ if and only if
 \begin{enumerate}
     \item $F$ projects on $\mathcal{P}$ to a stable/unstable leaf in $\mathcal{F}^{s,u}$ disjoint from $\Gamma$, or 
     \item $F$ projects on $\mathcal{P}$ to a stable/unstable separatrix of a point in $\Gamma$
     
 \end{enumerate}
 The two foliations $\clos{\mathcal{F}^{s}}$ and $\clos{\mathcal{F}^{u}}$ are transverse everywhere, except at their singularities, namely the points of $\clos{\Gamma}$. More particularly, we have

\begin{prop}\label{p.singularitiesoffoliations}
For every $\clos{\gamma}\in \clos{\Gamma}$, the stable leaves in $\clos{\mathcal{F}^{s}}$ that intersect $\clos{\gamma}$  form a countable set of leaves $...,s_{-2},s_{-1},s_0,s_1,s_2...$ ordered along $\mathbb{Z}$ and  satisfying the following:
\begin{itemize}
\item for all $k, l\in \mathbb{Z}$  such that $(k-l)$ is even there exists $g_{k,l}\in \pi_1(M-\Gamma_M)$ such that $g_{k,l}(s_k)=s_l$. If furthermore $\clos{\gamma}$ projects on $M$ to a periodic orbit with negative eigenvalues the previous stands for all $k, l\in \mathbb{Z}$
\item for all $k\in \mathbb{Z}$ $s_k\cap \clos{\Gamma}= \lbrace \clos{\gamma} \rbrace$
\item for all $k,m\in \mathbb{Z}$ $s_k$ is not separated from $s_{m}$ in $\clos{\mathcal{P}}$ (i.e. $s_k-\lbrace \clos{\gamma} \rbrace$ and  $s_{m}-\lbrace \clos{\gamma} \rbrace$ are non-separated stable leaves in $\widetilde{\mathcal{P}}$ ) if and only if $|k-m|=1$
\end{itemize}
\end{prop} 
\begin{proof}
The second point of the proposition results immediately from the fact that $\clos{\pi}(s_k)$ is a stable separatrix of $\gamma:=\clos{\pi}(\gamma)\in \Gamma$. Therefore, it contains a unique periodic point:  $\clos{\pi}(s_k)\cap\Gamma=\{\gamma\}$. 

The first and third point of the proposition are immediate consequences of our second  construction of $\clos{\mathcal{P}}$ and Remark \ref{r.projectequiv} (see Figure \ref{f.pbarform}). 

\begin{figure}
    \centering
    \includegraphics[scale=0.5]{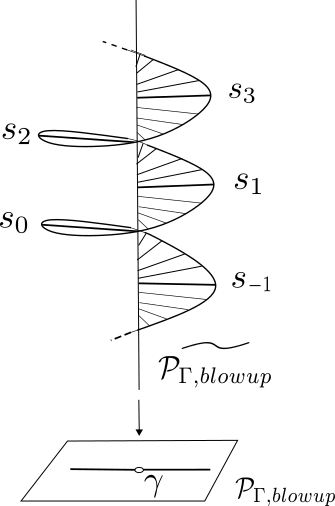}
   \caption{}
    \label{f.pbarform}
\end{figure}
\end{proof}

The space $\clos{\mathcal{P}}$ endowed with the action of $\pi_1(M-\Gamma_M)$ together with its two foliations $\clos{\mathcal{F}^{s,u}}$ resembles closely a bifoliated plane of some Anosov flow. By our previous discussion: 

\begin{itemize}
\item The action of $\pi_1(M-\Gamma_M)$ preserves the singular foliations $\clos{\mathcal{F}^s}$ and $\clos{\mathcal{F}^u}$. The previous action is closely related to the action of $\pi_1(M)$ on $\mathcal{P}$. Denote by $\clos{\phi}$ the morphism $\pi_1(M-\Gamma_M) \rightarrow \text{Hom}(\clos{\mathcal{P}})$, by $\phi$ the morphism $\pi_1(M) \rightarrow \text{Hom}(\mathcal{P})$ and by $\rho$ the natural morphism $\pi_1(M-\Gamma_M) \rightarrow \pi_1(M)$. By quotienting $\clos{\mathcal{P}}$ by the action of $\text{ker}(\rho)$, we obtain the bifoliated plane $\mathcal{P}$ of $\Phi$. Furthermore, by the constructions of the actions $\phi, \clos{\phi}$, for every element $g\in \pi_1(M-\Gamma_M)$ we have that $$\clos{\pi}\circ \clos{\phi}(g) = \phi(\rho(g)) \circ \clos{\pi}$$
\item Any two leaves $s\in \clos{\mathcal{F}^s}$ and $u \in \clos{\mathcal{F}^u}$  intersects at most at one point
\item For any point $x \in \clos{\mathcal{P}}-\clos{\Gamma}=\widetilde{\mathcal{P}}$ we have that $\text{Stab}(x) = \mathbb{Z}$ if and only if $x$ is periodic (i.e. corresponds to a periodic orbit in $M$)  and $\text{Stab}(x) = \lbrace \text{id} \rbrace$ if and only if $x$ is not periodic. 
\item If $g\in \pi_1(M-\Gamma_M)$ preserves a leaf $f\in \clos{\mathcal{F}^s}$, then $f$ carries a periodic point
\item If $g\in \pi_1(M-\Gamma_M)$ acts trivially on a stable or unstable leaf then $g=id$
\item The orbit of every leaf in  $\clos{\mathcal{F}^{s,u}}$ by $\pi_1(M-\Gamma_M)$ is dense
\end{itemize}

One major difference between $\clos{\mathcal{P}}$ and the bifoliated plane of some Anosov flow is given by the following proposition: 

\begin{prop}\label{p.stabilizersz2}
The stabilizer in $\pi_1(M-\Gamma_M)$ of any point in $\clos{\Gamma}$ is isomorphic to $\mathbb{Z}^2$
\end{prop}
\begin{proof}
Take $\clos{\gamma}\in \clos{\Gamma}$. By Proposition \ref{p.singularitiesoffoliations}, the set of stable leaves in $\clos{\mathcal{F}^s}$ intersecting  $\clos{\gamma}$ is countable and ordered along $\mathbb{Z}$. Using the orientation of $\widetilde{\mathcal{P}}$, we will denote those leaves by  $...s_{-2},s_{-1},s_0,s_1,s_2...$ following the anti-clockwise direction (see Figure \ref{f.pbarform}). Since the elements of $\pi_1(M-\Gamma_M)$ preserve $\clos{\mathcal{F}^{s,u}}$, then any element in $\text{Stab}(\clos{\gamma})$ permutes the set of leaves $\lbrace s_k|k\in \mathbb{Z}\rbrace$ and even more associates non-separated leaves to non-separated leaves. Therefore, by Proposition \ref{p.singularitiesoffoliations} any element $g\in \text{Stab}(\clos{\gamma})$ acts on the set $\lbrace s_k|k\in \mathbb{Z}\rbrace \equiv \mathbb{Z}$ as the composition of some symmetry on $\mathbb{Z}$ and some translation. But since $\pi_1(M-\Gamma_M)$ acts on $\clos{\mathcal{P}}$ by orientation preserving homeomorphisms, $g$ can  only act as a translation on $\lbrace s_k|k\in \mathbb{Z}\rbrace$. 

By Proposition \ref{p.singularitiesoffoliations}, the smallest possible non-trivial translation on $\lbrace s_k|k\in \mathbb{Z}\rbrace$ is a translation by $\pm 1$ or $\pm 2$. Take $t\in\pi_1(M-\Gamma_M)$ acting as a translation by $+1$ or $+2$ on $\lbrace s_k|k\in \mathbb{Z}\rbrace$, depending on whether $\clos{\gamma}$ has respectfully negative or positive eigenvalues. We will assume that $t\in \text{ker}(\rho)$, where $\rho$ is the natural morphism from $\pi_1(M-\Gamma_M)$ to $\pi_1(M)$. 

Concerning the elements of $\pi_1(M-\Gamma_M)$ that preserve $s_0$ (therefore also every $s_i$), by our discussion at the end of Section \ref{s.ptildebifoliated},  $\text{Stab}(s_0)= \mathbb{Z}$. Take $s$ the  generator of $\text{Stab}(s_0)$ acting on $s_0$ as an expansion. Recall that any element of $\pi_1(M-\Gamma_M)$ that fixes all the points of a stable leaf is the identity. Hence, we have that $\text{Stab}(\clos{\gamma})=<s,t>$. It now suffices to show that $s$ and $t$ commute. 

Indeed,  $\rho(tst^{-1})=\rho(t) \rho(s) \rho(t)^{-1}=\rho(s)$  and by construction $tst^{-1}\in \text{Stab}(s_0)$. By a classical fact in covering space theory, the unique element of $\pi(M-\Gamma_M)$ that fixes $s_0$ and that projects in $\pi(M)$ to $\rho(s)$ is $s$. Therefore, $tst^{-1}=s$ and since $M-\Gamma$ is aspherical, its fundamental group is torsion-free. Therefore,  $<t,s>=\mathbb{Z}^2$.
\end{proof}
Another way of proving that $<s,t>=\mathbb{Z}^2$ is by understanding the homotopy classes in $M$ corresponding to $s$ and $t$. Fix $x_0\in M-\Gamma_M$ and one of its lifts in $\widetilde{M-\Gamma_M}=\mathbb{R}^3$, say $\widetilde{x_0}$. Consider $\clos{\gamma}\in \clos{\Gamma}$, $x\in \clos{\mathcal{F}^s}(\clos{\gamma})$ very close to $\clos{\gamma}$ and $X$ a lift of $x$ in $\widetilde{M-\Gamma_M}$. Take $\delta$ a continuous path in $\widetilde{M-\Gamma_M}$ from $\widetilde{x_0}$ to $X$ and project it to a path $\delta_M$ in $M-\Gamma_M$ going from $x_0$ to $x_M$, the projection of $X$ on $M-\Gamma$.

The point $\clos{\gamma}$ corresponds to a periodic orbit $\gamma$ in $M$. Take $T$ the torus boundary of a tubular neighbourhood of $\gamma$.  Without any loss of generality we can assume that $x_M\in T$. In Section \ref{s.surgeries}, we defined a meridian and a parallel in $T$, denoted by $m$ and $p$ and forming a basis of the fundamental group of $T$. We also defined the class $P=2p-m$ that plays the role of a ``parallel" when doing surgeries on periodic orbits with negative eigenvalues. Assume without any loss of generality that $m,p$ and $P$ contain $x_M$. Let $m^M$ (resp. $p^M, P^M$) be the homotopy class in $\pi_1(M-\Gamma,x_0)$ obtained by juxtaposing $\delta_M$, then $m$ (resp. $p,P$) and $\delta_M^{-1}$.  By the classical construction of the action of the fundamental group on the orbit space (see \cite{Barbotthese}) and following the notations of the proof of Proposition \ref{p.stabilizersz2}: 

\begin{rema}\label{r.secondproofofcommutation}
When $\gamma$ has positive eigenvalues $s=p^M$ and $t=m^M$. When $\gamma$ has negative eigenvalues $s=P^M$ and $t=p^M$. 
\end{rema}
In both cases, $<s,t>=\pi_1(T)=\mathbb{Z}^2$ which gives us a second proof of the desired result.

\subsection{The main result: $\clos{\mathcal{P}}$ is a complete invariant of $\Phi$ up to surgeries on $\Gamma_M$
}\label{s.barbotgeneral}
\begin{theorem}\label{t.generalbarbot}
Let $M_1,M_2$ be two closed,  orientable 3-manifolds, $(\Phi_1, M_1)$, $(\Phi_2, M_2)$ two transitive Anosov flows, $\Gamma_1$ (resp. $\Gamma_2$) a finite set of periodic orbits of $\Phi_1$ (resp. $\Phi_2$) and $\clos{\mathcal{P}_1}$ (resp.$\clos{\mathcal{P}_2}$) the extension of the bifoliated plane $\widetilde{\mathcal{P}_1-\Gamma_1}$ (resp.$\widetilde{\mathcal{P}_2-\Gamma_2}$). The flow $(\Phi_2, M_2)$ can be obtained (up to orbital equivalence) by performing surgeries on $(\Phi_1, M_1)$ along the orbits $\Gamma_1$, and $\Gamma_2$ are the orbits of $\Phi_2$ corresponding to $\Gamma_1$ after surgery if and only if there exists a homeomorphism $h: \clos{\mathcal{P}_1}\rightarrow \clos{\mathcal{P}_2}$ such that: 

\begin{enumerate}
    \item the image by $h$ of any stable/unstable leaf in $\clos{\mathcal{F}_1^{s,u}}$ is a stable/unstable leaf in $\clos{\mathcal{F}_2^{s,u}}$ 
    \item there exists an isomorphism $\alpha: \pi_1(M_1-\Gamma_1) \rightarrow \pi_1(M_2-\Gamma_2)$ such that for every $g\in \pi_1(M_1-\Gamma_1)$ and every $x\in \clos{\mathcal{P}_1}$ we have $$h(g(x))= \alpha(g)(h(x))$$ 
\end{enumerate}

\end{theorem}

Thanks to this result, we will refer from now on to $\clos{\mathcal{P}_1}$ as \emph{the bifoliated plane of $\Phi_1$ up to surgeries on $\Gamma_1$}. 
\begin{proof}
Suppose that after performing surgeries on $(\Phi_1, M_1)$ along the orbits $\Gamma_1$, we obtain a flow that is orbitally equivalent to $(\Phi_2, M_2)$ and that $\Gamma_2$ is the set of periodic orbits of $\Phi_2$ corresponding to $\Gamma_1$ after surgery. This means that the flows $(\Phi_1-\Gamma_1, M_1-\Gamma_1)$ and $(\Phi_2-\Gamma_2, M_2-\Gamma_2)$ are orbitally equivalent. The orbital equivalence, say $H$, between $\Phi_1-\Gamma_1$ and $\Phi_2-\Gamma_2$ defines an isomorphism $\alpha: \pi_1(M_1-\Gamma_1)\rightarrow \pi_1(M_2-\Gamma_2)$ and lifts to an orbital equivalence $\widetilde{H}$ between the two flows lifted to the universal cover. Moreover, $\widetilde{H}$ is equivariant with respect to the action of the fundamental groups; hence by Proposition \ref{p.orbitspacetilde}, there exists a homeomorphism  $h:\widetilde{\mathcal{P}_1-\Gamma_1} \rightarrow \widetilde{\mathcal{P}_2-\Gamma_2}$ such that 
\begin{enumerate}
    \item the image by $h$ of any stable/unstable leaf in $\widetilde{\mathcal{F}_1^{s,u}}$ is a stable/unstable leaf in $\widetilde{\mathcal{F}_2^{s,u}}$ 
    \item for every $g\in \pi_1(M_1-\Gamma_1)$ and every $x\in \widetilde{\mathcal{P}_1}$ we have $$h(g(x))= \alpha(g)(h(x))$$ 
\end{enumerate}
Since $H$ sends a neihgbourhood of an orbit in $\Gamma_1$ to the neighbourhood of a unique point in $\Gamma_2$, it is easy to see that the homeomorphism $h$ sends a punctured neighbourhood of any point in $\clos{\Gamma_1}$ to a punctured neighbourhood of a unique point in $\clos{\Gamma_2}$. Therefore, $h$ can be extended to a homeomorphism from $\clos{\mathcal{P}_1}$ to  $\clos{\mathcal{P}_2}$. 

Let us now show the converse. Assume that there exists a homeomorphism $h: \clos{\mathcal{P}_1}\rightarrow \clos{\mathcal{P}_2}$ with the above properties. Let us first notice that the cardinals of $\Gamma_1$ and $\Gamma_2$ are equal. Indeed, using the Proposition \ref{p.stabilizersz2}, the cardinal of $\Gamma_1$ (resp. $\Gamma_2$) is equal to the number of orbits -for the action of $\pi_1(M_1-\Gamma_1)$ (resp. $\pi_1(M_2-\Gamma_2)$)- of points in $\clos{\mathcal{P}_1}$ (resp. $\clos{\mathcal{P}_2}$), whose stabilizers are isomorphic to $\mathbb{Z}^2$. By hypothesis, $h$ is equivariant for the action of the fundamental groups and sends points with $\mathbb{Z}^2$ stabilizers to points with $\mathbb{Z}^2$ stabilizers. We conclude that the cardinals of $\Gamma_1$ and $\Gamma_2$ are equal.

In the proof of Proposition \ref{p.stabilizersz2}, for every element $\clos{\gamma}$ in $\clos{\Gamma_1}$ (resp. $\clos{\Gamma_2}$), we constructed a basis $s^{M_1}_{\clos{\gamma}},t^{M_1}_{\clos{\gamma}}$ (resp. $s^{M_2}_{\clos{\gamma}},t^{M_2}_{\clos{\gamma}}$) of its stabilizer group, where $s^{M_1}_{\clos{\gamma}}$ (resp.$s^{M_2}_{\clos{\gamma}}$) fixes all the stable (or unstable) leaves intersecting $\clos{\gamma}$ and $t^{M_1}_{\clos{\gamma}}$ (resp.$t^{M_2}_{\clos{\gamma}}$) acts as a translation by either $+1$ or $+ 2$ on the previous set of leaves  depending on whether $\clos{\gamma}$ has respectfully negative or positive eigenvalues. We will assume that $s^{M_1}_{\clos{\gamma}}$ and $s^{M_2}_{\clos{\gamma}}$ act as expansions on every unstable leaf crossing $\clos{\gamma}$ and also that $\rho_1(t^{M_1}_{\clos{\gamma}})=id$ and $\rho_2(t^{M_2}_{\clos{\gamma}})=id$, where $\rho_1$ (resp.$\rho_2$) is the natural morphism from $\pi_1(M_1-\Gamma_1)$ (resp. $\pi_1(M_2-\Gamma_2)$) to $\pi_1(M_1)$ (resp. $\pi_1(M_2)$). The elements $s^{M_1}_{\clos{\gamma}}$, $t^{M_1}_{\clos{\gamma}}$ (resp. $s^{M_2}_{\clos{\gamma}}$, $t^{M_2}_{\clos{\gamma}}$) are uniquely defined in   $\pi_1(M_1-\Gamma_1)$ (resp.$\pi_1(M_2-\Gamma_2)$) by the above properties. Therefore by the equivariance of $h$ 
\begin{equation}\label{eq.bar1}
    \alpha(s^{M_1}_{\clos{\gamma}})= s^{M_2}_{h(\clos{\gamma})}
\end{equation}

Let us now remark that $h$ sends periodic points of $\clos{\Gamma_1}$ with positive (resp. negative) eigenvalues to periodic orbits of $\clos{\Gamma_2}$ with positive (resp. negative) eigenvalues. Indeed, let $\clos{\gamma}\in \clos{\Gamma_1}$.  By Proposition \ref{p.singularitiesoffoliations}, $\clos{\gamma}$ has positive eigenvalues if and only if the smallest possible translation in $\text{Stab}(\clos{\gamma})$ on the set of stable leaves intersecting $\clos{\gamma}$ is a translation by $\pm 2$. By using the equivariance of $h$, we have that the previous condition stands for $\clos{\gamma}$ if and only if it it stands for $h(\clos{\gamma})$. Therefore, $\clos{\gamma}$ has positive eigenvalues if and only if $h(\clos{\gamma})$ has positive eigenvalues. 

We deduce from the above that the element $\alpha(t^{M_1}_{\clos{\gamma}})$ corresponds to a translation by $+1$ or $+2$ (this depends on the eigenvalues of $\clos{\gamma}$) acting on the set of stable leaves intersecting $h(\clos{\gamma})$.  Therefore, there exists $k\in \mathbb{Z}$ such that \begin{equation}\label{eq.bar2}
\alpha(t^{M_1}_{\clos{\gamma}})= t^{M_2}_{h(\clos{\gamma})}+ k \cdot s^{M_2}_{h(\clos{\gamma})}
\end{equation}

Let $\gamma\in \Gamma_1$ be the periodic orbit in $M_1$ that is associated to $\clos{\gamma}$. Notice that $k$ does not depend on the choice of representative of $\gamma$ in $\clos{\mathcal{P}}$. Indeed, take $g\in \pi_1(M_1-\Gamma_1)$, we have that 
\begin{itemize}
    \item $s^{M_1}_{g(\clos{\gamma})}= gs^{M_1}_{\clos{\gamma}}g^{-1}$,  $t^{M_1}_{g(\clos{\gamma})}=gt^{M_1}_{\clos{\gamma}}g^{-1}$ 
    \vspace{0.7cm}
    \item $s^{M_2}_{h(g(\clos{\gamma}))}= \alpha(g)s^{M_2}_{h(\clos{\gamma})}\alpha(g)^{-1}$, $t^{M_2}_{h(g(\clos{\gamma}))}= \alpha(g)t^{M_2}_{h(\clos{\gamma})}\alpha(g)^{-1}$  \end{itemize}
Therefore,

$$\alpha(t^{M_1}_{g(\clos{\gamma})})= t^{M_2}_{h(g(\clos{\gamma}))}+ k \cdot s^{M_2}_{h(g(\clos{\gamma}))}$$
By the above, to every $\gamma\in \Gamma_1$ we can associate an integer $k(\gamma)$. Consider $(M_3,\Phi_3)$ the Anosov flow obtained from $(M_1,\Phi_1)$ by performing a surgery of coefficient $-k(\gamma)$ for every $\gamma\in \Gamma_1$. Let us show that $(M_3,\Phi_3)$ is orbitally equivalent to $(M_2,\Phi_2)$.  

Denote by $\Gamma_3\subset M_3$ the periodic orbits corresponding to $\Gamma_1$ after surgery. By the first part of this proof, there exists a homeomorphism  $H:\clos{\mathcal{P}_1}\rightarrow \clos{\mathcal{P}_3}$ satisfying the following: 
\begin{enumerate}
    \item the image by $H$ of any stable/unstable leaf in $\clos{\mathcal{F}_1^{s,u}}$ is a stable/unstable leaf in $\clos{\mathcal{F}_3^{s,u}}$ 
    \item using the identification of $\pi_1(M_1-\Gamma_1)$ and $\pi_1(M_3-\Gamma_3)$ and also the fact that the orbital equivalence between  $(M_1-\Gamma_1,\Phi_1-\Gamma_1)$ and $(M_3-\Gamma_3,\Phi_3-\Gamma_3)$ can be taken isotopic to the identity,  for every $g\in \pi_1(M_1-\Gamma_1)$ and every $x\in \clos{\mathcal{P}_1}$, we have $$H(g(x))= g(H(x))$$ 
\end{enumerate}

Furthermore, for any $\clos{\gamma}\in \clos{\Gamma_1}$ we have that \begin{equation}\label{eq.bar3}
    s^{M_1}_{\clos{\gamma}}= s^{M_3}_{H(\clos{\gamma})}
\end{equation}
\begin{equation}\label{eq.bar4}
    t^{M_3}_{H(\clos{\gamma})}= t^{M_1}_{\clos{\gamma}}- k \cdot s^{M_1}_{\clos{\gamma}}
\end{equation}
    Indeed, let us first consider the case where $\gamma$ has positive eigenvalues. Using the notations of Remark \ref{r.secondproofofcommutation}, $s^{M_1}_{\clos{\gamma}}$ (resp. $t^{M_1}_{\clos{\gamma}}$) is defined as the juxtaposition of $\delta_{M_1}$ followed by $p_1$ (resp. $m_1$) and by $\delta_{M_1}^{-1}$, where $p_1,m_1$ is a canonical basis of $T_1$, the boundary of a tubular neighbourhood of $\gamma\in \Gamma_1$. The elements $s^{M_3}_{H(\clos{\gamma})}$ and $t^{M_3}_{H(\clos{\gamma})}$ are similarly defined. But since $M_1-\Gamma_1$ is homeomorphic to $M_3 -\Gamma_3$, we can identify $\delta_{M_1}$ with $\delta_{M_3}$ and $T_1$ with $T_3$. After performing surgery, the natural parallel on $T_3$ didn't change, therefore $p_1=p_3$. However, the natural meridian of $T_3$ satisfies $m_3=m_1- k\cdot p_1$. This proves the above relations. 

Consider now the case where $\gamma$ has negative eigenvalues. Again using the notations of Remark \ref{r.secondproofofcommutation}, $s^{M_1}_{\clos{\gamma}}$ (resp. $t^{M_1}_{\clos{\gamma}}$) is defined as the juxtaposition of $\delta_{M_1}$ followed by $P_1=2p_1-m_1$ (resp. $p_1$) and by $\delta_{M_1}^{-1}$, where $p_1,m_1$ is a canonical basis of $T_1$, the boundary of a tubular neighbourhood of $\gamma\in \Gamma_1$. The elements $s^{M_3}_{H(\clos{\gamma})}$ and $t^{M_3}_{H(\clos{\gamma})}$ are similarly defined. Once again, since $M_1-\Gamma_1$ is homeomorphic to $M_3 -\Gamma_3$, we can identify $\delta_{M_1}$ with $\delta_{M_3}$ and $T_1$ with $T_3$. By  performing surgery on $\Phi_1$, we leave $P_1$ intact and we add $-2k$ copies of $P_1$ to the meridian, hence $m_3=m_1- 2k\cdot P_1$. Therefore, $2p_3=P_3+m_3=P_1+ m_1- 2k\cdot P_1=2p_1- 2k\cdot P_1$. This implies that $p_3=p_1-k\cdot P_1$, which proves the above relations.
Consider now the homeomorphism $K=H\circ h^{-1}: \clos{\mathcal{P}_2}\rightarrow \clos{\mathcal{P}_3}$. Using (\ref{eq.bar1}), (\ref{eq.bar2}), (\ref{eq.bar3}) and (\ref{eq.bar4}) we have that $K$ satisfies the following: 
\begin{enumerate}
    \item the image by $K$ of any stable/unstable leaf in $\clos{\mathcal{F}_2^{s,u}}$ is a stable/unstable leaf in $\clos{\mathcal{F}_3^{s,u}}$ 
    \item there exists an isomorphism $\beta: \pi_1(M_2-\Gamma_2) \rightarrow \pi_1(M_3-\Gamma_3)$ such that for every $g\in \pi_1(M_2-\Gamma_2)$, $\clos{\gamma}\in \clos{\Gamma_2}$ and $x\in \clos{\mathcal{P}_2}$ we have $$K(g(x))= \beta(g)(K(x))$$ \vspace{0.3pt} $$\beta(s^{M_2}_{\clos{\gamma}})= s^{M_3}_{K(\clos{\gamma})}$$ \vspace{0.3pt} $$\beta(t^{M_2}_{\clos{\gamma}})= t^{M_1}_{h^{-1}(\clos{\gamma})}- k \cdot s^{M_1}_{h^{-1}(\clos{\gamma})}= t^{M_3}_{K(\clos{\gamma})}$$
\end{enumerate}
Take $\clos{\gamma_1},...,\clos{\gamma_n}$ a representative of each $\pi_1(M_2-\Gamma_2)$-orbit in $\clos{\Gamma_2}$ and their associated orbits $\gamma_1,...,\gamma_n$ in $M$. Define $m^{M_2}_{\clos{\gamma_i}}=t^{M_2}_{\clos{\gamma_i}}$ when $\clos{\gamma_i}$ has positive eigenvalues and $m^{M_2}_{\clos{\gamma_i}}=2t^{M_2}_{\clos{\gamma_i}}- s^{M_2}_{\clos{\gamma_i}}$ if not (see our discussion prior to Remark \ref{r.secondproofofcommutation}). In both cases the homotopy class $m^{M_2}_{\clos{\gamma_i}}$ corresponds to a meridian around $\gamma_i$. By the Seifert-Van Kampen theorem, $$\text{ker}(\pi_1(M_2-\Gamma_2)\rightarrow \pi_1(M_2))=<m^{M_2}_{\clos{\gamma_1}},...,m^{M_2}_{\clos{\gamma_n}}>^{\pi_1(M_2-\Gamma_2)}$$ where $<A>^{\pi_1(M_2-\Gamma_2)}$ stands for the normal subgroup of $\pi_1(M_2-\Gamma_2)$ generated by $A$. We have that 
\begin{equation}\label{eq.imageofkernel}
K(\text{ker}(\pi_1(M_2-\Gamma_2)\rightarrow \pi_1(M_2)))= \text{ker}(\pi_1(M_3-\Gamma_3)\rightarrow \pi_1(M_3))
\end{equation}
We also remind that (see our discussion prior to Proposition \ref{p.stabilizersz2})  $$\quotient{\clos{\mathcal{P}_2}}{\text{ker}(\pi_1(M_2-\Gamma_2)\rightarrow \pi_1(M_2))}=\mathcal{P}_2$$
$$\quotient{\clos{\mathcal{P}_3}}{\text{ker}(\pi_1(M_3-\Gamma_3)\rightarrow \pi_1(M_3))}=\mathcal{P}_3$$
Finally, by (\ref{eq.imageofkernel}) and since $K$ is equivariant with respect to the group actions on $\clos{\mathcal{P}_2}$ and $\clos{\mathcal{P}_3}$, $K$  projects to a homeomorphism $k:\mathcal{P}_2 \rightarrow \mathcal{P}_3 $ that satisfies the following: 
\begin{enumerate}
    \item the image by $k$ of any stable/unstable leaf in $\mathcal{F}_2^{s,u}$ is a stable/unstable leaf in $\mathcal{F}_3^{s,u}$ 
    \item there exists an isomorphism $\mu: \pi_1(M_2) \rightarrow \pi_1(M_3)$ such that for every $g\in \pi_1(M_2)$ and $x\in \mathcal{P}_2$ we have $$k(g(x))= \mu(g)(k(x))$$ 
\end{enumerate}
By Theorem 3.4 of \cite{Ba1}, we deduce that $\Phi_2$ and $\Phi_3$ are orbitally equivalent and we get the desired result. 
\end{proof}

\section{Rectangle paths as coordinate systems in $\clos{\mathcal{P}}$}\label{s.homotopiesofpaths}

In view of Theorem B, we would like to show that the geometric type of a Markovian family caracterizes the bifoliated plane up to surgeries on the boundary periodic points $\clos{\mathcal{P}}$. In order to do so, in Section \ref{s.liftmarkovianfamilies} we will lift  Markovian families on $\clos{\mathcal{P}}$ and show that the lifted family is associated to the same class of geometric types as its projection on $\mathcal{P}$ (see Proposition \ref{p.liftedmarkovfamiliessamegeomtype}). Next, in Section \ref{s.rectanglepathscoordinates} we will explain how to use rectangle paths in order to navigate simultaneously into two bifoliated planes up to surgeries endowed with Markovian families associated to the same class of geometric types (see Theorem \ref{t.associatingrectanglepaths}). The main goal of this section is to use Markovian families as coordinate systems of the bifoliated plane up to surgeries $\clos{\mathcal{P}}$ and show that Markovian families that are associated to the same class of geometric types correspond to two  compatible coordinate systems in $\clos{\mathcal{P}}$ (see Theorem \ref{t.closedrectanglespathscorrespondtoclosedpaths}). The proof of this result is rather technical, but also the most important step in the proof of Theorem B.

\subsection{Markovian families in $\clos{\mathcal{P}}$}\label{s.liftmarkovianfamilies}
Let $M$ be a closed, oriented 3-manifold carrying a transitive Anosov flow $\Phi$. Let $\mathcal{P}$ be the bifoliated plane of $\Phi$, carrying a Markovian family $\mathcal{R}$. Denote by $\Gamma$ the boundary periodic points of $\mathcal{R}$, by $\clos{\mathcal{P}}$ the bifoliated plane of $\Phi$ up to surgeries on $\Gamma$ and by $\clos{\pi}$ the projection of $\clos{\mathcal{P}}$ on $\mathcal{P}$. Using as a starting point our definition of Markovian family in $\mathcal{P}$, we can extend the notion of Markovian family in $\clos{\mathcal{P}}$. This allows us to define rectangle paths in $\clos{\mathcal{P}}$ and use them later as coordinate systems for comparing  bifoliated planes up to surgeries.  

We define rectangles in $\clos{\mathcal{P}}$ in the exact same way as for $\mathcal{P}$ (see Definition \ref{d.rectanglesinplane}). We then define a Markovian family in $\clos{\mathcal{P}}$ as follows: 
\begin{defi}\label{d.markovfamilyinPbar}
A \emph{Markovian family of rectangles} in $\clos{\mathcal{P}}$ is a set of mutually distinct rectangles $(R_i)_{i \in I}$ covering $\clos{\mathcal{P}}$ such that 
\begin{enumerate}
\item $(R_i)_{i \in I}$ is the union of a finite number of orbits of rectangles by the action of $\pi_1(M-\Gamma)$ 
\item $(\clos{\pi}(R_i))_{i \in I}$ is a Markovian family of $\mathcal{P}$
\end{enumerate} 
\end{defi}
 
 \begin{rema}
 Let $\mathcal{R}$ be a Markovian family in $\mathcal{P}$. The lift $\clos{\mathcal{R}}$ of $\mathcal{R}$ on $\clos{\mathcal{P}}$  is also a Markovian family.
 \end{rema} Indeed, all points in $\Gamma\subset \mathcal{P}$ do not belong in the interior of some rectangle in $\mathcal{R}$. Therefore, $\clos{\mathcal{R}}$  is a family of rectangles in the sense of Definition \ref{d.rectanglesinplane}, that covers  $\clos{\mathcal{P}}$ and that satisfies the second property of Definition \ref{d.markovfamilyinPbar}. The family $\clos{\mathcal{R}}$ consists of a finite number of $\pi_1(M-\Gamma)$-orbits of rectangles as a result of Remark \ref{r.projectequiv}. We deduce that $\clos{\mathcal{R}}$ is a Markovian family in $\overline{\mathcal{P}}$. 
 
The definitions of the notions $k$-th successor/predecessor and $s$-crossing predecessor can be extended for Markovian families on $\overline{\mathcal{P}}$:  
\begin{defi}\label{d.successorinpbar}For any two rectangles $R_1,R_2 \in \overline{\mathcal{R}}$, we will say that $R_1$ is a \emph{predecessor} (resp. \emph{successor}) of $R_2$ if $\overset{\circ}{R_1}\cap \overset{\circ}{R_2} \neq \emptyset$ and the $\clos{\pi}(R_1)\subset \mathcal{P}$ is a predecessor (resp. successor) of $\clos{\pi}(R_2)$. 

We similarly define \emph{predecessors/successors of generation $k\geq 2$} and \emph{$s$-crossing predecessors/successors} on $\clos{\mathcal{P}}$. 
\end{defi}

We define a rectangle path in $\clos{\mathcal{P}}$ in the same way as in $\mathcal{P}$ (see Definition \ref{d.rectanglepath}). 

\begin{defi}
A point in $\overline{\mathcal{P}}$ will be called \emph{periodic} (resp. \emph{boundary arc}, \emph{boundary periodic}) if its projection on $\mathcal{P}$ is a periodic (resp. boundary arc, boundary periodic) point. 

A stable/unstable leaf $f$ of $\overline{\mathcal{P}}$ will be called \emph{periodic} if there exists $g\in \pi_1(M-\Gamma)$ such that $g.f=f$

\end{defi}
We also define a \emph{polygonal/closed polygonal/good polygonal curve} in $\overline{\mathcal{P}}$ in the same way as in $\mathcal{P}$ (see Definitions \ref{d.polygonalcurve} and \ref{d.goodcurve}).

By the same arguments as in Section \ref{markovianfamilytogeometrictype}, we can show that the family $\clos{\mathcal{R}}$ is canonically associated to a unique class of geometric types. In fact, we have the following result:

\begin{prop}\label{p.liftedmarkovfamiliessamegeomtype}
The Markovian families $\clos{\mathcal{R}}$ and $\mathcal{R}$ are associated to the same class of geometric types
\end{prop}
\begin{proof}
 Let us endow the foliations $\mathcal{F}^{s,u}$ with an orientation. This orientation canonically defines an orientation on $\clos{\mathcal{F}^{s,u}}$. Consider also a choice of representatives $r_1,...,r_n$ of every rectangle orbit in $\mathcal{R}$ and a lift $\clos{r_i}\in \clos{\mathcal{R}}$ of every $r_i$. Thanks to Remark \ref{r.projectequiv}, the rectangles $\clos{r_i}$ are representatives of every rectangle orbit in $\clos{\mathcal{R}}$. According to Remark \ref{r.canonicalassociationgeometrictype},  together with this choice of representatives and orientations, the Markovian families $\clos{\mathcal{R}}$ and $\mathcal{R}$ can be associated to a unique geometric type $\clos{G}=(\clos{n},(\clos{h_i})_{i \in \llbracket 1,n \rrbracket}, (\clos{v_i})_{i\in \llbracket 1,n \rrbracket}, \clos{\mathcal{H}}, \clos{\mathcal{V}},\clos{\phi},\clos{u})$ and $G=(n,(h_i)_{i \in \llbracket 1,n \rrbracket}, (v_i)_{i\in \llbracket 1,n \rrbracket}, \mathcal{H}, \mathcal{V},\phi, u)$ respectfully. 

Let us now go back to the construction of $G$ and $\clos{G}$ (see proof of Theorem \ref{t.associatemarkovfamiliestogeometrictype}). The number $n$ (resp. $\clos{n}$) in $G$ (resp. $\clos{G}$) is equal to the number of distinct orbits in $\mathcal{R}$ (resp.$\clos{\mathcal{R}}$) by the action of $\pi_1(M)$ (resp. $\pi_1(M-\Gamma)$) . By Remark \ref{r.projectequiv}, we get that $n=\clos{n}$. 

Next, the number $h_i$ (resp.$v_i$) will correspond to the number of successors (resp. predecessors) of $r_i$. Same for $\clos{h_i}$ and $\clos{v_i}$.  By Definition \ref{d.successorinpbar}, we easily have that $\clos{r_i}$ and $r_i$ have the same number of predecessors and successors. In other words, $h_i=\clos{h}_i$ and $v_i=\clos{v}_i$.

Now, let us order the successors (resp. predecessors) of every $r_i$ from bottom to top (resp. from left to right) using the orientations of $\mathcal{F}^{u}$ (resp. $\mathcal{F}^{s}$). We will denote the $k$-th successor (resp.predecessor) of $r_i$ for this order by $H_i^k$ (resp. $V_i^k$). Recall that $\mathcal{H}=\{H^k_i, i\in \llbracket 1, n\rrbracket, k\in \llbracket 1, h_i\rrbracket\}$ and $\mathcal{V}=\{V^k_i, i\in \llbracket 1, n\rrbracket, k\in \llbracket 1, v_i\rrbracket\}$. We similarly define $\clos{V_i^k}$, $\clos{H^k_i}$ for $\clos{r_i}$ and $\clos{\mathcal{H}}, \clos{\mathcal{V}}$. By our choice of orientations on $\clos{\mathcal{F}^{s,u}}$ we have that $\clos{\pi}(\clos{H^k_i})=H_i^k$ and $\clos{\pi}(\clos{V_i^k})=V_i^k$. Therefore, we can canonically identify  $\mathcal{H}$ with $\clos{\mathcal{H}}$ and $\mathcal{V}$ with  $\clos{\mathcal{V}}$. 

Also, $\phi(H^k_i)=V^l_j$ if and only if there exists $g\in \pi_1(M)$ such that $g(H^k_i)=r_j$ and $g(r_i)$ is the $l$-th predecessor (from left to right) of $r_j$. Recall that the previous $g$ is unique if it exists and that $u(H^k_i)=+1$ if $g$ preserves the orientations of the stable/unstable foliations and $u(H^k_i)=-1$ if not. The functions $\clos{\phi}$ and $\clos{u}$ are similarly defined. Thanks to Remark \ref{r.projectequiv}, we deduce that $\phi(H^k_i)=V^l_j$ if and only if $\clos{\phi}(\clos{H^k_i})=\clos{V^l_j}$. Finally, take $\clos{h}\in\clos{\mathcal{H}}$, $h=\clos{\pi}(\clos{h})\in \mathcal{H}$, $g$ (resp. $\clos{g}$) the unique element in $\pi_1(M)$ (resp. $\pi_1(M-\Gamma)$) such that $g(h)=r_j$ (resp. $\clos{g}(\clos{h})=\clos{r_j}$). It is not difficult to see that $g$ preserves the orientation of the foliations if and only if $\clos{g}$ does. Hence, $u(h)=\clos{u}(\clos{h})$. Finally, using the identifications between $\mathcal{H}$ and $\clos{\mathcal{H}}$, $\mathcal{V}$ and $\clos{\mathcal{V}}$, we have that $\phi$ and $\clos{\phi}$ (resp. $u$ and $\clos{u}$) define the same functions from $\mathcal{H}$ to $\mathcal{V}$, which proves the desired result.
\end{proof}
\subsection{Rectangle paths as coordinate systems} \label{s.rectanglepathscoordinates}

Let $(M_1,\Phi_1)$ and $(M_2,\Phi_2)$ be two transitive Anosov flows, $\mathcal{P}_1$ and $\mathcal{P}_2$ their bifoliated planes, $\mathcal{F}_{1,2}^{s}$,$\mathcal{F}_{1,2}^{u}$ the stable and unstable foliations in $\mathcal{P}_{1,2}$. Assume that $\mathcal{P}_1$ and $\mathcal{P}_2$ carry two Markovian families $\mathcal{R}_1$ and $\mathcal{R}_2$ associated canonically by Theorem \ref{t.associatemarkovfamiliestogeometrictype} to the same class of geometric types. By appropriately choosing representatives for every rectangle orbit in $\mathcal{R}_1$ and $\mathcal{R}_2$, and orientations for the foliations $\mathcal{F}_{1,2}^{s}$,$\mathcal{F}_{1,2}^{u}$, we may assume that $\mathcal{R}_1$ and $\mathcal{R}_2$ are canonically associated to the same geometric type $(n,(h_i)_{i \in \llbracket 1,n \rrbracket}, (v_i)_{i\in \llbracket 1,n \rrbracket}, \mathcal{H}, \mathcal{V},\phi, u)$, thanks to Remark \ref{r.canonicalassociationgeometrictype} and Lemma \ref{l.geomtypeinclass}.

Let us denote $\Gamma_{1,2}$ (resp. $\Gamma_{1,2}^{M_{1,2}}$) the boundary periodic orbits associated to $\mathcal{R}_{1,2}$ in $\mathcal{P}_{1,2}$ (resp. in $M_{1,2}$), $\clos{\mathcal{P}_{1,2}}$ the bifoliated planes of $\Phi_1,\Phi_2$ up to surgeries on $\Gamma_{1,2}$, $\clos{\mathcal{F}^{s,u}_{1,2}}$ their stable and unstable singular foliations, $\clos{\mathcal{R}_{1,2}}$ the lifts of $\mathcal{R}_{1,2}$ on $\clos{\mathcal{P}_{1,2}}$ and $\clos{\Gamma_{1,2}}$ the lifts of $\Gamma_{1,2}$ on $\clos{\mathcal{P}_{1,2}}$. We would like to show that any rectangle path in $\clos{\mathcal{P}_{1}}$ corresponds to a rectangle path in $\clos{\mathcal{P}_{2}}$.

By Proposition \ref{p.liftedmarkovfamiliessamegeomtype}, since $\mathcal{R}_1$ and $\mathcal{R}_2$ are associated to the same geometric type, we can endow $\clos{\mathcal{F}^{s,u}_{1,2}}$ with orientations and we can chose representatives in every rectangle orbit in $\clos{\mathcal{R}_{1,2}}$  so that $\clos{\mathcal{R}_1}$ and $\clos{\mathcal{R}_2}$ are  associated to the same geometric type $(n,(h_i)_{i \in \llbracket 1,n \rrbracket}, (v_i)_{i\in \llbracket 1,n \rrbracket}, \mathcal{H}, \mathcal{V},\phi, u)$. By the proof of Proposition  \ref{p.liftedmarkovfamiliessamegeomtype}, every $\pi_1(M_{1,2}-\Gamma_{1,2}^{M_{1,2}})$-orbit of rectangles in $\clos{\mathcal{R}_{1,2}}$ corresponds to a unique $i\in\llbracket 1,n\rrbracket$ and a unique couple $(h_i,v_i)$ of the geometric type. 
\begin{defi}
We will say that $R^1 \in \clos{\mathcal{R}_1}$ (resp. $R^2 \in \clos{\mathcal{R}_2}$) are of \emph{type $i$} if the $\pi_1(M_{1}-(\Gamma_{1})_{M_{1}})$-orbit (resp. $\pi_1(M_{2}-(\Gamma_{2})_{M_{2}})$ ) of $R^1$ (resp. $R^2$) in $\clos{\mathcal{P}_1}$ (resp.$\clos{\mathcal{P}_2}$) corresponds to the same integer $i\in\llbracket 1,n\rrbracket$ and couple $(h_i,v_i)$ of the geometric type. 
\end{defi}

Fix $r_0^1\in \clos{\mathcal{R}_1}$ and $r_0^2 \in \clos{\mathcal{R}_2}$ a rectangle of the same type as $r_0^1$. From now on, we will call $r_0^{1,2}$ the \textit{origin rectangle} in $\clos{\mathcal{P}_{1,2}}$ and we will also call any rectangle path in $\clos{\mathcal{P}_{1,2}}$ starting from $r_0^{1,2}$ a \textit{centered rectangle path}.

\begin{theorem}\label{t.associatingrectanglepaths}
Any centered rectangle path $r^1_0,...,r^1_n$ in $\clos{\mathcal{P}_{1}}$ corresponds canonically to a unique centered rectangle path $r^2_0,...,r^2_n$ in $\clos{\mathcal{P}_{2}}$ such that for every $i\in \llbracket 0,n \rrbracket$ $r^1_i$ and $r^2_i$ are of the same type. 

\end{theorem}
\begin{proof}
Take $r^1_0,...,r^1_n$ a centered rectangle path in $\clos{\mathcal{P}_{1}}$.  

Let us start by constructing $r^2_1$. Since $r^1_0,...,r^1_n$ is a rectangle path $r^1_1$ is a predecessor or a successor of $r^1_0$. Assume without any loss of generality that it is a successor. Using the orientation on $\clos{\mathcal{F}^{u}_{1}}$, assume that $r^1_1$ is the $k-$th successor of $r^1_0$ from bottom to top. Take $r^2_1$ to be the $k$-th successor of $r^2_0$ from bottom to top (we use here the orientation on $\clos{\mathcal{F}^{u}_{2}}$). By the proof of Proposition \ref{p.liftedmarkovfamiliessamegeomtype} and the fact that $\clos{\mathcal{R}_{1}}$ and $\clos{\mathcal{R}_{2}}$ correspond to the same geometric type, we have that $r^2_1$ is of the same type as $r^1_1$. We construct in the exact same way $r^2_2,...,r_n^2$ by induction. 
\end{proof}
We will call the rectangle path $r^2_0,...,r^2_n$  the \textit{associated rectangle path} of  $r^1_0,...,r^1_n$. In the next pages, we will use rectangle paths and associated rectangle paths as coordinate systems that will allow us to compare $\clos{\mathcal{P}_1}$ and $\clos{\mathcal{P}_2}$. The main goal of this section is to prove that those coordinate systems are compatible:
\begin{theorem}\label{t.endingbysamerectangle}
Any two centered rectangle paths in $\clos{\mathcal{P}_1}$ ending by the same rectangle are associated to two centered rectangle paths in $\clos{\mathcal{P}_2}$ ending by the same rectangle.
\end{theorem} 
The previous theorem is equivalent to the following:
\begin{theorem}\label{t.closedrectanglespathscorrespondtoclosedpaths}
Any closed centered rectangle path in $\clos{\mathcal{P}_1}$ is associated to a closed centered rectangle path in $\clos{\mathcal{P}_2}$.
\end{theorem} 
\begin{proof}[Proof of the equivalence of theorems \ref{t.endingbysamerectangle} and \ref{t.closedrectanglespathscorrespondtoclosedpaths}]$\quad$

(\ref{t.endingbysamerectangle} $\Rightarrow$ \ref{t.closedrectanglespathscorrespondtoclosedpaths}): Take $r^1_0$ the trivial rectangle path and $R$ any closed centered rectangle path in $\clos{\mathcal{P}_1}$. Since $R$ is closed, it also ends at $r^1_0$. By Theorem \ref{t.endingbysamerectangle}, the two rectangle paths are associated to two centered rectangle paths of $\clos{\mathcal{P}_2}$ ending at the same rectangle. But the trivial rectangle path $r^1_0$ is associated by Theorem \ref{t.associatingrectanglepaths} to the trivial rectangle path $r^2_0$. Therefore, $R$ is associated to a closed centered rectangle path in  $\clos{\mathcal{P}_2}$.

(\ref{t.closedrectanglespathscorrespondtoclosedpaths}$\Rightarrow$ \ref{t.endingbysamerectangle}): 
Take two centered rectangle paths $r^1_0,...r^1_n$ and $r^1_0=R^1_0,R^1_1,...,R^1_m=r^1_n$ in $\clos{\mathcal{P}_1}$. The rectangle path $r^1_0,...r^1_n, R^1_{m-1}, R^1_{m-2},...,R^1_1,r^1_0$ is a closed centered rectangle path in $\clos{\mathcal{P}_1}$. Therefore, by Theorem \ref{t.closedrectanglespathscorrespondtoclosedpaths} it is associated to a closed centered rectangle path in $\clos{\mathcal{P}_2}$, say $r^2_0,...r^2_n, R^2_{m-1}, R^2_{m-2},...,R^2_1,r^2_0$. It is not difficult to see, that $r^2_0,...r^2_n$ is the associated rectangle path of $r^1_0,...r^1_n$ and that $r^2_0=R^2_0,R^2_1,...,R^2_{m-1},R^2_{m}:=r^2_n$ is the associated rectangle path of $R^1_0,R^1_1,...,R^1_m$. We deduce that the two associated rectangle paths end by the same rectangle. 

\end{proof}

Our goal from now on will be to prove Theorem \ref{t.closedrectanglespathscorrespondtoclosedpaths}, which is the most important step in the proof of Theorem B. In order to do that we will introduce the notion of \emph{homotopy of rectangle paths}, we will show that only closed rectangle paths can be homotopic to the trivial rectangle path (see Section \ref{s.closedpathsarehomotopicallytr}), that any closed centered rectangle path in $\clos{\mathcal{P}_1}$ is homotopic to the trivial rectangle path $r_0^1$ (see Section \ref{s.proofoftheoremhomtrivialpath}) and finally that any homotopy of rectangle paths on $\clos{\mathcal{P}_1}$ is associated to a homotopy of rectangle paths in $\clos{\mathcal{P}_2}$ (see Section \ref{s.proofoftheoremcorrespondancehom}).

\subsection{Different types of homotopies of rectangle paths}\label{s.closedpathsarehomotopicallytr}
Take $(M,\Phi)$ a transitive Anosov flow, $\mathcal{P}$ its bifoliated plane, $\mathcal{F}^{s,u}$ the stable/unstable foliations in $\mathcal{P}$, $\mathcal{R}$ a Markovian family on $\mathcal{P}$, $\Gamma \subset \mathcal{P}$ the set of boundary points of $\mathcal{R}$, $\clos{\mathcal{P}}$ the bifoliated plane of $\Phi$ up to surgeries on $\Gamma$, $\clos{\mathcal{R}}$ and  $\clos{\Gamma}$ the lifts of $\mathcal{R}$ and $\Gamma$ on  $\clos{\mathcal{P}}$, and $\clos{\mathcal{F}^{s,u}}$ the singular stable/unstable foliations in $\clos{\mathcal{P}}$. 

By repeating the arguments used in Section \ref{s.rectanglepathsinP}, we can prove the analogues of Lemma  \ref{l.createpolygonalcurve} and Propositions \ref{p.curvetorectanglepath} and \ref{p.rectanglepathtocurve} for $\clos{\mathcal{P}}$. It therefore remains true that we can associate good polygonal curves in $\clos{\mathcal{P}}$ to rectangle paths of $\clos{\mathcal{R}}$ and vice versa. In the following lines, we will use this association in order to define the notion of homotopy for rectangle paths. 

We will fix for the rest of this section $R_0,...,R_n$ a rectangle path in $\clos{\mathcal{R}}$.
\begin{defi}
We will say that two rectangle paths are \emph{homotopic by a homotopy of type $\mathcal{A}$} if they are equal. 

Two good polygonal curves $\gamma,\delta:[0,1]\rightarrow \clos{\mathcal{P}}$ associated to the rectangle path $R_0,...,R_n$ will be called \emph{homotopic by a homotopy of type $\mathcal{A}$} (relatively to $R_0,...,R_n$).

\end{defi}
The notion of homotopy of type $\mathcal{A}$ between two polygonal curves $\gamma$ and $\delta$ associated to the same rectangle path can be described in terms of the existence of a homotopy $H:[0,1]^2\rightarrow \clos{\mathcal{P}}$ from $\gamma$ to $\delta$ satisfying the following property: for every $t\in [0,1]$ there exists a $(n+2)$-uple $d_0=0<d_1<...<d_{n+1}=1$ such that $H(t,[d_k,d_{k+1}])\subset R_k$. It is possible to prove that the two definitions coincide. The same applies for homotopies of type $\mathcal{B}$ and $\mathcal{C}$ that will later be defined. For the sake of simplicity, we will define homotopies of good polygonal curves in terms of the associated rectangle paths, rather than the existence of a homotopy satisfying a number of properties.  

Notice that the notion of homotopy of type $\mathcal{A} $ for good polygonal curves depends on the rectangle path relatively to which we perform it. Indeed, take $\gamma$ a good polygonal curve contained inside $R\in \clos{\mathcal{R}}$ and also in the union of two predecessors of $R$, $R_1$ and $R_2$ such that $\gamma(0)\in R_1$ and $\gamma(1)\in R_2$. By Proposition \ref{p.curvetorectanglepath}, we can associate $\gamma$ to a non-trivial rectangle path starting from $R_1$ and also to the trivial rectangle path $R$. Notice that relatively to $R$ the curve $\gamma$ is homotopic by a homotopy of type $\mathcal{A}$ to a point. Relatively to the non-trivial rectangle path starting from $R_1$, this is not the case. 

\begin{conv}
 By an abuse of language, whenever the context makes it clear, we will omit the rectangle paths relatively to which we will perform homotopies of good polygonal curves.  
\end{conv}

\begin{figure}[h!]
\includegraphics[scale=0.75]{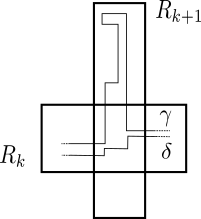}
\caption{$\gamma$ and $\delta$ are homotopic by a homotopy of type $\mathcal{B}$}
\label{f.homotopyoftypeB}
\end{figure}
\begin{defi}\label{d.homotopyB}
Consider $R_0,...,R_n$ a rectangle path such that  there exists $k\in \llbracket 0,n-2\rrbracket$ for which $R_k=R_{k+2}$. We are going to say that the rectangle paths $R_0,..,R_k,R_{k+3},R_{k+4},...$ and $R_0,...,R_n$ are \emph{homotopic by a homotopy of type $\mathcal{B}$}. 

Consider $\gamma$, $\delta$ two good polygonal curves associated to the rectangle paths $R_0,...,R_n$ and $R_0,..,R_k,R_{k+3},R_{k+4},...,R_n$ respectively (see Figure \ref{f.homotopyoftypeB}). We will say that $\gamma$ and $\delta$ are \emph{homotopic by a homotopy of type $\mathcal{B}$} (relatively to $R_0,...,R_n$ and $R_0,..,R_k,R_{k+3},R_{k+4},...,R_n$).  
\end{defi}
It is not difficult to see that both homotopies of type $\mathcal{A}$ or type $\mathcal{B}$ (seen as maps $H:[0,1]^2\rightarrow \clos{\mathcal{P}}$) can be performed in the interior of rectangles without ever crossing a boundary arc point. We would like at this point to define a homotopy allowing good polygonal curves to cross boundary arc points in $\overline{\mathcal{P}}$. We will name this homotopy, \emph{homotopy of type $\mathcal{C}$}. In order to introduce the notion of homotopy of type $\mathcal{C}$, we will first need to introduce the notion of \textit{cycle} around a boundary arc point: 
\begin{defi}\label{d.cyclearcpoint}
Take $p\in \clos{\mathcal{P}}$ a boundary arc point and $L_0\in \clos{\mathcal{R}}$ such that $L_0$ contains $p$ and two germs of quadrants of $p$. Assume without any loss of generality that $p\in \partial^s L_0$ (see Figure \ref{f.cycle}). Consider $r$ the stable boundary component of $L_0$ containing $p$. By Remark \ref{r.caracterisationarcpoints}, there exists $L_1$ an $r$-crossing predecessor of $L$ such that $p\in \partial{L_0}\cap \partial{L_1}$. 

Consider $r_1$ the unstable boundary component of $L_1$ containing $p$. Again, by Remark \ref{r.caracterisationarcpoints}, there exists $L_2$ an $r_1$-crossing successor of $L_1$ such that $p\in \partial{L_2}\cap \partial{L_1}$ and $\overset{\circ}{L_2}\cap \overset{\circ}{L_0}=\emptyset$. By the same exact procedure, we construct $L_3$ and $L_4$ such that $\overset{\circ}{L_3}\cap \overset{\circ}{L_1}=\emptyset$ and $\overset{\circ}{L_4}\cap \overset{\circ}{L_2}=\emptyset$. We are going to call $(L_0,L_1,L_2,L_3,L_4)$ a \emph{cycle} around $p$ starting from $L_0$.
\end{defi}
\begin{figure}[h!]
\includegraphics[scale=1]{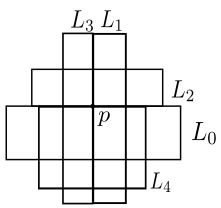}
\caption{}
\label{f.cycle}
\end{figure}

\begin{rema}\label{r.existenceofcycle}
For any boundary arc point $p$, there exists a cycle around $p$.\end{rema}

Indeed, by Definition \ref{d.cyclearcpoint}, it suffices to show that there exists $L \in\clos{\mathcal{R}}$ containing $p$ and the germs of two quadrants of $p$. Take $L\in\clos{\mathcal{R}}$ containing $p$. Since $p$ is a boundary arc point, it is contained in the boundary of $L$. If furthermore $p$ is not a corner point of $L$, then $L$ contains the germs of two quadrants of $p$. If $p$ is a corner point, let us name $s$ the stable boundary of $L$ in which it is contained. Consider $L'$ the $s$-crossing predecessor of $L$ that contains $p$. The rectangle $L'$ has the desired property.

Boundary arc points, by definition, are not periodic points, they are therefore regular points of the foliations $\clos{\mathcal{F}^{s}}$ and $\clos{\mathcal{F}^{u}}$. Take $p$ a boundary arc point and $(L_0,L_1,L_2,L_3,L_4)$ a cycle around of $p$. Let us denote $G_{(\epsilon,\epsilon')}$ a germ of the $(\epsilon,\epsilon')$ quadrant of $p$, where $\epsilon,\epsilon'\in \lbrace -,+ \rbrace$. By the definition of a cycle, up to cyclic permutation of the indices of the $L_i$ a cycle verifies one of the two following statements: 

\vspace{0.2cm}
\hspace{-1.5cm}\begin{minipage}{0.5\linewidth}
\center{\underline{Positive cycle}}
\begin{itemize}
    \item $\overset{\circ}{L_0}\cap G_{(-,-)}\neq \emptyset$, $\overset{\circ}{L_0}\cap G_{(+,-)}\neq \emptyset$ 
    \item $\overset{\circ}{L_1}\cap G_{(+,-)}\neq \emptyset$, $\overset{\circ}{L_1}\cap G_{(+,+)}\neq \emptyset$
    \item $\overset{\circ}{L_2}\cap G_{(+,+)}\neq \emptyset$, $\overset{\circ}{L_2}\cap G_{(-,+)}\neq \emptyset$
    \item $\overset{\circ}{L_3}\cap G_{(-,+)}\neq \emptyset$, $\overset{\circ}{L_3}\cap G_{(-,-)}\neq \emptyset$
    \item $\overset{\circ}{L_4}\cap G_{(-,-)}\neq \emptyset$, $\overset{\circ}{L_4}\cap G_{(+,-)}\neq \emptyset$
\end{itemize} 
\end{minipage}
\hspace{-0.5cm}\begin{minipage}{0.5\linewidth}
\center{\underline{Negative cycle}}
\begin{itemize}
    \item $\overset{\circ}{L_0}\cap G_{(+,-)}\neq \emptyset$, $\overset{\circ}{L_0}\cap G_{(-,-)}\neq \emptyset$ 
    \item $\overset{\circ}{L_1}\cap G_{(-,-)}\neq \emptyset$, $\overset{\circ}{L_1}\cap G_{(-,+)}\neq \emptyset$
    \item $\overset{\circ}{L_2}\cap G_{(-,+)}\neq \emptyset$, $\overset{\circ}{L_2}\cap G_{(+,+)}\neq \emptyset$
    \item $\overset{\circ}{L_3}\cap G_{(+,+)}\neq \emptyset$, $\overset{\circ}{L_3}\cap G_{(+,-)}\neq \emptyset$
    \item $\overset{\circ}{L_4}\cap G_{(+,-)}\neq \emptyset$, $\overset{\circ}{L_4}\cap G_{(-,-)}\neq \emptyset$
\end{itemize}
\end{minipage}

\vspace{0.2cm}
In the first case, we will call the cycle \textit{positive} and in the second one \textit{negative}. Notice in particular, that for every cycle $(L_0,L_1,L_2,L_3,L_4)$ around $p$, we have that $\overset{\circ}{L_4}\cap \overset{\circ}{L_0}\neq \emptyset$. 
\begin{lemm}\label{l.twocycles}
Take $p\in \clos{\mathcal{P}}$ a boundary arc point and $L_0 \in \clos{\mathcal{R}}$ such that $L_0$ contains $p$ and the germs of two quadrants of $p$. There exist exactly two cycles around $p$ starting from $L_0$: one positive and one negative. 
\end{lemm}
\begin{proof}
Without any loss of generality, we can assume that $p$ is contained in a stable boundary component of $L_0$, say $s$, and that a germ of the $(-,-)$ and $(+,-)$ quadrants of $p$ is contained in $L_0$. By Remark \ref{r.caracterisationarcpoints}, since $p$ is a boundary arc point and any two distinct $s$-crossing predecessors of $L_0$ intersect only along their boundaries, there exist exactly two $s$-crossing predecessors $L_1,L_1'$ of $L_0$ containing $p$: $L_1$ contains a germ of the $(+,+)$ and $(+,-)$ quadrants of $p$ and $L_1'$ a germ of the $(-,+)$ and $(-,-)$ quadrants of $p$. Since $L_1$ is an $s$-crossing predecessor, the point $p$ is contained in the interior of some unstable boundary component of $L_1$. Therefore, by our previous argument the rectangle $L_2$ of Definition \ref{d.cyclearcpoint} is uniquely defined. Similarly, the rectangles $L_3$ and $L_4$ of Definition \ref{d.cyclearcpoint} are uniquely defined. The same argument applies for $L_1'$. We thus obtain exactly two cycles around $p$ starting from $L_0$: one positive and one negative. 
\end{proof}
\begin{lemm}
Take $p\in \clos{\mathcal{P}}$ a boundary arc point. If  $(L_0,L_1,L_2,L_3,L_4)$ is a cycle around $p$, then  $(L_0,L_3,L_2,L_1,L_4)$ is also a cycle around $p$. 
\end{lemm}
\begin{proof}
Let us assume without any loss of generality that $p \in \partial^s L_0$. Denote by $s$ the stable boundary component of $L_0$ containing $p$. Let us show that $L_3$ is an $s$-crossing predecessor of $L_0$. 

If $s'$ is the stable boundary component of $L_2$ containing $p$, then by construction $L_3$ is a $s'$-crossing predecessor of $L_2$. By  Remark \ref{r.caracterisationarcpoints}, the unstable boundary of $L_3$ intersects $s'$ along boundary arc points. By the Markovian intersection property, $L_3\cap L_0$ is a vertical subrectangle of $L_0$ (see Figure \ref{f.cycle}). Since $L_3\cap s$ consists of two boundary arc points, we deduce from Remark \ref{r.caracterisationarcpoints} that $L_3$ is a $s$-crossing predecessor of $L_0$. We show in the exact same way, that $L_2$ is a crossing successor of $L_3$, that $L_1$ is a crossing predecessor of $L_2$ and that $L_4$ is a crossing successor of $L_1$. Finally, since $(L_0,L_1,L_2,L_3,L_4)$ is a cycle, we have that $\overset{\circ}{L_2}\cap \overset{\circ}{L_0}=\emptyset$, $\overset{\circ}{L_3}\cap \overset{\circ}{L_1}=\emptyset$ and $\overset{\circ}{L_4}\cap \overset{\circ}{L_2}=\emptyset$. We conclude that $(L_0,L_3,L_2,L_1,L_4)$ is also a cycle around $p$. 
\end{proof}
\begin{rema}\label{r.cycles}
\begin{itemize}
    \item Notice that in the above lemma, the rectangles $L_1,L_2,L_3,L_4$ are canonically associated to $p$ independently of the choice of $L_0$. Indeed, assume that $L_1$ contains the point $p$ and a germ of the $(+,+)$ and $(+,-)$ quadrants of $p$. Then $L_1$ corresponds to the unique rectangle in $\clos{\mathcal{R}}$ containing in its unstable boundary $p$ and the boundary arc point after $p$ in $\clos{\mathcal{F}^s_+(p)}$.
    \item Generally, $L_0\neq L_4$. This results from the fact that, contrary to the case of successors/predecessors, if $L_1$ is a $s$-crossing predecessor of $L_0$, then $L_0$ is not necessarily a $u$-crossing successor of $L_1$ (see Figure \ref{f.cycle}).
\end{itemize}
\end{rema}
\begin{defi}\label{d.generalisedrectpaths}
Consider $L_0,...,L_n$ a sequence of rectangles in $\clos{\mathcal{R}}$ such that for every $i$, the rectangle $L_{i+1}$ is a successor or predecessor or $u$-successor or $s$-predecessor of $L_i$. We will call $L_0,...,L_n$ a \emph{generalised rectangle path}. By Proposition \ref{l.npredecessor} and Lemma \ref{l.existenceofpredecessors}, for every $i$ there exist a unique monotonous rectangle path (see Definition \ref{d.rectanglepath})  $L_i=L_{i0},L_{i1},...L_{i(s(i)+1)}=L_{i+1}$ going from $L_i$ to $L_{i+1}$. 
We are going to say that the generalised rectangle path $L_0,...,L_n$ is \emph{associated} to the rectangle path $L_0,L_{01},...,L_{0s(0)},L_1, L_{11},...,L_{1s(1)},L_2,...., L_{n-1},L_{(n-1)1},...,L_{(n-1)s(n-1)}, L_{n}$.
\end{defi}
\begin{defi}\label{d.homotopyC}
Consider $R_0,...,R_n$ a rectangle path, $p\in \clos{\mathcal{P}}$ a boundary arc point, $L_0\in \clos{\mathcal{R}}$ such that $L_0$ contains $p$ and the germs of two quadrants of $p$. Denote by $(L_0,L_1,L_2,L_3,L_4)$ and $(L_0,L_3,L_2,L_1,L_4)$ the two cycles around $p$ starting from $L_0$. Consider $(L_0,L_1...,L_k)$ the first $k+1$ terms of the first cycle with $k\in \llbracket 1, 4\rrbracket $ and $(L_0,L_3...,L_k)$ the part of the second cycle starting from $L_0$ and ending at $L_k$. 

Following the notations of Definition \ref{d.generalisedrectpaths}, we can associate $L_0,L_1,...,L_k$ and $L_0,L_3,...,L_k$ to the rectangle paths $L_0,L_{01},...,L_{0s(0)},L_1, L_{11},...,L_{1s(1)},...,L_k$ and \newline{}$L_0,L'_{01},...,L'_{0s'(0)},L_3, L'_{31},...,L'_{3s'(3)},...,L_k$ respectively. Assume that the first rectangle path is contained in $R_0,...,R_n$. In other words, assume that $R_0,...,R_n$ is of the form $$ R_0,...,R_m,L_0,L_{01},...,L_{0s(0)},L_1, L_{11},...,L_{1s(1)},..., L_k,R_l,R_{l+1},...,R_n$$  
Consider now the rectangle path 
$$R'_0,...,R'_{n'}:=R_0,...,R_m,L_0,L'_{01},...,L'_{0s'(0)},L_3, L'_{31},...,L'_{3s'(3)},...,L_k,R_l,R_{l+1},...,R_n$$

We are going to say that the rectangle paths $R_0,...,R_n$ and $R'_0,...,R'_{n'}$ are \emph{homotopic by a homotopy of type $\mathcal{C}$}.

Take $\gamma$ and $\delta$ two good polygonal curves associated to the rectangle paths $R_0,...,R_n$ and $R'_0,...,R'_{n'}$ respectively (see Figure \ref{f.homotopyC}). We will say that $\gamma$ is homotopic to $\delta$ by a \emph{homotopy of type $\mathcal{C}$} (relatively to $R_0,...,R_n$ and $R'_0,...,R'_{n'}$). 
\end{defi}
\begin{figure}[h!]
\includegraphics[scale=0.6]{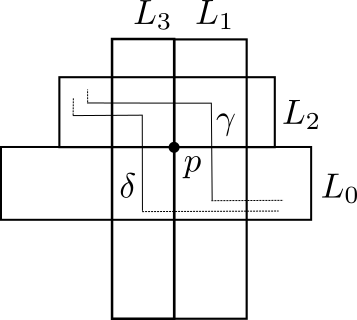}
\caption{}
\label{f.homotopyC}
\end{figure}


\begin{defi}
Take $R:=r_0,r_1,...,r_n$ and $R':=r_0,r'_1,...,r'_m$ two centered rectangle paths. We are going to say that $R$ and $R'$ are \emph{homotopic} if there exists $R_0:=R, R_1, ..., R_s:=R'$ a sequence of rectangle paths such that  $R_i$ is homotopic to $R_{i+1}$ by a homotopy of type $\mathcal{A}$, $\mathcal{B}$ or $\mathcal{C}$ (relatively to $R_i$ and $R_{i+1}$). 
\end{defi}

 Let us remark that because of our definition of homotopies of type $\mathcal{A},\mathcal{B} $ and $\mathcal{C}$, a closed rectangle path is only homotopic to closed rectangle paths. Therefore, a rectangle path homotopic to the trivial rectangle path, is necessarily closed. It turns out that the converse of the above result is also true:  

\begin{theorem}\label{t.homotopictotrivialpath}
Any centered closed rectangle path $r_0,r_1,...,r_n$ in $\clos{\mathcal{P}}$ is homotopic to the trivial rectangle path $r_0$

\end{theorem}
The above theorem constitutes the key argument in the proof of Theorem \ref{t.closedrectanglespathscorrespondtoclosedpaths}, which we will later show that it implies  Theorem B. 

\subsection{Proof of Theorem \ref{t.homotopictotrivialpath}} 
\label{s.proofoftheoremhomtrivialpath}
The proof of the above theorem is rather technical and will be split two parts: 
\begin{enumerate}
    \item Proving Theorem \ref{t.homotopictotrivialpath} in the case of rectangle paths that can be associated to simple closed curves: 
    \begin{prop}\label{p.homotopictotrivialsimplecase}
    Let $\gamma$ be a simple closed and good polygonal curve associated to the rectangle path $R_0,..,R_n$. 
    \begin{itemize}
        \item If $R_0=R_n$, then $R_0,...,R_n$ is homotopic to the trivial rectangle path.
        \item If $R_n$ is a successor or  predecessor of generation $k$ of $R_0$ and $R_0=r_0,r_1,..r_k=R_n$ is the unique monotonous rectangle path from $R_0$ to $R_n$, then $R_0,...,R_n$ is homotopic to $r_0,...,r_k$. 
    \end{itemize}
    \end{prop}
    \item Proposition \ref{p.homotopictotrivialsimplecase} implies the general case 
\end{enumerate}

\subsubsection*{Proof of Proposition \ref{p.homotopictotrivialsimplecase}}
We will prove  Proposition \ref{p.homotopictotrivialsimplecase} by induction. 
Let $\gamma$ be a good closed polygonal curve associated to $R_0,...,R_n$ and $M$ be the number of boundary arc points in the interior of $\gamma$ (i.e. the closure of the bounded connected component of $\clos{\mathcal{P}}-\gamma$). 

Notice here that $M<\infty$. Indeed, by compactedness, we can cover the interior of $\gamma$ by a finite number of rectangles of our Markovian family $\clos{\mathcal{R}}$. By Lemmas  \ref{l.crossingrectanglesnoperiodicpoints},  \ref{l.crossingrectangleswithperiodicpoints} and Remark \ref{r.caracterisationarcpoints}, the number of boundary arc points in the interior of $\gamma$ is finite except if a point of $\clos{\Gamma}$ is in the interior of $\gamma$. This is impossible, since $\gamma$ is good and no loop in $\clos{\mathcal{P}}$ can go around a point in $\clos{\Gamma}$. 

Let $n$ be the length of the good polygonal curve $\gamma$. We will apply an induction on the well ordered set of couples $(M,n)\in \mathbb{N}\times \mathbb{N}$ endowed with the lexicographic order: $(M_1,n_1)<(M_2,n_2)$ if and only if $M_1<M_2$ or $M_1=M_2$ and $n_1<n_2$.  At every step of our induction, we  will perform a finite sequence of homotopies of type $\mathcal{A}, \mathcal{B}$ or $\mathcal{C}$ on $\gamma$ -therefore also on its associated rectangle path-, that will produce a new simple, good and closed polygonal curve $\gamma'$ such that $(M(\gamma'),n(\gamma'))<(M(\gamma),n(\gamma))$
 
Notice that any simple and good polygonal curve has length at least equal to 4. Hence, the smallest such curves have the form of a rectangle. We will therefore initialize our induction by considering the case where $\gamma$  is a rectangle containing no boundary arc points in its interior. 

Let us first fix some notations. By Remark \ref{r.polygonalcurverectassociation}, there exists $0=c_0<c_1<...c_{n+1}=1$ a $(n+2)$-uple such that $\gamma(c_i,c_{i+1})\subset R_i$ and a function $\textit{Rect}_{\gamma,R_0}:[0,1]\rightarrow \lbrace R_0,...,R_n\rbrace$ associated to $c_0<c_1<...c_{n+1}$ sending points of $\gamma$ to rectangles. Fix such a collection of $c_i$. We will assume that we have chosen the $c_i$ as in  Remark \ref{r.choiceci}. Also, by a small abuse of language and notation, for the sake of simplicity if $A$ is a connected subset of $ \gamma([0,1])$, then we will denote $\textit{Rect}_{\gamma,R_0}(\gamma^{-1}(A))$ by  $\textit{Rect}_{\gamma,R_0}(A)$ and we will refer to it as the rectangle path associated to $A$.

\subsubsection*{Initializing the induction} 
Assume that $(M(\gamma),n(\gamma))=(0,4)$. We will show that under this hypothesis $\gamma$ is associated to a rectangle path of the desired form.

 By Remark \ref{r.gamma0},  $\gamma(0)$ is a corner of the rectangle $\gamma$. Let us denote by $u',s,u,s'$, following the order in which $\gamma$ visits them, the four segments contained in $\gamma$.

We will associate a segment of $\gamma$, say $S$, to $0$ if the rectangle path $\textit{Rect}_{\gamma,R_0} (S)$ is trivial and to $1$ if not. We can therefore associate $\gamma$ to an element of $\{0,1\}^4$. Our method for associating rectangle paths to curves (see Proposition \ref{p.curvetorectanglepath}) restricts the elements of $\{0,1\}^4$ that can be associated to $\gamma$. 

\begin{lemm}\label{l.possiblecases}
If $\gamma$ is a rectangle, then the only elements of $\{0,1\}^4$ that can be associated to $\gamma$ are $(0,0,0,0),(0,1,1,1)$, $(0,1,0,0), (1,0, 0,0)$ and $(1,1,1,1)$. If furthermore, the interior of $\gamma$ contains no boundary arc points, $(1,1,1,1)$ and $(0,1,1,1)$ are impossible.
\end{lemm}
\begin{proof}
Without any loss of generality, let us assume  that $u,u'$ are unstable segments and $s,s'$ are stable segments. 

If $(u',s)$ is associated to $(0,0)$, then Remark \ref{r.choiceci} implies that $u'$ and $s$ do not exit $R_0$ and since $R_0$ is trivially bifoliated $s',u\subset R_0$. Therefore, in this case, since $\gamma$ is a rectangle, $\gamma$ does not exit $R_0$ at all and is associated to $(0,0,0,0)$. 

Suppose now that $(u',s)$ is associated to $(1,1)$. We can assume without any loss of generality that the (oriented by $\gamma$) segments $u'$ and $s$ are positively oriented (for our choice of orientation of $\clos{\mathcal{F}^{s,u}}$). In this case, Remark \ref{r.choiceci} implies that $u'$ will exit $R_0$ in order to visit $R_{u_1}$, a $\partial^s_+R_0$-crossing predecessor of $R_0$ (see Figure \ref{f.impossiblecasesc}), next $u'$ will exit $R_{u_1}$ in order to visit $R_{u_2}$, a $\partial^s_+R_{u_1}$-crossing predecessor of $R_{u_1}$, and so on, until it exits $R_{u_{N-1}}$ and visits  $R_{u_N}:=\textit{Rect}_{\gamma,R_0} (u'\cap s)$. Similarly, $s$ will exit $R_{u_N}$ in order to visit $R_{u_{N+1}}$, a $\partial^u_+R_{u_N}$-crossing successor of $R_{u_N}$. Next, $s$ will exit $R_{u_{N+1}}$ in order to visit $\partial^u_+R_{u_{N+1}}$-crossing successor of $R_{u_{N+1}}$, and so on until it exits $R_{u_{m-1}}$ and visits  $R_m:=\textit{Rect}_{\gamma,R_0} (u\cap s)$ (see Figure \ref{f.impossiblecasesc}).  

Notice that $R_{u_m}$ (resp. $R_0$) is a successor of some generation of $R_{u_{N+1}}$ (resp. $R_{u_{N-1}}$). Also, since $u'$ exits $R_{u_{N-1}}$, we have that $R_{u_{N+1}}$ and $R_{u_{N-1}}$ correspond two distinct successors (of some generation) of $R_{u_N}$. Using the previous facts, it is not difficult to see that $\int{R_{u_m}}\cap \int{R_0}=\emptyset$. Consequently, if $u$ does not exit $R_{u_m}$, then the curve $\gamma$ cannot be a rectangle. Hence, as before, $u$ exits $R_{u_m}$ in order to visit a $\partial^s_-R_{u_{m}}$-crossing predecessor $R_{u_{m+1}}$ of $R_{u_m}$. After this, $u$ exits $R_{u_{m+1}}$ in order to visit a  crossing predecessor of $R_{u_{m+1}}$ and so on until it exits $R_{u_{k-1}}$ and visits $R_{u_k}:=\textit{Rect}_{\gamma,R_0} (u\cap s')$. By the same arguments as before, $\int{R_{u_k}}\cap \int{ R_{u_N}}=\emptyset$; hence $s'$ has to exit $R_{u_k}$ in order to reach $\gamma(0)\in R_{u_N}$. We conclude that if $(u',s)$ is associated to $(1,1)$, then $(u',s,u,s')$ is associated to $(1,1,1,1)$. 
\begin{figure}[h]
\centering 
\includegraphics[scale=0.7]{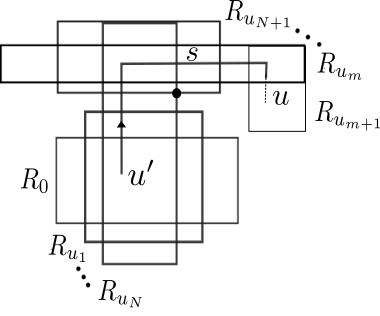}
\caption{}
\label{f.impossiblecasesc}
\end{figure}

If $(u',s)$ is associated to $(1,0)$, then as before $u'$ exits $R_0$ in order to visit $R_{u_1}$, a crossing predecessor of $R_0$, next $u'$ exits $R_{u_1}$ in order to visit a crossing predecessor of $R_{u_1}$ and so on until it exits $R_{u_{N-1}}$ and visits $R_{u_N}:=\textit{Rect}_{\gamma,R_0}(u'\cap s)$ a predecessor of some generation of $R_0$ that contains completely $u'$. Since $s$ by hypothesis does not exit $R_N$ and $R_N$ is trivially bifoliated, $u,s'$ do not exit $R_N$ either. Therefore, if $(u',s)$ is associated to $(1,0)$, then $(u',s,u,s')$ is associated to $(1,0,0,0)$.

We show in a similar way that if $(u',s)$ is associated to $(0,1)$, then $(u',s,u,s')$ is associated to $(0,1,0,0)$ or $(0,1,1,1)$.

Finally, let us show that if $\gamma$ is associated to an element of $\{0,1\}^4$ with three consecutive $1$s, then the interior of $\gamma$ contains boundary arc points. Assume that $\gamma$ is associated to $(1,1,1,1)$ (the case $(0,1,1,1)$ follows by the same argument). Following our previous notations (see the case $(u',s)$ associated to $(1,1)$), $u'\subset R_{u_N}$, $R_{u_{N+1}}$ is a $\partial^u_+R_{u_N}$-crossing successor of $R_{u_N}$ and $s$ exits $R_{u_N}$. By Remark \ref{r.incompletedomain}, the domain of $u'$ and $s$ coincide with the  interior of $\gamma$, therefore the interior of $\gamma$ is trivially bifoliated and contains the  $\partial^u_+R_{u_N}\cap \partial^s_-R_{u_{N+1}}$ (see Figure \ref{f.impossiblecasesc}), which is a boundary arc point according to Remark \ref{r.caracterisationarcpoints}. 
\end{proof}
We now ready to describe the initialization of our induction. Recall that $\gamma$ is associated to the rectangle path $R_0,...,R_n$. Since $\gamma$ is a rectangle and contains no boundary arc points in its interior, we have the following 3 cases:  

\begin{enumerate}
    \item If $\gamma$ corresponds to $(0,0,0,0)$ by our proof of Lemma \ref{l.possiblecases}, we have that $\gamma$ never exits $R_0$, it is therefore associated to the trivial rectangle path. 
    \item If $\gamma$ corresponds to $(1,0,0,0)$ by our proof of Lemma \ref{l.possiblecases} (see the case where $(u',s)$ is associated to $(1,0)$), $\gamma$ is associated to a decreasing rectangle path of the form $R_0,...,R_{u_N}$. 
    \item Similarly, to the previous case, if  $\gamma$ corresponds to $(0,1,0,0)$, then $\gamma$ is associated to an increasing rectangle path. 
\end{enumerate}
We deduce that if  $(M(\gamma),n(\gamma))=(0,4)$, then Proposition \ref{p.homotopictotrivialsimplecase} is true. Let us now describe the step of our induction. 

\subsubsection*{Induction step: the case where $\gamma$ is a rectangle with $M(\gamma)>0$}

Following our previous notations, by our proof of Lemma \ref{l.possiblecases}, we have that $\gamma$ is associated to either $(1,1,1,1)$ or $(0,1,1,1)$. Suppose that $\gamma$ is associated to $(1,1,1,1)$. Assume once again that the (oriented by $\gamma$) segments $u'$ and $s$ are positively oriented (for our choice of orientation of $\clos{\mathcal{F}^{s,u}}$). 
\begin{figure}
    \centering
    \includegraphics[scale=0.7]{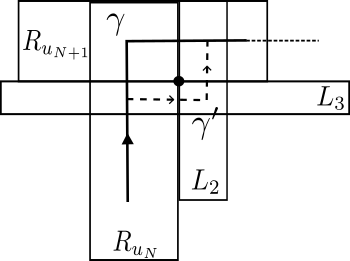}
    \caption{}
    \label{f.caserectanglehomtypec}
\end{figure}

During our proof of Lemma \ref{l.possiblecases}, we associated to $u'$ and $s$ a generalized rectangle path (see Definition \ref{d.generalisedrectpaths})  $\mathcal{R}_{u',s}:= R_0,R_{u_1},...,R_{u_N}, R_{u_{N+1}},...,R_{u_m}$. Notice that by construction, the rectangle path associated to $\mathcal{R}_{u',s}$ corresponds to the first part of the rectangle path associated to $\gamma$. Also, we showed that $u'\subset R_{u_N}$, $s$ exits $R_{u_N}$ in order to visit $R_{u_{N+1}}$ and that $\partial^u_+R_{u_N}\cap \partial^s_-R_{u_{N+1}}$ is a boundary arc point contained in the interior of $\gamma$ (see Figure \ref{f.impossiblecasesc}). Consider the negative cycle $(R_{u_N},R_{u_{N+1}},L_2,L_3,L_4)$ around the previous boundary arc point starting from  $R_{u_N}$. Using the previous cycle, we can  deform $\gamma$ by a homotopy of type $\mathcal{C}$ to a good and simple polygonal curve $\gamma'$ containing strictly less boundary arc points in its interior (see Figure \ref{f.caserectanglehomtypec}).  

We treat the case where $\gamma$ is associated to $(0,1,1,1)$ in a similar way.

\subsubsection*{Induction step in the general case} We will assume here that $\gamma$ is not a rectangle. In this case, our algorithm begins by choosing one tangency $s$ (with respect to the stable or unstable foliation) of type (a) or (b) with a complete domain $D$ (see Lemma \ref{l.canonicalneighbourhoodtangency}) such that we have the two following conditions:
\begin{itemize}
    \item $\gamma(0)\notin D$ or $\gamma(0)$ is a corner point of the rectangle $D$ that is not in $s$
    \item the interior of one side of $\partial D$ is disjoint from $\gamma$
\end{itemize}
We will call a tangency with the above properties, a tangency with property $(\star)$. If furthermore, $\gamma(0)\notin D$ we will say that the tangency satisfies the strong $(\star)$ property. 

\begin{lemm}\label{l.existencetangenciespropstar}
Let $\gamma$ be a good, simple and closed polygonal curve of length strictly greater than 4. There exists $s$ a tangency of $\gamma$ (with the respect to the stable or unstable foliation) with property $(\star)$. 
\end{lemm}
\begin{proof}Indeed, take $s$ to be a stable tangency of $\gamma$ of type (a) or (b) (the existence of at least 2 of such tangencies is assured by Lemma \ref{l.numbertangencies}). Suppose that the domain of $s$, that we will denote by $D$, is not complete. In that case, by Remark \ref{r.incompletedomain}, there exists a finite number of (stable) segments of $\gamma$ in the interior of the stable boundary of $D$ that is not $s$. Take $s'$ one such segment. Since $\int{D}\cap \gamma=\emptyset$ and $D$ is incomplete, $s'$ is a tangency of type (c) or (d). Therefore, to every tangency of type (a) or (b) with an incomplete domain we can associate at least one tangency of type (c) or (d). We deduce by Lemma \ref{l.numbertangencies} that there exist at least two tangencies of type (a) or (b) for $\gamma$, whose domains are complete. 


Consider now $s$ a tangency (with respect to the stable foliation) of type (a) or (b) with a complete domain $D$. The function $\gamma:\mathbb{S}^1\rightarrow \clos{\mathcal{P}}$ induces a cyclic order on the segments forming $\gamma$. Denote by $u$ (resp. $u'$) the unstable segment of $\gamma$ after (resp. before) $s$ (see Figure \ref{f.reduction1case}) and $s'$ the stable boundary component of $D$ that is not $s$. By  definition, the domain of $s$ satisfies,  $\int{D}\cap \gamma=\emptyset$, hence if  $\gamma(0)\in D$ then $\gamma\in \partial D$. 

If $\gamma(0)\in D$, then by Remark \ref{r.gamma0}, $\gamma(0)$ cannot belong to the interior of $s$, $u$ or $u'$. Therefore, $\gamma(0)\in s'$ or $\gamma(0)\in \partial s$. Since there are at least 2 tangencies (with respect to the stable foliation) of type (a) or (b) with a complete domain, we can  assume without any loss of generality that $\gamma(0) \notin s$. 

Assume now that $\int{s'}$ intersects $\gamma$. We can assume without any loss of generality that $u\subset D$ (since $D$ is complete). Since $\gamma$ is good, any two stable segments of $\gamma$ do not belong to the same stable leaf of $\clos{\mathcal{F}}^s$. Therefore, we have necessarily that the segment of $\gamma$ after $u$, say $S$, is contained in $s'$. Since $\gamma$ is not a rectangle $S \subsetneq s'$; we are therefore in the situation of Figure  \ref{f.reductioncase1bis}. Hence, $u$ is a tangency (with respect to the unstable foliation) of type (a) or (b) with a complete domain.  If $\gamma(0)\notin u$, then $u$ has property $(\star)$. If $\gamma(0)\in u$, then take $s_{fin}\neq s$ to be another stable tangency of $\gamma$ with a complete domain (we showed that there are at least 2) and  notice that $\gamma(0)\notin s_{fin}$. If  $s_{fin}$ doesn't satisfy $(\star)$, then  repeat our argument for $s$ of this paragraph in order to construct an unstable tangency with propery $(\star)$. 
\end{proof}
\begin{figure}[h]
\centering 
\includegraphics[scale=0.5]{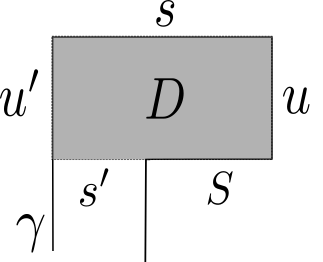}
\caption{}
\label{f.reductioncase1bis}
\end{figure}
Even when $\gamma$ is not a rectangle, it is not always possible to find a tangency with strong property $(\star)$. However, there exist  cases where this is possible:
\begin{lemm}\label{l.existencetangenciesgood}
Let $\gamma$ be a simple, closed and good polygonal curve that is not a rectangle. Assume that $\gamma(0)$ belongs to a tangency (with respect to the stable or unstable foliation) of type (a) or (b). Then there exists a tangency with strong property $(\star)$. 
\end{lemm}
\begin{proof}
Indeed, assume without any loss of generality that $\gamma(0)\in S$, where $S$ is a tangency with respect to the stable foliation of type (a) or (b). By Lemma \ref{l.numbertangencies}, there exists $s$ is a  tangency with respect to the stable foliation with of type (a) or (b) such that $\gamma(0)\notin s$. By our proof of Lemma \ref{l.existencetangenciespropstar}, we can assume that $s$ has a complete domain $D$. Notice that $D$ cannot contain $\gamma(0)$, since a domain of a (stable) tangency of type (a) or (b) cannot intersect a (stable) tangency of type (a) or (b), except when $\gamma$ is a rectangle. By our proof Lemma \ref{l.existencetangenciespropstar}, either $s$ has strong property $(\star)$ and we get the desired result or we are in the case of Figure \ref{f.reductioncase1bis}. In this case, using the notations of Figure \ref{f.reductioncase1bis}, the segment $u$ has property $(\star)$ and its domain $D'\subset D$ does not contain $\gamma(0)$. 
\end{proof}

We will now describe our induction step in the following two subcases:
\begin{enumerate}
    \item there exists $s$ a tangency with strong property $(\star)$ 
    \item no such tangency exists
\end{enumerate}
\subsubsection*{Induction step in the case where a tangency with strong $(\star)$ property exists}

Assume without any loss of generality that the stable tangency $s$ of $\gamma$ has strong property $(\star)$. The function $\gamma:[0,1]\rightarrow \clos{\mathcal{P}}$ endows the segments forming $\gamma$ with a total order, for which $s$ is neither the first nor the last segment (since $\gamma(0)\notin s$). Let us denote by $U'$ (resp. $U$) the unstable segment of $\gamma$ before (resp. after) $s$ and by $s'$ the stable side of the domain of $s$, say $D$,  that is not $s$. We define $u':=U'\cap D$ and $u:=U\cap D$. We are therefore -up to a change of orientation and up to interchanging $u'$ and $u$- in the case of Figure \ref{f.reduction1case}. 
\begin{figure}[h]
\centering 
\includegraphics[scale=0.55]{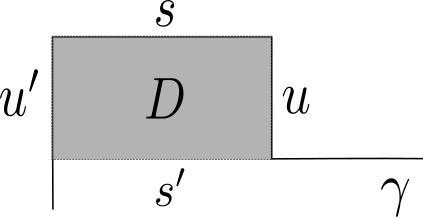}
\caption{}
\label{f.reduction1case}
\end{figure}


As in the previous case, we will associate a segment of $\gamma$, say $S$, to $0$ if $\textit{Rect}_{\gamma,R_0} (S)$ consists of one rectangle and to $1$ if not. We can therefore associate $U',s,U$ to an element of $\{0,1\}^3$. 

We are now ready to describe the induction step in this case: 
\begin{enumerate}[(1{$'$})]
    \item If $(U',s,U)$ is associated to $(0,0,0)$, then  $\textit{Rect}_{\gamma,R_0}(U')=\textit{Rect}_{\gamma,R_0}(U)=\textit{Rect}_{\gamma,R_0}(s)=\{R\}$ and by a homotopy of type $\mathcal{A}$ we can push $s$ along $D\subset R$ in order that $s$ be identified with $s'$ (see Figure \ref{f.reduction1case}). Hence, $\gamma$ is homotopic to a simple, good and closed polygonal curve of strictly smaller length.
    \begin{figure}[h]
  \begin{minipage}[ht]{0.4\textwidth}
    \hspace{0.7cm}   
    \includegraphics[width=0.6\textwidth]{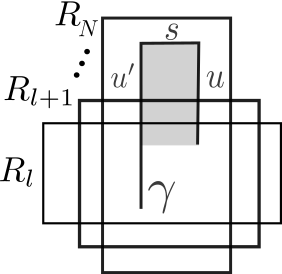}
     
    \caption*{(a)}
    
  \end{minipage}
 \begin{minipage}[ht]{0.4\textwidth}

    \includegraphics[width=0.6\textwidth]{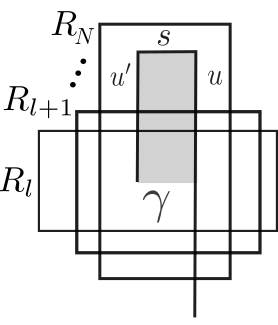}
    
    \caption*{\hspace{-2cm}(b)}
    
  \end{minipage}
\caption{}
\label{f.impossiblecasesab}
\end{figure}
    \item If $(U',s,U)$ is associated to $(1,0,1)$, we are in the case of case of Figure \ref{f.impossiblecasesab}b, where $\{R_l,...,R_N\}=\textit{Rect}_{\gamma,R_0}(U')$, $N>l$ and $u'=U'$. By pushing $\gamma$ along $D$ so that $s$ be identified with $s'$, we can construct a simple, good and closed polygonal curve $\gamma'$ of strictly smaller length.

    Let us show that the previous movement corresponds to a homotopy of type $\mathcal{A}$. Indeed, our  movement pushes the  segments $u',s,u$ of $\gamma$ to the segment $s'$ of $\gamma'$, while keeping $
\gamma-(u'\cup s \cup u)$ fixed. Let $a,b\in [0,1]$ such that $\gamma([a,b])=U'\cup s \cup u$. By eventually reparametrizing $\gamma'$ (this corresponds to a homotopy of type $\mathcal{A}$), we can assume that $\gamma'([a,b])=s'$. Similarly to $\gamma$ we define the function $\textit{Rect}_{\gamma',R_0}$. Notice that the rectangle paths associated to $\gamma_{|[0,a]}\equiv \gamma'_{|[0,a]} $ by $\textit{Rect}_{\gamma',R_0}$ and $\textit{Rect}_{\gamma,R_0}$ are identical, that  $\textit{Rect}_{\gamma',R_0}(s')=\textit{Rect}_{\gamma,R_0}(a)=\{R_l\}$ and that both $u_{\text{after}}:=U-u\subset \gamma'$ and $U\subset \gamma$ will exit $R_N$ from one of its stable boundaries in order to enter a crossing  predecessor of $R_N$, therefore $\{R_l,R_{l+1},...,R_N \} \subset \textit{Rect}_{\gamma',R_0}(u_{\text{after}})$. Since $\gamma$ and $\gamma'$ coincide on $[b,1]$, we deduce that $\gamma'$ is associated to the same rectangle path as $\gamma$ and our movement corresponds indeed to a homotopy of type $\mathcal{A}$.  
    \item If $(U',s,U)$ is associated to $(0,0,1)$, we are in the case of  Figure \ref{f.impossiblecasesab}b, where $N=l$ and $U'=u'$. As in case $(2')$, by pushing $\gamma$ along $D$ so that we erase $U'$ and thus performing a homotopy of type $\mathcal{A}$, we get  that $\gamma$ is homotopic to a simple,  good and closed polygonal curve of strictly smaller length. 
    \item If $(U',s,U)$ is associated to $(1,1,0)$, we are in the case of Figure \ref{f.impossiblecasescsecond}, where $\{R_l,...,R_N\}=\textit{Rect}_{\gamma,R_0}(U')$,  $\{R_N,R_{N+1}
,...,R_m\}=\textit{Rect}_{\gamma,R_0}(s)$, $u=U$ and $l<N<m$. By pushing $\gamma$ along $D$ so that we erase $U$ and thus performing a homotopy of type $\mathcal{A}$, we obtain that $\gamma$ is homotopic to a simple, good and closed polygonal curve of strictly smaller length. 
\begin{figure}[h]
\centering 
\includegraphics[scale=0.5]{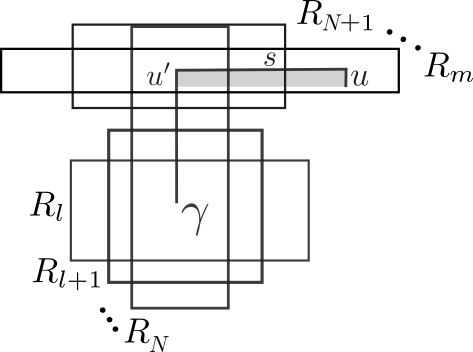}
\caption{}
\label{f.impossiblecasescsecond}
\end{figure}
\item If $(U',s,U)$ is associated to $(0,1,0)$, we are again in the case of Figure \ref{f.impossiblecasescsecond}, with $m>N=l$. As in case $(4')$, by pushing $\gamma$ along $D$ so that we erase $U$ or $U'$ and thus performing a homotopy of type $\mathcal{A}$, we obtain that $\gamma$ is homotopic to a simple, good and closed polygonal curve of strictly smaller length. 
\item Assume now that $(U',s,U)$ is associated to $(1,1,1)$. Denote by $x_{su}$ (resp. $x_{su'}, x_{s'u'}$) the unique point of intersection of $s$ and $u$ (resp. $s$ and $u'$, $s'$ and $u'$), by $R_{u'}$ and $R_s$ the rectangles $\textit{Rect}_{\gamma,R_0}(x_{su'})$ and $\textit{Rect}_{\gamma,R_0}(x_{su})$ respectfully. Notice that since $\textit{Rect}_{\gamma,R_0}(s)$ is not trivial, $R_s\neq R_{u'}$ (see Figure \ref{f.impossiblecase1}). Also, since $\textit{Rect}_{\gamma,R_0}(U)$ is not trivial $U$ will exit $R_s$ in order to enter a crossing predecessor of $R_s$, say $R_u$. By the same argument, by following $U'$ negatively starting from $x_{su'}$, $U'$ must exit $R_s$. We are therefore in the case of Figure \ref{f.impossiblecase1}. As in the case described in \textit{Induction step: the case where $\gamma$ is a rectangle with $M(\gamma)>0$}, by performing a homotopy of type $\mathcal{C}$ inside $D$, we get that $\gamma$ is homotopic to a simple, good and closed polygonal curve with strictly less boundary arc points in its interior. 
\begin{figure}
    \centering
    \includegraphics[scale=0.5]{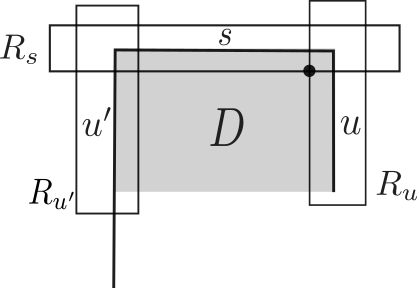}
    \caption{$D$ contains boundary arc points}
    \label{f.impossiblecase1}
\end{figure}

\item If $(U',s,U)$ is associated to $(0,1,1)$, by a similar argument as in the previous case we can show that 
\begin{itemize}
    \item If $U'\not\subset R_s:= \textit{Rect}_{\gamma,R_0}(x_{su})$, we are in the case of Figure \ref{f.impossiblecase1} (with $U'\subset R_{u'}$). Again, by performing a homotopy of type $\mathcal{C}$ inside $D$, we get that $\gamma$ is homotopic to a simple, good and closed polygonal curve with strictly less boundary arc points in its interior.  
    \item If $U'\subset R_s$, then we are in the case of Figure \ref{f.case011}, where $\{R_l,...,R_N\}=\textit{Rect}_{\gamma,R_0}(s)$,  $\{R_N,R_{N+1}
,...,R_m\}=\textit{Rect}_{\gamma,R_0}(U)$ and $u'=U'$. In this case, by pushing $\gamma$ along $D$ so that we erase $U'$ and thus performing a homotopy of type $\mathcal{A}$, we get  that $\gamma$ is homotopic to a simple,  good and closed polygonal curve of strictly smaller length.
     
\end{itemize}

\begin{figure}[h]
\centering 
\includegraphics[scale=0.55]{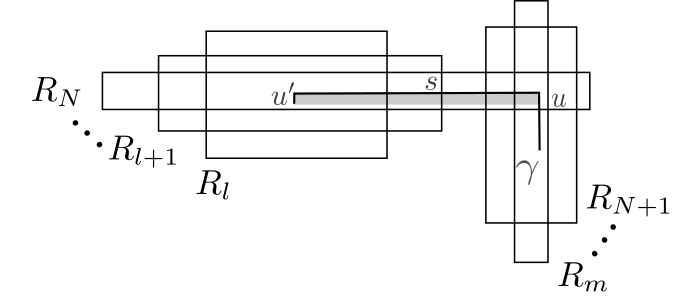}
\caption{}
\label{f.case011}
\end{figure}
    \item If $(U',s,U)$ is associated to $(1,0,0)$, we are in the case of case of Figure \ref{f.impossiblecasesab}a, where $\{R_l,...,R_N\}=\textit{Rect}_{\gamma,R_0}(u')$ and $N>l$. Let $A=\{R_l,...,R_N\}$ and  $a,b \in [0,1]$ such that $\gamma([a,b])=u'\cup s \cup u$.
    Before stating the induction step in this case, let us first state a very useful lemma: 
    
    \begin{lemm}\label{l.increasingrectanglepaths} 
    Let $B:=L_1,...,L_k$ be a decreasing rectangle path. Consider any rectangle path of the form $L_p,w_B,L_{p'}$, where $p,p'\in \llbracket 1,k\rrbracket$ and $w_B$ is a rectangle path formed by rectangles in $B$. If $p'>p$ (resp. $p>p'$, $p=p'$), the previous rectangle path is homotopic by a sequence of homotopies of type  $\mathcal{B}$ to a decreasing (resp. increasing, trivial) rectangle path. 
    \end{lemm}
    The above lemma can be easily proven by an induction on the length of $w_B$; we will therefore omit its proof.  
    
    \begin{enumerate}
        
        \item Assume that $\gamma_{|[b,1]}$ eventually exits $\underset{i=l}{\overset{N}{\cup}} R_i$ in order to visit a crossing predecessor (or successor) of $R_{k}\in A$. By pushing $\gamma$ along $D$ so that $s$ be identified with $s'$,  we obtain a simple, good and closed curve $\gamma'$ of strictly smaller length. Let us show that the previous movement corresponds to a sequence of homotopies of type $\mathcal{A}$ or $\mathcal{B}$.

    \par{\quad{} As in the case $(2')$, our  movement pushes the  segments $u',s,u$ of $\gamma$ to the segment $s'$ of $\gamma'$, while keeping $
    \gamma-(u'\cup s \cup u)$ fixed. By eventually reparametrizing $\gamma'$ (this corresponds to a homotopy of type $\mathcal{A}$), assume that $\gamma'([a,b])=s'$. Therefore, $\gamma\equiv \gamma'$ on $[0,a]\cup [b,1]$ and the rectangle paths (starting from $R_0$) associated to $\gamma_{|[0,a]}\equiv \gamma'_{|[0,a]} $ are identical. Let $R_m$ be the last rectangle in the previous rectangle paths. By eventually changing our initial choice of $c_i$ (see our discussion prior to \textit{Initializing the induction}), therefore  also the function $\textit{Rect}_{\gamma,R_0}$, we can assume that $R_m=R_l$. It suffices to show that the rectangle paths (starting from $R_l$) associated to $\gamma_{|[a,1]}$ and $\gamma'_{|[a,1]}$ are homotopic by a sequence of homotopies of type $\mathcal{A}$ and $\mathcal{B}$. 

    Since $\gamma([a,b])=u'\cup s \cup u \subset \underset{i=l}{\overset{N}{\cup}} R_i$,  $\gamma'([a,b])=s'\subset  R_l$ and by hypothesis $\gamma'([b,1])\equiv \gamma([b,1])$ remain in 
    $\underset{i=l}{\overset{N}{\cup}} R_i$ until they  exit this region in order to enter a crossing predecessor (or successor) of $R_k$, we have that there exists $R\notin A$ a predecessor (resp. successor) of $R_k$ and $W$ a rectangle path such that the rectangle paths (starting from $R_l$) associated to $\gamma'_{|[a,1]} $ and  $\gamma_{|[a,1]} $ are of the form  $R_l,w'_A,R_{k},R, W$ and $R_l,w_A,R_{k},R,W$ respectively, where $w'_A, w_A$ are words formed by rectangles in $A$. Notice that the previous two rectangle paths end in the same way since $\gamma\equiv \gamma'$ on $ [b,1]$. We conclude by remarking that $R_l,w_A,R_{k}$ and $R_l,w'_A,R_{k}$ are homotopic by a sequence of homotopies of type $\mathcal{B}$, thanks to  Lemma \ref{l.increasingrectanglepaths}.}
        
       \vspace{0.5cm}  
    \item Assume finally that    $\gamma([a,1])\subset \underset{i=l}{\overset{N}{\cup}} R_i$. In this case, we will show directly that the rectangle path associated to $\gamma$ is homotopic to a trivial or monotonous rectangle path. 
    
    \quad{}Consider $\gamma'$ the curve obtained by pushing $\gamma$ along $D$ so that $s$ be identified with $s'$. The curve $\gamma'$ is simple, good, closed and of strictly smaller length than $\gamma$. By eventually reparametrizing $\gamma'$, assume once again that $\gamma\equiv \gamma'$ on $[0,a]\cup[b,1]$; hence the rectangle paths (starting from $R_0$) associated to $\gamma_{|[0,a]}\equiv \gamma'_{|[0,a]} $ are identical. As before, we will assume that the last rectangle of the previous rectangle paths is $R_l$. Since $\gamma([a,1]),\gamma'([a,1]) \subset \underset{i=l}{\overset{N}{\cup}} R_i$, there exist $R_{k'},R_k\in A$ and $w_A,w'_A$ words formed by rectangles in $A$ such that the rectangle paths (starting from $R_l$) associated to  $\gamma'_{|[a,1]} $ and  $\gamma_{|[a,1]}$ are respectively of the form  $R_l,w'_A,R_{k'}$ and $R_l,w_A,R_{k}$. Notice that $\int{R_k}\cap \int{R_{k'}}\neq \emptyset$, since $\gamma(1)=\gamma'(1)$.

    \quad{}Assume without any loss of generality that $R_{k'}$ is a predecessor of some generation of $R_0$ (recall that $\gamma(0)=\gamma(1)\in \int{R_0}$) and apply the induction hypothesis for $\gamma'$, which verifies $(M(\gamma'),n(\gamma'))<(M(\gamma),n(\gamma))$. We have that the rectangle path associated to $\gamma'$ is homotopic to the decreasing path  $R_0=r_0',r'_1,...,r'_p=R_{k'}$. Consider the rectangle path associated to $\gamma$, which is of the form:  $$R_0,R_1,...,R_l,w_A,R_{k}$$ By Lemma \ref{l.increasingrectanglepaths}, the previous rectangle path is homotopic by a sequence of homotopies of type $\mathcal{B}$ to the rectangle path $$R_0,R_1,...,R_l,w'_A,R_{k'},w'^{-1}_A, R_l, w_A,R_{k}$$ where $w'^{-1}_A$ is the rectangle path $w_A'$ followed from its end to its beginning. The first part of this rectangle path is by hypothesis homotopic to the decreasing rectangle path $R_0=r_0',r'_1,...,r'_p=R_{k'}$. Therefore, the previous rectangle path is homotopic to $$R_0=r_0',r'_1,...,r'_p=R_{k'},w'^{-1}_A, R_l, w_A,R_{k}$$ Recall that the second part of this rectangle path consists only of rectangles in $A$ and that $\int{R_k}$ intersects $\int{R_{k'}}$. If $R_k=R_{k'}$, then by Lemma \ref{l.increasingrectanglepaths} we have that the rectangle path associated to $\gamma$ is homotopic to the decreasing path $R_0,r'_1,...,r'_{p}=R_{k}$, which gives us the desired result. If $R_{k}$ is a predecessor of some generation of $R_{k'}$, then by Lemma \ref{l.increasingrectanglepaths}, $R_{k'},w'^{-1}_A, R_l, w_A,R_{k}$ is homotopic to the decreasing rectangle path $r'_p=R_{k'}, r'_{p+1},...,r'_{p+s}=R_k$. We deduce that the rectangle path associated to $\gamma$ is homotopic to the decreasing rectangle path $R_0=r_0',r'_1,...,r'_p,r'_{p+1},...,r'_{p+s}=R_k$, which gives us the desired result. Finally, assume that $R_k$ is a successor of some generation of $R_{k'}$. By the same argument the rectangle path associated to $\gamma$ is homotopic to the concatenation of a decreasing (from $0$ to $p$) and an increasing (from $p$ to $(p+s)$) rectangle path $R_0,r'_1,...,r'_p=R_{k'},r'_{p+1},...,r'_{p+s}=R_k$. However, since $\int{R_k}\cap \int{R_0}\neq \emptyset$, we get that $r'_{p+1}=r'_{p-1}$. Therefore, by a homotopy of type $\mathcal{B}$, we can decrease the length of the previous path and keep doing this until we obtain an increasing,  decreasing or trivial rectangle path, which gives us the desired result.  
     \end{enumerate}  
\end{enumerate}


\vspace{0.5cm}
\subsubsection*{Induction step in the case where no tangency with strong $(\star)$ property exists}

Assume that the domains of all (stable and unstable) tangencies with property $(\star)$ contain $\gamma(0)$.

Consider $s$ a tangency with property $(\star)$ and $D$ its domain (its existence is ensured by Lemma \ref{l.existencetangenciespropstar}). We will assume without any loss of generality that $s$ is a stable tangecny. Let us  denote once again by $U'$ (resp.$U$) the unstable segment of $\gamma$ before (resp. after) $s$, by $s'$ the stable side of $D$ that is not $s$, $u':=U'\cap D$ and $u:=U\cap D$. We are thus -up to a change of orientation and up to interchanging $u'$ and $u$- in the case of Figure \ref{f.reduction1case}. Recall that by definition of property $(\star)$, $\gamma(0)\notin s$, therefore $\gamma(0)\in s' \cap U'$ or $\gamma(0)\in s' \cap U$.

\begin{lemm}\label{l.reductiontofirstsubcase} 
Let $R_0$ and $R_n$ be the first and last rectangles of the rectangle path associated to $\gamma$. If $\gamma(0)\in U'$ (resp. $\gamma(0)\in U$), $\textit{Rect}_{\gamma,R_0} (U')$ (resp. $\textit{Rect}_{\gamma,R_0} (U)$) is trivial and $U'\subset R_n$ (resp. $U\subset R_0$), then $\gamma$ is homotopic by a homotopy of type $\mathcal{A}$ to a simple, closed and good polygonal curve having at least one tangency with strong property $(\star)$. 
\end{lemm}
\begin{proof}
Indeed, by a homotopy of type $\mathcal{A}$ we can reparametrize $\gamma$ so that $\gamma(0)\in s$. We obtain the desired result by Lemma \ref{l.existencetangenciesgood}.
\end{proof}

We are now ready to describe the induction step in this final case. 

\begin{enumerate}[(1{$''$})]
    \item If $(U',s,U)$ is associated to $(0,0,0)$, by our argument in case $(1')$  we can push $\gamma$ by a homotopy of type $\mathcal{A}$ along $D$ so that $s$ comes arbitrarily close to $s'$. By Lemma \ref{l.reductiontofirstsubcase} a reparametrization of $\gamma$ by a homotopy of type $\mathcal{A}$ has a tangency with strong property $(\star)$. We deduce the induction step by applying cases $(1')-(8')$.
    \item If $(U',s,U)$ is associated to $(1,0,1)$, we are in the case of Figure \ref{f.impossiblecasesab}b, with $u'=U'$ and therefore $\gamma(0)\in u'$. By our argument in case (2$'$), we can push $\gamma$ by a homotopy of type $\mathcal{A}$ along $D$ so that $s$ comes arbitrarily close to $s'$. By doing so  $\textit{Rect}_{\gamma,R_0}(u')$ becomes trivial and therefore by Lemma \ref{l.reductiontofirstsubcase} up to a homotopy of type $\mathcal{A}$, $\gamma$ has a tangency with strong property $(\star)$. We deduce the induction step by applying cases $(1')-(8')$.
    \item If $(U',s,U)$ is associated to $(0,0,1)$, we are in the case of  Figure \ref{f.impossiblecasesab}b, where $N=0$. We apply the same argument as in case $(2'')$. 
    \item If $(U',s,U)$ is associated to $(1,1,0)$, then we are in the case of Figure \ref{f.impossiblecasescsecond}, where $\{R_l,...,R_N\}=\textit{Rect}_{\gamma,R_0}(u')$,  $\{R_N,R_{N+1}
,...,R_m\}=\textit{Rect}_{\gamma,R_0}(s)$, $u=U$ and $l<N<m$. By the same arguments as before, up to a homotopy of type $\mathcal{A}$, $\gamma$ has a tangency with strong property $(\star)$. We deduce the induction step by applying cases $(1')-(8')$.
    \item If $(U',s,U)$ is associated to $(0,1,0)$, then either $\gamma(0)\in u\cap s'$ or $\gamma(0)\in u'\cap s'$. In both cases, by the same arguments as before, up to a homotopy of type $\mathcal{A}$, $\gamma$ has a tangency with strong property $(\star)$. We deduce the induction step by applying cases $(1')-(8')$.
\item If $(U',s,U)$ is associated to $(1,1,1)$, then either $\gamma(0)\in u\cap s'$ or $\gamma(0)\in u'\cap s'$. In both cases, we can apply in the exact same way the homotopy performed in (6$'$). 
\item If $(U',s,U)$ is associated to $(0,1,1)$, as in the case ($7'$):
\begin{itemize}
    \item If $U'\not\subset \textit{Rect}_{\gamma,R_0}(x_{su})$, we are in the case of Figure \ref{f.impossiblecase1} and we have that $\gamma(0)\in u\cap s'$ or $\gamma(0)\in u'\cap s'$. In both cases, we can apply in the exact same way the homotopy performed in (7$'$).   
    \item If $U'\subset \textit{Rect}_{\gamma,R_0}(x_{su})$, then we are in the case of Figure \ref{f.case011}. In this case, $u'=U'$ and  $\gamma(0)\in u'$. Once again, we can push $\gamma$ by a homotopy of type $\mathcal{A}$ along $D$ so that $s$ comes arbitrarily close to $s'$; hence by Lemma \ref{l.reductiontofirstsubcase}, a reparametrization of $\gamma$ has a tangency with strong property $(\star)$. We deduce the induction step by applying cases $(1')-(8')$.
     
\end{itemize}
    \item If $(U',s,U)$ is associated to $(1,0,0)$, we are in the case of case of Figure \ref{f.impossiblecasesab}a, where $\{R_l,...,R_N\}=\textit{Rect}_{\gamma,R_0}(u')$ and $N>l$.  Let $A=\{R_l,...,R_N\}$ and $a,b\in[0,1]$ such that $\gamma([a,b])=u'\cup s \cup u$. Assume first that $\gamma(0)\in U' \cap s' $. 
    \begin{enumerate}
        
        \item As in case ($8'$), if $\gamma_{|[b,1]}$ eventually exits $\underset{i=l}{\overset{N}{\cup}} R_i$ in order to visit a crossing predecessor (or successor) of $R_{k}\in A$. By our argument in  case (8$'$a) we can push $\gamma$ along $D$  by a sequence of homotopies of type $\mathcal{A}$ and $\mathcal{B}$ so that $s$ comes arbitrarily close to $s'$. Hence, by Lemma \ref{l.reductiontofirstsubcase} up to a homotopy,  $\gamma$ has a tangency with strong property $(\star)$. We obtain the induction step by applying the cases $(1')-(8')$.
        
    \item If now  $\gamma([a,1]) \subset \underset{i=l}{\overset{N}{\cup}} R_i$, since $\gamma(0)\in U'\cap s'$, we get that $\gamma([0,1]) \subset \underset{i=l}{\overset{N}{\cup}} R_i$. Therefore, the rectangle path associated to $\gamma$ is a word formed by rectangles in $A$. The previous rectangle path is homotopic to an increasing or decreasing or trivial rectangle path, thanks to Lemma \ref{l.increasingrectanglepaths}. 
    
\item \underline{Assume now that $\gamma(0) \in s' \cap u$}. If $U'\subset D$ then we can apply the exact same argument as in case $(8'b)$. 
    
    \quad{}If not, push $\gamma$ along $D$ so that $s$ comes very close to $s'$. By Lemma \ref{l.reductiontofirstsubcase}, by reparametrizing this new curve we can assume that it has a tangency with strong property $(\star)$. Denote the reparametrized curve by $\gamma'$. By applying our algorithm (cases $(1')-(8')$), $\gamma'$ is homotopic to a simple, good, closed polygonal curve $\gamma''$ of strictly smaller length. By the induction hypothesis the rectangle path associated to $\gamma''$ is homotopic to a monotonous or trivial path. Hence, the same applies to the rectangle path associated to $\gamma'$. By the exact same arguments as in case, $(8'b)$ we can now conclude that the rectangle path associated to $\gamma$ is homotopic to a monotonous or trivial rectangle path. 
\end{enumerate}
\end{enumerate}
This finishes the description of our induction step and the proof of Proposition \ref{p.homotopictotrivialsimplecase}. 
\subsubsection*{Proposition \ref{p.homotopictotrivialsimplecase} implies  Theorem \ref{t.homotopictotrivialpath}} 
Consider $r_0,...,r_n$ a closed rectangle path in $\clos{\mathcal{P}}$ and $\gamma$ a closed and good polygonal curve associated to the previous rectangle path. 

By our definition of good polygonal curve (see Definition \ref{d.polygonalcurve}), we get that $\gamma$ has only a finite number, say $M$, of self-intersections that are also transverse. We will show the desired result by induction on $M$. If $M=0$, then we obtain the desired result from Proposition \ref{p.homotopictotrivialsimplecase}. Suppose now that $M>0$ and that the result stands for all $n\in \llbracket 0, M-1\rrbracket$.

By Remark \ref{r.polygonalcurverectassociation}, there exists $0=c_0<c_1<...c_{n+1}=1$ a $(n+2)$-uple such that $\gamma(c_i,c_{i+1})\subset r_i$ and a function $\textit{Rect}_{\gamma,r_0}:[0,1]\rightarrow \lbrace r_0,...,r_n\rbrace$ associated to $c_0<c_1<...c_{n+1}$ sending points of $\gamma$ to rectangles. We will assume that we have chosen the $c_i$ as in  Remark \ref{r.choiceci}. 

Consider $x_0 \in [0,1]$ the biggest element in $[0,1]$ for which $\gamma_{|[0,x_0)}$ is injective. Take $y\in [0,x_0)$ such that $\gamma(y)=\gamma(x_0)$. Let $r_k=\textit{Rect}_{\gamma,r_0}(y)$, $R_{|[y,x_0]}$ (resp. $R_{|[0,y]}$) be the rectangle path starting from $r_k$ (resp. $r_0$) associated to the simple and 
good polygonal curve $\gamma_{|[y,x_0]}:[y,x_0]\rightarrow \gamma([y,x_0])$ (resp.$\gamma_{|[0,y]}$). By an eventual change of our choice of $c_i$, we may assume that $R_{|[y,x_0]}$ and $R_{|[0,y]}$  coincide with $\textit{Rect}_{\gamma,r_0}([y,x_0])$ and $\textit{Rect}_{\gamma,r_0}([0,y])$ respectfully.  

Since $\gamma_{|[y,x_0]}$ is simple and closed, by Proposition \ref{p.homotopictotrivialsimplecase} its associated rectangle path $\textit{Rect}_{\gamma,r_0}([y,x_0])=\{r_k,...,r_l\}$ is homotopic to an increasing, decreasing or trivial rectangle path, $r_k=R_0,R_{1}...,R_s=r_l$. Similarly, consider $\gamma'$ the good, closed polygonal curve formed by $\gamma([x_0,1])$ followed by $\gamma([0,y])$. Notice that $\gamma'$ has at most $M-1$ self-intersections and therefore by our induction hypothesis its associated rectangle path (starting from $r_l$)  $r_l,...,r_n=r_0,...,r_k$ is homotopic to an increasing, decreasing or trivial rectangle path $r_l=R'_0,R'_1...,R'_m=r_k$.

Consider now the rectangle path associated to $\gamma$: $r_0,...,r_n$. By a sequence of homotopies of type $\mathcal{B}$ the previous rectangle path is homotopic to $$r_0,..r_k,...r_l,..,r_n=r_0,r_1,...,r_{k-1},r_k,r_{k-1},...,r_0$$
By our previous discussion the above rectangle path is homotopic to 
\begin{equation}\label{eq.rectpath}
r_0,..r_k=R_0,R_1...,R_s=r_l=R_0',R'_1...,R'_m=r_k,r_{k-1},...,r_0
\end{equation}

If $r_k=R_0,R_1...,R_s=r_l$ is the trivial rectangle path, then $r_k=r_l$ and $r_l=R'_0,R'_1...,R'_m=r_k$ is also the trivial rectangle path. In this case, $r_0,...,r_n$ is homotopic to the rectangle path $r_0,...,r_{k-1},r_k,r_{k-1}...,r_0$, which is homotopic to a trivial rectangle path by a sequence of homotopies of type $\mathcal{B}$. 

Assume now without any loss of generality that $r_k=R_0,R_1...,R_s=r_l$ is a decreasing rectangle path; hence $r_l$ is a predecessor of some generation of $r_k$. We deduce that $r_l=R_0',R'_1...,R'_m=r_k$ is an increasing rectangle path. Furthermore, since the predecessors (resp. successors) of any rectangle have disjoint interiors, we have that there exists a unique decreasing (resp. increasing) rectangle path from $r_k$ to $r_l$ (resp. from $r_l$ to $r_k$). We deduce that the rectangle paths $r_k=R_0,R_1...,R_s=r_l$ and $r_k=R_m', R_{m-1}'...,R'_1, R'_0=r_l$ are exactly the same. Therefore, by a sequence of homotopies of type $\mathcal{B}$ the rectangle path \ref{eq.rectpath} is homotopic to: 
$$r_0,...r_k,r_{k-1},..,r_0$$
Finally, the above rectangle path is homotopic to the trivial rectangle path by a sequence of homotopies of type $\mathcal{B}$; we thus get 
the desired result. 
\subsection{Proof of Theorem \ref{t.closedrectanglespathscorrespondtoclosedpaths}}
\label{s.proofoftheoremcorrespondancehom}
By following the notations introduced in the beginning of Section \ref{s.homotopiesofpaths}, it suffices to prove that two centered and homotopic rectangle paths in $\clos{\mathcal{P}_1}$ correspond to centered and homotopic rectangle paths in $\clos{\mathcal{P}_2}$. Indeed, let $r^1_0,...,r^1_n$ be a closed and centered rectangle path in $\clos{\mathcal{P}_1}$. By Theorem \ref{t.homotopictotrivialpath}, the previous rectangle path is homotopic to the trivial rectangle path $r_0^1$. If homotopic rectangle paths in $\clos{\mathcal{P}_1}$ correspond to homotopic rectangle paths in $\clos{\mathcal{P}_2}$, then $r^2_0,...,r^2_n$ is homotopic to $r^2_0$. By our discussion, prior to Theorem \ref{t.homotopictotrivialpath}, this implies that  $r^2_0,...,r^2_n$ is closed, which gives us the desired result. Let us now show that homotopic rectangle paths in $\clos{\mathcal{P}_1}$ correspond to homotopic rectangle paths in $\clos{\mathcal{P}_2}$.

Since any homotopy of rectangle paths can be described as a sequence of homotopies of type $\mathcal{A},\mathcal{B}$ or $\mathcal{C}$, it suffices to show that homotopies of type $\mathcal{A}$ (resp.$\mathcal{B}$, $\mathcal{C}$) in $\clos{\mathcal{P}_1}$ correspond to homotopies of type $\mathcal{A}$ (resp. $\mathcal{B}$, $\mathcal{C}$) in $\clos{\mathcal{P}_2}$. Indeed, consider $R_1,R_2$ two centered rectangle paths in $\clos{\mathcal{P}_1}$ that are homotopic by a homotopy of type $\mathcal{A}$. The result is trivially true for homotopies of type $\mathcal{A}$.  

Consider now $R_1,R_2$ two centered rectangle paths in $\clos{\mathcal{P}_1}$ that are homotopic by a homotopy of type $\mathcal{B}$. By Definition \ref{d.homotopyB}, the paths $R_1,R_2$ are of the form  $r^1_0,...,r^1_k,r^1_{k+1},r^1_k,r^1_{k+2},...r^1_n$ and $r^1_0,...,r^1_k,r^1_{k+2},...,r^1_n$. It is not difficult to see that $R_1,R_2$ correspond in $\clos{\mathcal{P}_2}$ to two  rectangle paths of the form $r^2_0,...,r^2_k,r^2_{k+1},r^2_k,r^2_{k+2},...r^2_n$ and  $r^2_0,...,r^2_k,r^2_{k+2},...,r^2_n$. The previous two rectangle paths are homotopic by a homotopy of type $\mathcal{B}$. 

Finally, consider $R_1,R_2$ two centered rectangle paths in $\clos{\mathcal{P}_1}$ that are homotopic by a homotopy of type $C$. Following the notations of Definition \ref{d.homotopyC}, we may assume that $R_1$ and $R_2$ are respectively of the form: $${r^1}_0,...,{r^1}_k,{L^1}_0,{L^1}_{01},...,{L^1}_{0s(0)},{L^1}_1, {L^1}_{11},...,{L^1}_{1s(1)},..., {L^1}_k,{r^1}_l,{r^1}_{l+1},...,{r^1}_n$$  
 $${r^1}_0,...,{r^1}_k,{L^1}_0,{L'^1}_{01},...,{L'^1}_{0s'(0)},{L^1}_3, {L'^1}_{31},...,{L'^1}_{3s'(3)},...,{L^1}_k,{r^1}_l,{r^1}_{l+1},...,{r^1}_n$$
where $({L^1}_0,{L^1}_1,{L^1}_2,{L^1}_3,{L^1}_4)$ and $({L^1}_0,{L^1}_3,{L^1}_2,{L^1}_1,{L^1}_4)$ are the two cycles starting from $L^1_0$ around a boundary arc point $p^1\in L^1_0 - \{\text{corners of $L^1_0$}\} $, $k\in \llbracket 1,4 \rrbracket$,  ${L^1}_0,{L^1}_{01},...$ $,{L^1}_{0s(0)},{L^1}_1, {L^1}_{11},$ $...,{L^1}_{1s(1)},..., {L^1}_k$ is the rectangle path associated to the generalized rectangle path $({L^1}_0,{L^1}_1,...,{L^1}_k)$ and ${L^1}_0,{L'^1}_{01},...,{L'^1}_{0s'(0)},{L^1}_3, {L'^1}_{31},...,{L'^1}_{3s'(3)},...,{L^1}_k$ is the rectangle path associated to the generalized rectangle path $({L^1}_0,{L^1}_3,...,{L^1}_k)$. 

Assume that $r^2_0,...,r^2_k,L^2_0$ is the rectangle path $\clos{\mathcal{P}_2}$ associated to $r^1_0,...,r^1_k,L^1_0$. Let us show now that the cycle $(L^1_0,L^1_1,L^1_2,L^1_3,L^1_4)$ corresponds to a cycle of $\clos{\mathcal{P}_2}$ starting from $L^2_0$. Indeed, assume without any loss of generality that $L_0^1$ contains a germ of the $(-,-)$ and $(+,-)$ quadrants of $p^1$ and that $(L^1_0,L^1_1,L^1_2,L^1_3,L^1_4)$ is a positive cycle (see Figure \ref{r.existenceofcycle}). Let $s^1_0$ be the uppermost stable boundary component of $L_0^1$. Notice that $p^1\in s^1_0$. By our definition of a cycle (see Definition  \ref{d.cyclearcpoint}), $L_1^1$ corresponds to the $s^1_0$-crossing predecessor of $L_0^1$ containing $p$ in its leftmost unstable boundary. Assume now that $L^1_0,L^1_{01},...,L^1_{0s(0)},L^1_1$ corresponds to $L^2_0,L^2_{01},...,L^2_{0s(0)},L^2_1 $ in $\clos{\mathcal{P}_2}$. By the above, $L^2_1$ corresponds to a $s^2_0$-crossing predecessor of $L_0^2$, where $s^2_0$ is the uppermost stable boundary component of $L_0^2$. Furthermore, since $L_1^1$ is not  the leftmost $s^1_0$-crossing predecessor of $L_0^1$ ($L^1_3$ is an $s^1_0$-crossing predecessor of $L_0^1$ at the left of $L_1^1$, see Figure \ref{r.existenceofcycle}), neither is $L_1^2$. We conclude that $p^1$ corresponds to the unique point of intersection, say $p^2\in \clos{\mathcal{P}_2}$, of $s^2_0$ with the leftmost unstable boundary component of $L_1^2$. Notice that $p^2$ is not a corner point of $L_0^2$ and $L_0^2$ contains a germ of the $(-,-)$ and $(+,-)$ quadrants of $p^2$.

Similarly, $L^1_2$ is associated in $\clos{\mathcal{P}_2}$ to a crossing successor of $L_1^2$, say $L_2^2$, such that $p^2\in \partial L_1^2 \cap \partial L_2^2$ and $\int{L_0^2}\cap \int{L_2^2}=\emptyset$. By repeating the previous arguments, we can prove that the positive cycle $(L^1_0,L^1_1,L^1_2,L^1_3,L^1_4)$ is associated in $\clos{\mathcal{P}_2}$  to the positive cycle of $p^2$ starting from $L^2_0$. Similarly, $(L^1_0,L^1_3,L^1_2,L^1_1,L^1_4)$ is associated in $\clos{\mathcal{P}_2}$ to the negative cycle of $p^2$ starting from $L_0^2$.  

We conclude that the rectangle paths $R_1$ and $R_2$ are respectively associated to two rectangle paths in $\clos{\mathcal{P}_2}$ of the form: 

$$r^2_0,...,r^2_k,L^2_0,L^2_{01},...,L^2_{0s(0)},L^2_1, L^2_{11},...,L^2_{1s(1)},..., L^2_k,r^2_l,r^2_{l+1},...,r^2_n$$  
 $$r^2_0,...,r^2_k,L^2_0,L'^2_{01},...,L'^2_{0s'(0)},L^2_3, L'^2_{31},...,L'^2_{3s'(3)},...,L^2_k,r^2_l,r^2_{l+1},...,r^2_n$$
where $(L^2_0,L^2_1,L^2_2,L^2_3,L^2_4)$ and $(L^2_0,L^2_3,L^2_2,L^2_1,L^2_4)$ are the two cycles starting from $L^2_0$ around $p^2\in L^2_0$. Also, $L^2_0,L^2_{01},...,L^2_{0s(0)},L^2_1, L^2_{11},...,L^2_{1s(1)},..., L^2_k$ is the rectangle path associated to the generalized rectangle path $(L^2_0,L^2_1,...,L^2_k)$ and $L^2_0,L'^2_{01},...,L'^2_{0s'(0)},L^2_3, L'^2_{31},$ $...,L'^2_{3s'(3)},...,L^2_k$ is the rectangle path associated to the generalized rectangle path $(L^2_0,L^2_3,...,L^2_k)$. By definition, the previous rectangle paths are homotopic by a homotopy of type $\mathcal{C}$, which gives us the desired result. 

\section{Proof of Theorem B}\label{s.mainresult}

Let $(M_1,\Phi_1)$ and $(M_2,\Phi_2)$ be two transitive Anosov flows, $\mathcal{P}_1$ and $\mathcal{P}_2$ their bifoliated planes, $\mathcal{F}_{1,2}^{s}$,$\mathcal{F}_{1,2}^{u}$ the stable and unstable foliations in $\mathcal{P}_{1,2}$. Assume that $\mathcal{P}_1$ and $\mathcal{P}_2$ carry two Markovian families $\mathcal{R}_1$ and $\mathcal{R}_2$ associated canonically by Theorem \ref{t.associatemarkovfamiliestogeometrictype} to the same class of equivalent geometric types. By choosing in an appropriate way orientations on $\mathcal{F}_{1,2}^{s}$,$\mathcal{F}_{1,2}^{u}$ and also representatives of every rectangle orbit, thanks to Lemma \ref{l.geomtypeinclass}, we may assume that $\mathcal{R}_1$ and $\mathcal{R}_2$ are canonically associated to the same geometric type  $(R_1,...,R_n,(h_i)_{i \in \llbracket 1,n \rrbracket}, (v_i)_{i\in \llbracket 1,n \rrbracket}, \mathcal{H}, \mathcal{V},\phi, u)$. 

Let us denote $\Gamma_{1,2}$ (resp. $(\Gamma_{1,2})_{M_{1,2}}$) the boundary periodic orbits associated to $\mathcal{R}_{1,2}$ in $\mathcal{P}_{1,2}$ (resp. in $(M_{1,2},\Phi_{1,2})$), $\clos{\mathcal{P}_{1,2}}$ the bifoliated planes of $\Phi_1,\Phi_2$ up to surgeries on $\Gamma_{1,2}$, $\clos{\mathcal{F}^{s,u}_{1,2}}$ their stable and unstable singular foliations, $\clos{\mathcal{R}_{1,2}}$ the lifts of $\mathcal{R}_{1,2}$ on $\clos{\mathcal{P}_{1,2}}$,  $\clos{\Gamma_{1,2}}$ the lifts of $\Gamma_{1,2}$ on $\clos{\mathcal{P}_{1,2}}$ and $r_0^{1,2}$ the origin rectangles in $\clos{\mathcal{P}_{1,2}}$.

In Section \ref{s.rectpaths}, we established that rectangle paths allow us to navigate inside the bifoliated plane, thus endowing the bifoliated plane with a ``system of coordinates". According to Theorem \ref{t.associatingrectanglepaths}, any centered rectangle path in $\clos{\mathcal{P}_{1}}$ can be naturally associated to a rectangle path in $\clos{\mathcal{P}_{2}}$. In the previous section, we showed that this correspondence gives birth to a compatible system of coordinates in $\clos{\mathcal{P}_{1,2}}$: closed and centered rectangle paths in $\clos{\mathcal{P}_{1}}$ correspond to closed and centered rectangle paths in $\clos{\mathcal{P}_{2}}$. We would like now to show Theorem B. In order to do so, by Theorem \ref{t.generalbarbot} it suffices to show   
that there exists a homeomorphism $h: \clos{\mathcal{P}_1}\rightarrow \clos{\mathcal{P}_2}$ and an isomorphism $\alpha: \pi_1(M_1-\Gamma_1) \rightarrow \pi_1(M_2-\Gamma_2)$ such that: 

\begin{itemize}
    \item the image by $h$ of any stable/unstable leaf in $\clos{\mathcal{F}_1^{s,u}}$ is a stable/unstable leaf in $\clos{\mathcal{F}_2^{s,u}}$ 
    \item for every $g\in \pi_1(M_1-\Gamma_1)$ and every $x\in \clos{\mathcal{P}_1}$ we have $$h(g(x))= \alpha(g)(h(x))$$ 
\end{itemize}

We will split the proof of the previous statement in two parts: 
\begin{enumerate}
    \item the construction of an isomorphism $\alpha: \pi_1(M_1-\Gamma_1) \rightarrow \pi_1(M_2-\Gamma_2)$
    \item the construction of a homeomoprhism $h: \clos{\mathcal{P}_1}\rightarrow \clos{\mathcal{P}_2}$ with the desired properties 
\end{enumerate}

Let us begin with the construction of $\alpha$: 
\begin{prop}\label{p.existenceiso}
Under the hypotheses of Theorem B, the groups $\pi_1(M_1-\Gamma_1)$ and $ \pi_1(M_2-\Gamma_2)$ are isomorphic. 
\end{prop}
\begin{proof}
Consider $T(r_0^1)\subset \clos{\mathcal{R}_1}$ the set of all the rectangles in $ \clos{\mathcal{P}_1}$ that have the same type as $r_0^1$. By enumerating the set of these rectangles by $\mathbb{Z}$ we obtain a homomorphism from $\pi_1(M_1-\Gamma_1)$ to the group of automorphisms of $\mathbb{Z}$, $Aut(\mathbb{Z})$. 

The previous homomorphism is injective. Indeed, consider $g$ an element in $\pi_1(M_1-\Gamma_1)$ that fixes one rectangle $R$ in $T(r_0^1)$. This implies that $g$ has a fixed point in $R$ and therefore it acts as an expansion or contraction on its stable leaf, which contradicts the fact that $g(R)=R$. 

Consider now $T(r_0^2)\subset \clos{\mathcal{R}_2}$ the set of all the rectangles in $ \clos{\mathcal{P}_2}$ that have the same type as $r_0^2$. Let us choose an enumeration of the elements of $T(r_0^2)$. Take $R\in T(r_0^2)$ and any centered rectangle path in $ \clos{\mathcal{P}_2}$ ending in $R$ (such a rectangle path exists thanks to Corollary \ref{c.rectpathsstartingending}). This rectangle path corresponds to a centered rectangle path in $ \clos{\mathcal{P}_1}$ ending at $R'\in T(r_0^1)$. Notice that $R'$ does not depend on the choice of rectangle path thanks to Theorem \ref{t.endingbysamerectangle}. We can therefore associate $R$ to the same element of $\mathbb{Z}$ as $R'$. This defines an enumeration of the elements of $T(r_0^2)$ and once again we have an injective homomorphism from $\pi_1(M_2-\Gamma_2)$ to $Aut(\mathbb{Z})$. Therefore, $\pi_1(M_2-\Gamma_2)$ and $\pi_1(M_1-\Gamma_1)$ are subgroups of $Aut(\mathbb{Z})$. It suffices to show that they correspond to the exact same subgroup. 

Consider $g^1\in \pi_1(M_1-\Gamma_1)\leq Aut(\mathbb{Z})$ the element sending $r_0^1$ to $L^1\in T(r_0^1)$ (notice that there is a unique element doing this by our previous arguments) and $g^2 \in \pi_1(M_2-\Gamma_2)\leq Aut(\mathbb{Z})$ the element sending $r_0^2$ to $L^2 \in T(r_0^2)$, where $L^2$ is associated to the same element of $\mathbb{Z}$ as $L^1$. We would like to show that for every $k\in \mathbb{Z}$, $g^1(k)=g^2(k)$. 

Take $k\in \mathbb{Z}$ and its corresponding rectangles in $\clos{\mathcal{R}_1}, \clos{\mathcal{R}_2}$, say $R^1,R^2$ respectively. Consider $r_0^1,r_1^1,...,L^1$ a rectangle path in $ \clos{\mathcal{P}_1}$ starting from $r_0^1$ and ending at $L^1$ and its corresponding rectangle path $r_0^2,r_1^2,...,L^2$ in $ \clos{\mathcal{P}_2}$ (notice that the latter path, by definition of $L^1$ and $L^2$, must end at $L^2$). Consider also $r_0^1,{r'_1}^1,...,R^1$ a rectangle path in $ \clos{\mathcal{P}_1}$ starting from $r_0^1$ and ending at $R^1$ and its corresponding rectangle path $r_0^2,{r'_1}^2,...,R^2$ in $ \clos{\mathcal{P}_2}$ (notice once again that by the definition of $R^1,R^2$ the latter path must end at $R^2$). 

It is now not difficult to see that the rectangle path $r_0^1,r_1^1,...,L^1=g^1(r_0^1),g^1({r'_1}^1),...,g^1(R^1)$ corresponds to $r_0^2,r_1^2,...,L^2=g^2(r_0^2),g^2({r'_1}^2),...,g^2(R^2)$. Therefore, $g^1(R^1)$ and $g^2(R^2)$ are associated to the same element in $\mathbb{Z}$, which gives us the desired result. 
\end{proof}
We are now ready to proceed to the construction of the homemorphism $h: \clos{\mathcal{P}_1}\rightarrow \clos{\mathcal{P}_2}$. We are first going to define $h$ as a map from  $\clos{\mathcal{R}_1}$ to $\clos{\mathcal{R}_2}$ and then we are going to show that this map extends to a homeomorphism from $\clos{\mathcal{P}_1}$ to $ \clos{\mathcal{P}_2}$.

Indeed, let $R\in \clos{\mathcal{R}_1}$. Consider a centered rectangle path in $\clos{\mathcal{P}_1}$ ending at $R$ and its corresponding centered rectangle path in $\clos{\mathcal{P}_2}$ ending at $R'\in \clos{\mathcal{R}_2}$. We define $h(R)=R'$. Notice that $R'$ does not depend on our initial choice of rectangle path, thanks to Theorem \ref{t.endingbysamerectangle}. Also, if we denote by $\alpha$ the isomorphism given by the previous proposition, by our previous proof, we have that for any $g\in \pi_1(M_1-\Gamma_1)$ and $R\in \clos{\mathcal{R}_1}$: \begin{equation}\label{eq.equivarianceh}
    h(g(R))=\alpha(g)(h(R))
\end{equation} Finally, by applying the same arguments for rectangles in $\clos{\mathcal{R}_2}$ we can construct a right and left inverse for $h$, therefore $h: \clos{\mathcal{R}_1}\rightarrow  \clos{\mathcal{R}_2}$ is a bijection that is equivarient with respect to the actions of the fundemental groups. 

Theorem B follows from the next proposition: 
\begin{prop}\label{p.existencehomeo}
The previous map $h$ extends to a unique homeomorphism from $\clos{\mathcal{P}_1}$ to $\clos{\mathcal{P}_2}$ such that: 

\begin{itemize}
    \item the image by $h$ of any stable/unstable leaf in $\clos{\mathcal{F}_1^{s,u}}$ is a stable/unstable leaf in $\clos{\mathcal{F}_2^{s,u}}$ 
    \item for every $g\in \pi_1(M_1-\Gamma_1)$ and every $x\in \clos{\mathcal{P}_1}$ we have $$h(g(x))= \alpha(g)(h(x))$$ 
\end{itemize}
\end{prop}
\begin{proof}

Let $x\in\clos{\mathcal{P}_1}$. Recall that a quadrant of $x$ is  the closure of a connected component of $\clos{\mathcal{P}_1}$ minus the union of stable and unstable leaves crossing $x$. For any quadrant $Q$ of $x$, there exists a unique bi-infinite sequence $...R^{Q}_{-1}, R^{Q}_0,R^{Q}_1,...$ of rectangles in $\clos{\mathcal{R}_1}$ such that: 
\begin{itemize} 
\item for every $i$, $R^{Q}_{i+1}$ is a predecessor of $R^{Q}_i$ 
\item for every $i$, $R^{Q}_i$ contains a neighbourhood of $x$ in $Q$
\end{itemize}
Indeed, by the definition of a Markovian family there exists one rectangle $R^{Q}_0$ containing a neighbourhood of $x$ in the quadrant $Q$. By Lemma \ref{l.existenceofpredecessors}, there exists a unique predecessor and successor of $R^{Q}_0$ containing a neighbourhood of $x$ in $Q$ and by Lemma \ref{l.npredecessor} any other rectangle with this property is a predecessor or a successor of some generation of $R^{Q}_0$. 

By Lemmas \ref{l.infiniteintersectionverticalrectangles} and \ref{l.infiniteintersectionhorizontalrectangles}, we have that $\{x\}=\overset{+\infty}{\underset{k=0}{\cap}} R^{Q}_{k}$. We therefore define $h(x):=\overset{+\infty}{\underset{k=0}{\cap}} h(R^{Q}_{k})\in \clos{\mathcal{P}_2}$. Notice that for every $k$, $h(R^{Q}_k)$ is a predecessor of $h(R^{Q}_{k-1})$, therefore once again $\overset{+\infty}{\underset{k=0}{\cap}} h(R^{Q}_{k})$ corresponds to a unique point in $\clos{\mathcal{P}_2}$. We would now like to show that this point does not depend on the initial choice of quadrant for $x$. 

\vspace{0.5cm}
\underline{Independence from the choice of quadrant} 

Assume first that $x$ is not a boundary periodic point and that the previous quadrant $Q$ was the $(\epsilon,\epsilon')$ quadrant of $x$. Consider the quadrant $(\epsilon,-\epsilon')$ of $x$ and $$i:= \inf \{k\in\mathbb{Z}| \forall j\geq k ~ R^{(\epsilon, \epsilon')}_j \text{ contains a germ of the $(\epsilon,\epsilon')$ and $(\epsilon,-\epsilon')$ quadrants of $x$ }\}$$ Let us first notice that such an $i$ belongs in $ \mathbb{Z}\cup \{-\infty\}$. Indeed, suppose that $R^{(\epsilon, \epsilon')}_0$ contains a germ of the $(\epsilon,\epsilon')$ quadrant of $x$, but not of the $(\epsilon,-\epsilon')$ quadrant of $x$. This implies that $x$ belongs to a stable boundary component of $R^{(\epsilon, \epsilon')}_0$, say $s$. By Lemmas  \ref{l.crossingrectanglesnoperiodicpoints},  \ref{l.crossingrectangleswithperiodicpoints} and \ref{l.npredecessor} there exists  $k_0$ sufficiently big such that  $R^{(\epsilon, \epsilon')}_{k_0}$ is a $s$-crossing predecessor of $R^{(\epsilon, \epsilon')}_0$, that contains a germ of the $(\epsilon,\epsilon')$ and $(\epsilon,-\epsilon')$ quadrants of $x$. By the Markovian intersection property, for every $k\geq k_0$, $R^{(\epsilon, \epsilon')}_{k}$ also contains a germ of the $(\epsilon,\epsilon')$ and $(\epsilon,-\epsilon')$ quadrants of $x$. We thus get that $i \in \mathbb{Z}\cup \{-\infty\}$. 

If $i=-\infty$, then the sequence associated to the quadrant $(\epsilon,-\epsilon')$ is the exact same sequence as for $(\epsilon,\epsilon')$. This can happen for instance when $x$ does not belong to the stable or unstable leaf of some boundary periodic point. Therefore, in this case we get the desired result. 

If $i\in \mathbb{Z}$, then for every $j\geq i$ the rectangle $R_j^{(\epsilon,\epsilon')}$ is contained in the sequence associated to the quadrant $(\epsilon,-\epsilon')$. Let us denote this sequence $...,R^{(\epsilon, -\epsilon')}_{-1},R^{(\epsilon, -\epsilon')}_0,...,R^{(\epsilon, -\epsilon')}_{i-1},R^{(\epsilon, -\epsilon')}_i=R^{(\epsilon, \epsilon')}_i,R^{(\epsilon, -\epsilon')}_{i+1}=R^{(\epsilon, \epsilon')}_{i+1},...,R^{(\epsilon, -\epsilon')}_n=R^{(\epsilon, \epsilon')}_n,...$. The rectangle $R^{(\epsilon, -\epsilon')}_{i-1}$ is the successor of $R^{(\epsilon, -\epsilon')}_i=R^{(\epsilon, \epsilon')}_i$ that intersects $\clos{\mathcal{F}^s}(x)$ and that is not $R^{(\epsilon, \epsilon')}_{i-1}$. Assume without any loss of generality that this is the successor of $R^{(\epsilon, \epsilon')}_i$ that is right above $R^{(\epsilon, \epsilon')}_{i-1}$ (see Figure \ref{f.hmap}). In this case, $h$ will send $R^{(\epsilon, -\epsilon')}_{i-1}$ to the successor of $h(R^{(\epsilon, -\epsilon')}_i)$ that is right above $h(R^{(\epsilon, \epsilon')}_{i-1})$. Furthermore, since $R^{(\epsilon, -\epsilon')}_{i-2}$ contains $x$ and a germ of its $(\epsilon,-\epsilon')$ quadrant, $R^{(\epsilon, -\epsilon')}_{i-2}$  is the unique successor of $R^{(\epsilon, -\epsilon')}_{i-1}$ that intersects $\clos{\mathcal{F}^s}(x)$, which corresponds to the lowermost successor of $R^{(\epsilon, -\epsilon')}_{i-1}$. Therefore, $h(R^{(\epsilon, -\epsilon')}_{i-2})$ will correspond to the lowermost successor of $h(R^{(\epsilon, -\epsilon')}_{i-1})$, the successor of $h(R^{(\epsilon, -\epsilon')}_{i-1})$ that intersects $\clos{\mathcal{F}^s}(h(x))$. We can thus show thanks to Lemma \ref{l.infiniteintersectionhorizontalrectangles} that $\overset{i-1}{\underset{k=-\infty}{\cap}} h(R^{(\epsilon, -\epsilon')}_{k})$ corresponds to a stable segment in $\clos{\mathcal{F}^s}(h(x))$ crossing $h(R^{(\epsilon, -\epsilon')}_{i-1})$. 
\begin{figure}
    \centering
    \includegraphics[scale=0.5]{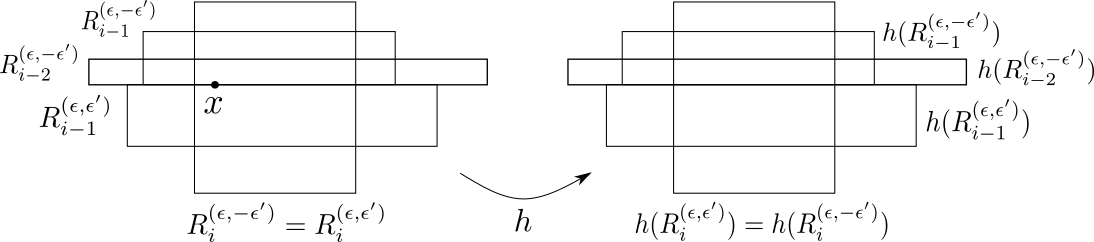}
    \caption{}
    \label{f.hmap}
\end{figure}
We also know thanks to Lemma \ref{l.infiniteintersectionverticalrectangles} that $\overset{+\infty}{\underset{k=i}{\cap}} h(R^{(\epsilon, -\epsilon')}_{k})=\overset{+\infty}{\underset{k=i}{\cap}} h(R^{(\epsilon, \epsilon')}_{k})$ corresponds to an unstable segment in $\clos{\mathcal{F}^u}(h(x))$ crossing $h(R^{(\epsilon, -\epsilon')}_{i})$. We conclude that $\overset{\infty}{\underset{k=-\infty}{\cap}} h(R^{(\epsilon, -\epsilon')}_{k})=\{h(x)\}$. By a similar argument we can show that $\overset{\infty}{\underset{k=-\infty}{\cap}} h(R^{(-\epsilon, -\epsilon')}_{k})=\overset{\infty}{\underset{k=-\infty}{\cap}} h(R^{(-\epsilon, \epsilon')}_{k})= \{h(x)\}$, which gives us the desired result. 

Assume now that $x$ is a boundary periodic point. In this case, we can define $h(x)$ by the same exact argument for some choice of quadrant $Q$ of $x$ ($x$ has infinitely many quadrants in $\clos{\mathcal{P}_1}$). Consider now $Q'\neq Q$ the quadrant of $x$ that intersects $Q$ along a stable leaf of $x$. By a similar argument, we can show that for every $N$, $\overset{N}{\underset{k=-\infty}{\cap}} h(R^{Q'}_{k})$ corresponds to a stable segment in $\clos{\mathcal{F}^s}(h(x))$. Furthermore, if $t$ is the generator of the stabilizer of a stable or unstable leaf of $x$, then for any quadrant $q$ of $x$ we have that  $(t^k(R^{q}_{0}))_{k\in\mathbb{Z}}$ is a subsequence of $ (R^{q}_{k})_{k\in\mathbb{Z}}$. Therefore, $\big(\alpha(t)^k(h(R^{q}_{0}))\big)_{k\in\mathbb{Z}}$ is also a subsequence of $ (h(R^{q}_{k}))_{k\in\mathbb{Z}}$. Our previous arguments imply that both $\overset{+\infty}{\underset{k=-\infty}{\cap}} h(R^{Q'}_{k})$ and $\overset{+\infty}{\underset{k=-\infty}{\cap}} h(R^{Q}_{k})$ are fixed by $\alpha(t)$ and belong to the same stable leaf in $\clos{\mathcal{F}^s}$. They thus correspond to the same point $h(x)$, which gives us the desired result.

\vspace{0.5cm}
\underline{$h(\clos{\Gamma_1})\subset \clos{\Gamma_2}$ and $h(\clos{\mathcal{P}_1}-\clos{\Gamma_1})\subset \clos{\mathcal{P}_2}-\clos{\Gamma_2}$ }

Indeed, assume that  $x\in\clos{\Gamma_1}$ and $h(x)\in \clos{\mathcal{P}_2}-\clos{\Gamma_2}$. There are infinitely many quadrants around $x$, therefore there exists an infinite  family of rectangles $(R_i)_{i\in \mathbb{N}}$ of $\clos{\mathcal{R}_1}$ intersecting $x$ and having 2 by 2 disjoint interiors. For every $i$ the rectangle $h(R_i)$ intersects $h(x)$ that has only finitely many quadrants. Therefore, there exist $i\neq j$ such that $\int{h(R_i)}\cap \int{h(R_j)} \neq \emptyset$. Therefore, $h(R_i)$ is a predecessor or successor of some generation of $h(R_j)$. This contradicts that the fact that $h$ defines a bijection from $\clos{\mathcal{R}_1}$ to $\clos{\mathcal{R}_2}$ that respects the relations of predecessor/successor. 

By a similar argument we can show that $h(\clos{\mathcal{P}_1}-\clos{\Gamma_1})\subset \clos{\mathcal{P}_2}-\clos{\Gamma_2}$
\vspace{0.5cm}

\underline{$h$ is a homeomorphism} 

By applying all the above arguments for $\clos{\mathcal{P}_2}$, we can show that $h$ admits a right and left inverse therefore $h$ is a bijection. 

Consider now $x\in \clos{\mathcal{P}_1}-\clos{\Gamma_1}$ and $y:=h(x)\in \clos{\mathcal{P}_2}-\clos{\Gamma_2}$. Denote by $(R^{(\epsilon, \epsilon')}_{k}(x))_{k\in \mathbb{Z}}$ the sequence of rectangles defined in the previous paragraphs for $x$ and the quadrant $(\epsilon, \epsilon')$. The continuity of $h$  on $x$ follows for the fact that $ (\underset{(\epsilon, \epsilon')\in \{+,-\}^2}{\cup}\overset{N}{\underset{k=-N}{\cap}} R^{(\epsilon, \epsilon')}_{k}(x))_{N\in \mathbb{N}}$ and $ (\underset{(\epsilon, \epsilon')\in \{+,-\}^2}{\cup}\overset{N}{\underset{k=-N}{\cap}} h(R^{(\epsilon, \epsilon')}_{k}(x)))_{N\in \mathbb{N}}$ form respectively a basis of neighbourhoods of $x$ in $\clos{\mathcal{P}_1}$ and of $y$ in $\clos{\mathcal{P}_2}$.

    Consider now $x\in \clos{\Gamma_1}$, $y:=h(x)\in \clos{\Gamma_2}$,  $\{Q_k(y)|k\in \mathbb{Z}\}$ (resp. $\{Q_k(x)|k\in \mathbb{Z}\}$) the set of all quadrants of $y$ (resp. $x$) and $\mathcal{F}$ the set of all functions from $\mathbb{Z}$ to $\mathbb{N}^{*}$. Without any loss of generality we will assume that $h(Q_k(x))=Q_k(y)$. Denote by $(R^{Q_k(x)}_{l}(x))_{l\in \mathbb{Z}}$ the sequence of rectangles defined in the previous paragraphs for $x$ and the quadrant $Q_k(x)$. Once again the continuity of $h$ on $x$ follows from the fact that $ (\underset{k\in \mathbb{Z}}{\cup}\overset{f(k)}{\underset{l=-f(k)}{\cap}} R^{Q_k(x)}_{l}(x))_{f\in \mathcal{F}}$ and $ (\underset{k\in \mathbb{Z}}{\cup}\overset{f(k)}{\underset{l=-f(k)}{\cap}} h(R^{Q_k(x)}_{l}(x)))_{f\in \mathcal{F}}$ form respectively a basis of neighbourhoods of $x$ in $\clos{\mathcal{P}_1}$ and $y$ in $\clos{\mathcal{P}_2}$. 
    
   The above prove that $h$ is continuous. By applying the same arguments to $h^{-1}$, we get that $h$ is a homeomorphism. 

\vspace{0.5cm}
\underline{$h$ sends oriented stable/unstable leaves to oriented stable/unstable leaves}

Consider $x\in \clos{\mathcal{P}_1}$. Notice that for any choice of quadrant of $x$, say $q$, we have that $\clos{\mathcal{F}^u}(x)= \overset{\infty}{\underset{k=0}{\cup}}\overset{\infty}{\underset{l=k}{\cap}} R^{q}_{l}$ Therefore, $$h(\clos{\mathcal{F}^u}(x))=\overset{\infty}{\underset{k=0}{\cup}}\overset{\infty}{\underset{l=k}{\cap}} h(R^{q}_{l})=\clos{\mathcal{F}^u}(h(x))$$
In the same way, we can show that $h$ sends stable leaves in $\clos{\mathcal{P}_1}$ to stable leaves in $\clos{\mathcal{P}_2}$. 

Next, recall that  $\clos{\mathcal{F}_1^{s,u}}$ and  $\clos{\mathcal{F}_2^{s,u}}$ have been endowed with orientations such that the map $h$ sends the predecessors of any rectangle $R\in \clos{\mathcal{R}_1}$ ordered from left to right, to the predecessors of $h(R)$ ordered from left to right. Same for the successors of any rectangle in $\clos{\mathcal{R}_1}$. We deduce that $h$ respects the orientations of $\clos{\mathcal{F}_1^{s,u}}$ and  $\clos{\mathcal{F}_2^{s,u}}$. 
\vspace{0.5cm}

\underline{$h$ is equivariant with respect to the actions of the fundamental groups}

Consider $x\in \clos{\mathcal{P}_1}$ and $g\in \pi_1(M_1-\Gamma_1)$. For any quadrant $q$ of $x$, we have that (see Proposition \ref{eq.equivarianceh}): $$h(g(x))=h(g(\overset{+\infty}{\underset{l=-\infty}{\cap}} R^{q}_{l}))= \overset{+\infty}{\underset{l=-\infty}{\cap}} h(g(R^q_l))=\overset{+\infty}{\underset{l=-\infty}{\cap}} \alpha(g)(h(R^q_l))=\alpha(g)(h(x))$$
\end{proof}

\section{Cycles around boundary periodic points}\label{s.theoremsDE}
Fix $(M,\Phi)$ a transitive Anosov flow, $\mathcal{P}$ its bifoliated plane endowed with an orientation, $\mathcal{F}^{s,u}$ the stable/unstable foliations in $\mathcal{P}$ endowed with an orientation,  $\mathcal{R}$ a Markovian family on $\mathcal{P}$ and $\Gamma \subset \mathcal{P}$ the set of boundary points of $\mathcal{R}$.

In Section \ref{s.homotopiesofpaths} we introduced the notion of cycle around a boundary arc point of $\clos{\mathcal{P}}$. We would like in this section to adapt this notion of cycle for boundary periodic points, in order to obtain information about the neighbourhood of these orbits in $M$. This will allow us to prove Theorems D and E. 

\begin{lemm}\label{l.nocyclearoundperiod}
Let $p$ be a boundary periodic point in $\mathcal{P}$ and $R\in\mathcal{R}$ a rectangle containing $p$. Suppose that $p$ is contained in the stable boundary component $s$ of $R$. There exists no $s$-crossing predecessor of $R$ containing $p$.
\end{lemm}
\begin{proof}
Suppose that such a rectangle $L$ exists. Consider $g\in \text{Stab}(p)$ such that the stable boundaries of $g(R)$ become very thin and the unstable ones very long. The intersection of the rectangles $g(R)$ and $L$ is not Markovian. Absurd.
\end{proof}

According to the previous lemma, we can not extend Definition \ref{d.cyclearcpoint} for boundary periodic points.

Take $p\in \mathcal{P}$ a boundary periodic point and $L_0\in \mathcal{R}$ containing $p$. Assume without any loss of generality that $L_0$ contains a germ of the $(+,+)$ quadrant of $p$ (there are only a finite number of rectangles with this property up to the action of $\text{Stab}(p)$). Consider now the set $Q_{(-,+)}$ of rectangles $r$ intersecting a germ of the $(-,+)$ quadrant of $p$ and such that $L_0\cap \mathcal{F}^u_+(p) \subseteq r\cap \mathcal{F}^u_+(p)$. The set $Q_{(-,+)}$ is not stable under the action of $\text{Stab}(p)$, but it is contained in the union of a finite number of $\text{Stab}(p)$-orbits of rectangles. We can therefore define $L_1$ as the unique rectangle in  $Q_{(-,+)}$ for which for all $r\in Q_{(-,+)}$ $$L_1\cap \mathcal{F}^u_+(p)\subseteq r\cap \mathcal{F}^u_+(p)$$ Notice that $L_1$ may be equal to $L_0$, if $L_0$ intersects a germ of the $(+,+)$ and $(-,+)$ quadrants of $p$. 

Next, we define $Q_{(-,-)}$, the set of rectangles $r$ intersecting a germ of the $(-,-)$ quadrant of $p$ and such that $L_1\cap \mathcal{F}^s_-(p) \subset r\cap \mathcal{F}^s_-(p)$. Again, we define $L_2$ as the unique rectangle in $Q_{(-,-)}$ for which all $r\in Q_{(-,-)}$ $$L_2\cap \mathcal{F}^s_-(p)\subseteq r\cap \mathcal{F}^s_-(p)$$ Finally, if the rectangle $L_0$ intersects a germ of the $(+,-)$ quadrant of $p$, we take $L_3=L_0$. Otherwise, we define analogously $Q_{(+,-)}$ and we define $L_3$ as the unique rectangle in $Q_{(+,-)}$ for which all $r\in Q_{(+,-)}$ $$L_3\cap \mathcal{F}^u_-(p)\subseteq r\cap \mathcal{F}^u_-(p)$$ 
\begin{figure}
    \centering
    \includegraphics[scale=0.5]{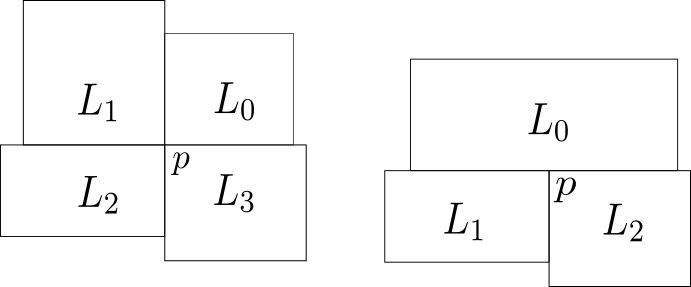}
    \caption{Some examples of positive pre-cycles starting from $L_0$}
    \label{f.precycles}
\end{figure}
 By Lemma \ref{l.nocyclearoundperiod} and our previous discussion, we have that any two rectangles among $L_0,L_1,L_2,L_3$ have either disjoint interiors or are equal (see Figure \ref{f.precycles}). By eventually removing a rectangle appearing twice, we can assume that the rectangles $L_0,L_1,L_2,L_3$ have two by two disjoint interiors. 
We will call $(L_0,L_1,L_2,L_3)$ \emph{the positive pre-cycle around $p$ starting from $L_0$}. By changing the cyclic order in which we visit the quadrants around $p$, we can define similarly \emph{the negative pre-cycle around $p$ starting from $L_0$}. 

\begin{rema}\label{r.precyclesfiniteuptoaction}
To every rectangle intersecting $p$ we associate canonically a positive and a negative pre-cycle around $p$. The image of a pre-cycle around $p$ by an element of $\text{Stab}(p)$ is a pre-cycle around $p$. Therefore, there is a finite number of pre-cycles around $p$ up to the action of $\text{Stab}(p)$. 
\end{rema}

\begin{prop}
To any pre-cycle around $p$ we can associate canonically a closed rectangle path of $\mathcal{P}$.
\end{prop}
\begin{proof}
Since a generalized rectangle path can canonically be associated to a rectangle path (see Definition \ref{d.generalisedrectpaths}), it suffices to show that we can canonically associate any pre-cycle to a generalized rectangle path in $\mathcal{P}$ beginning and ending by the same rectangle. Let us consider $L_0$ intersecting $p$ and $L_0,...,L_k$ its associated positive pre-cycle. The case of negative pre-cycles follows from the same arguments. Assume without any loss of generality that the previous cycle is of the from $L_0,L_1,L_2,L_3$, where the $L_i$ have two by two disjoint interiors and that $L_0$ contains a germ of the $(+,+)$ quadrant of $p$. All the other cases, follow by the same arguments.  

We begin our generalized rectangle path by $L_0$. Let $s_0=L_0\cap \mathcal{F}^u_+(p)$ and $r_0$ be the $s_0$-crossing successor of $L_0$ containing the endpoint of $s_0$ that is not $p$ (this rectangle exists by Lemma \ref{l.crossingrectangleswithperiodicpoints}). By construction of $L_1$, $r_0$ is a successor of $L_1$ of some generation. Similarly, we define $s_1:=L_1\cap \mathcal{F}^s_-(p)$, $s_2:=L_2\cap \mathcal{F}^u_-(p)$, $r_1$ as the $s_1$-crossing predecessor of $L_1$ containing the endpoint of $s_1$ that is not $p$ and $r_2$ as the $s_2$-crossing successor of $L_2$ containing the endpoint of $s_2$ that is not $p$. Again, $r_1$ (resp. $r_2$) is a predecessor (resp. successor) of $L_2$ (resp. $L_3$) of some generation. 

Finally, if $L_0\cap \mathcal{F}^s_+(p)\subset L_3\cap \mathcal{F}^s_+(p)$ we define $s_3:=L_0\cap \mathcal{F}^s_+(p)$ and $r_3$ as the $s_3$-crossing predecessor of $L_0$ intersecting the endpoint of $s_3$ that is not $p$. If instead, $L_3\cap \mathcal{F}^s_+(p)\subset L_0\cap \mathcal{F}^s_+(p)$ we define $s_3:=L_3\cap \mathcal{F}^s_+(p)$ and $r_3$ as the $s_3$-crossing predecessor of $L_3$ intersecting the endpoint of $s_3$ that is not $p$. In any case, $r_3$ is a predecessor of some generation of $L_0$ and also $L_3$.

We deduce that $(L_0,r_0,L_1,r_1,L_2,r_2,L_3,r_3,L_0)$ is a generalized rectangle path in $\mathcal{P}$ that is canonically associated to the pre-cycle $(L_0,L_1,L_2,L_3)$, which gives us the desired result.
\end{proof}
\begin{defi}
Take $p$ a boundary periodic point. The set of rectangle paths associated to all pre-cycles around $p$ will be called the set of \emph{cycles} of $p$ and will be denoted by $\text{Cycl}(p)$.
\end{defi}
\begin{rema}\label{r.cyclesfiniteuptoaction}
If $c$ is a pre-cycle around $p$ and $g\in \text{Stab}(p)$. The image by $g$ of the cycle  associated to $c$, is the cycle  associated to the pre-cycle $g(c)$. Therefore, by Remark  \ref{r.precyclesfiniteuptoaction}, the set of cycles around any boundary periodic orbit $p$  consists of a finite number of orbits by the action of $\text{Stab}(p)$.
\end{rema}

More generally, 
\begin{rema}\label{r.orbitofcycles}
For any $g\in \pi_1(M)$ and any boundary periodic point $p$, $g(\text{Cycl}(p))=\text{Cycl}(g(p))$. Therefore, the set of all cycles around any boundary periodic point in the $\pi_1(M)$-orbit of $p$, consists of a finite number of rectangle paths up to the action of $\pi_1(M)$.
\end{rema}
  We will now prove that each of the previous orbits of rectangle paths can be described as a closed rectangle path in any geometric type associated to $\mathcal{R}$ by Theorem \ref{t.associatemarkovfamiliestogeometrictype}. 

\begin{defi}
Consider a geometric type $G=(n,(h_i)_{i \in \llbracket 1,n \rrbracket}, (v_i)_{i\in \llbracket 1,n \rrbracket}, \mathcal{H}, \mathcal{V},\phi, u)$, where $\mathcal{H}=\{H_i^j, i\in \llbracket 1,n \rrbracket, j\in \llbracket 1, h_i \rrbracket\}$ and $  \mathcal{V}=\{V_i^j, i\in \llbracket 1,n \rrbracket, j\in \llbracket 1, v_i \rrbracket\}$. We define a \emph{rectangle path} in $G$ as a finite sequence of the form $i_0,s_0,i_1,s_1,i_2,...,s_m, i_{m+1}$ where \begin{itemize}     
    \item for all $k\in \llbracket 0, m+1\rrbracket$, $i_k\in \llbracket 1,n\rrbracket$ 
\item for all $k\in \llbracket 0, m \rrbracket$,  $s_k$ is of the form $V_{i_k}^{\bullet}$ or $H_{i_k}^{\bullet}$. 
     \item for all $k\in \llbracket 0, m\rrbracket$, if $s_k\in \mathcal{V}$ (resp. $s_k\in \mathcal{H}$)  then it is the image by $\phi$ (resp. $\phi^{-1}$) of an element of $\mathcal{H}$ (resp. $\mathcal{V}$) of the form $H_{i_{k+1}}^{\bullet}$ (resp. $V_{i_{k+1}}^{\bullet}$). 
\end{itemize}
Furthermore, the rectangle path $i_0,s_0,i_1,s_1,i_2,...,s_m, i_{m+1}$ will be called \emph{closed} when 
$i_0=i_{m+1}$
\end{defi}
Recall that we may think of a geometric type $G$ as a set of rectangles $R_1,...,R_n$, endowed each with a collection of horizontal and vertical subrectangles identified by the map $(\phi,u)$. In these terms, the reader may think of a rectangle path in $G$ as a sequence of the form $R_0,s_0,...,R_m,s_m,R_{m+1}$, where for every $k\in \llbracket 0,m\rrbracket $, $s_k$ corresponds to a vertical (resp. or horizontal) subrectangle of $R_k$ and a horizontal (resp. or vertical) subrectangle of $R_{k+1}$. 
\begin{prop}\label{p.projectionofrectanglepathsingeomtype}
Let $G=(n,(h_i)_{i \in \llbracket 1,n \rrbracket}, (v_i)_{i\in \llbracket 1,n \rrbracket}, \mathcal{H}, \mathcal{V},\phi, u)$ be one of the geometric types associated to $\mathcal{R}$ by Theorem \ref{t.associatemarkovfamiliestogeometrictype}. Endow $\mathcal{F}^{s,u}$ with orientations and choose  $r_1,...,r_n$ representatives of every rectangle orbit in $\mathcal{R}$, so that $\mathcal{R}$ is canonically associated to $G$, thanks to Lemma \ref{l.geomtypeinclass}. We have that: 
\begin{itemize}
    \item any rectangle path in $\mathcal{P}$ corresponds to a unique rectangle path in $G$
    \item conversely, any rectangle path in $G$ corresponds to a unique $\pi_1(M)$-orbit of rectangle paths in $\mathcal{P}$ 
\end{itemize}  
\end{prop}
\begin{proof}
Consider $l_0,l_1,...,l_k$ a rectangle path in $\mathcal{P}$. Once again we will think of $G$ as a set of rectangles $R_1,...,R_n$. By our proof of Theorem \ref{t.associatemarkovfamiliestogeometrictype}, every $l_i$ corresponds to a unique rectangle $R_{l_i}$ of the geometric type $G$. Furthermore, for every $i\in \llbracket 1,k\rrbracket $ $l_{i}$ is a successor or predecessor of $l_{i-1}$, therefore $l_i$ corresponds to a unique horizontal or vertical subrectangle $s_{l_i}$ of $R_{l_{i-1}}$. We therefore associate the rectangle path $l_0,...,l_k$ to the sequence $R_{l_0},s_{l_1},....,R_{l_{k-1}},s_{l_k}, R_{l_{k}}$. Using our proof of Theorem \ref{t.associatemarkovfamiliestogeometrictype}, is easy to check that the above sequence is a rectangle path in $G$. 

Consider now  $R_{0},s_{1},....,R_{k-1},s_{k}, R_{k}$ a rectangle path in $G$. By our proof of Theorem \ref{t.associatemarkovfamiliestogeometrictype}, the rectangle $R_{0}$ corresponds to a unique $\pi_1(M)$ orbit of rectangles in $\mathcal{P}$. Let $l_0$ be a representative of this orbit. Next, $s_{1}$ corresponds to a unique successor or predecessor of $l_0$, say $l_1$ and $R_2$ to a unique predecessor or successor of $l_1$, say $l_2$. By repeating the previous procedure, we can associate $R_{0},s_{1},....,R_{k-1},s_{k}, R_{k}$ to the rectangle path $l_0,...,l_k$ in $\mathcal{P}$. The previous rectangle path  depends on our choice of $l_0$. If we replace $l_0$ by $g(l_0)$, where $g\in \pi_1(M)$, it is easy to check that the resulting rectangle path will be the image by $g$ of $l_0,...,l_k$.  We conclude that the rectangle path  $R_{0},s_{1},....,R_{k-1},s_{k}, R_{k}$ is associated to a unique $\pi_1(M)$ orbit of rectangle paths in $\mathcal{P}$. 
\end{proof}

According to Theorem B, an equivalence class of geometric types describes an Anosov flow up to surgeries along boundary periodic orbits. We would like at this point to define a combinatorial object describing an Anosov flow up to just orbital equivalence.  

\begin{defi}
Let $G$ be a geometric type and $\mathcal{A}$ a finite set of closed rectangle paths in $G$. We will call $(G,\mathcal{A})$ a \emph{geometric type with cycles}.

Two geometric types with cycles $(G,\mathcal{A})$ and $(G',\mathcal{A}')$ will be called \emph{equivalent} if there exists an equivalence $h$ between $G$ and $G'$ (see Definition \ref{d.equivalentgeomtypes}) that bijectively associates closed rectangle paths in $\mathcal{A}$ to closed rectangle paths in $\mathcal{A}'$. 
\end{defi}
\begin{rema}
Using Remark \ref{r.finitenumbergeometrictypes}, we can show that any equivalence class of geometric types with cycles contains only a finite number of geometric types with cycles. 
\end{rema}
By Remark \ref{r.canonicalassociationgeometrictype}, for each choice of orientations on $\mathcal{F}^{s,u}$ and each choice of representatives for every rectangle orbit in $\mathcal{R}$, we can associate $\mathcal{R}$ to a unique geometric type $G$. By Remark \ref{r.finitenumbergeometrictypes}, the set of geometric types obtained in this way is finite. 

Fix one of the previous geometric types $G$ (together with the choice of representatives and orientations associated to it). Associate now every boundary periodic point $p$ of $\mathcal{R}$ to the set of closed rectangle paths $\text{Cycl}(p)$ (this association is canonical). By Proposition \ref{p.boundaryperiodicpoints} and Remark \ref{r.orbitofcycles},  the union of all the above cycles corresponds to finitely many closed rectangle paths up to the action of $\pi_1(M)$. Hence, by Proposition \ref{p.projectionofrectanglepathsingeomtype}, we get that we can canonically associate to $\mathcal{R}$ a finite set $\mathcal{A}$ of closed rectangle paths in $G$. We will call $(G,\mathcal{A})$ the geometric type with cycles associated to $\mathcal{R}$ (for this choice of orientations and representatives). 

By repeating our proof of Theorem \ref{t.associatemarkovfamiliestogeometrictype}, it is now not hard to show 

\textbf{Theorem D}\textit{
Let $M$ be an orientable, closed and connected 3-manifold and $\Phi$ a transitive
Anosov flow on $M$. To any Markovian family $\mathcal{R}$ of $\Phi$ we can associate canonically a unique and finite class of equivalent geometric types with cycles.}  
\vspace{0.22cm}

We would like now to prove Theorem E according to which an equivalence clas of geometric types with cycles describes the flow up to orbital equivalence. More precisely,

\textbf{Theorem E}\textit{
Let $(M_1,\Phi_1),(M_2,\Phi_2)$ be two transitive Anosov flows, $\mathcal{R}_1,\mathcal{R}_2$ two Markovian families in $\mathcal{P}_1,\mathcal{P}_2$ associated to the same equivalence class of geometric types with cycles. Then $\Phi_1$ is orbitally equivalent to $\Phi_2$.}
\begin{proof}
Let $\Gamma_1,\Gamma_2$ be the boundary periodic points of $\mathcal{R}_1,\mathcal{R}_2$ and  $\clos{\mathcal{P}_1}$,  $\clos{\mathcal{P}_2}$ the bifoliated planes of $\Phi_1$, $\Phi_2$ up to surgeries along $\Gamma_1$,$\Gamma_2$. Using Lemma \ref{l.geomtypeinclass},  it is not hard to show that by choosing properly orientations on $\mathcal{F}^{s,u}_1,\mathcal{F}^{s,u}_2$ and representatives for every rectangle orbit in $\mathcal{R}_1$ and $\mathcal{R}_2$, we may assume that the Markovian families $\mathcal{R}_1$ and $\mathcal{R}_2$ are associated to the same geometric type with cycles. By Theorem B and Theorem \ref{t.generalbarbot}, there exist an isomorphism $\alpha: \pi_1(M_1-\Gamma_1)\rightarrow \pi_1(M_2-\Gamma_2)$ and a homeomorphism $h:\clos{\mathcal{P}_1}\rightarrow \clos{\mathcal{P}_2}$ such that: 

\begin{itemize}
    \item the image by $h$ of any oriented stable/unstable leaf in $\clos{\mathcal{F}_1^{s,u}}$ is an oriented stable/unstable leaf in $\clos{\mathcal{F}_2^{s,u}}$ 
    \item for every $g\in \pi_1(M_1-\Gamma_1)$ and every $x\in \clos{\mathcal{P}_1}$ we have $$h(g(x))= \alpha(g)(h(x))$$ 
\end{itemize}

Let us now go back to our proof of Theorem \ref{t.generalbarbot} and show that under the additional hypothesis on the cycles of $\mathcal{R}_1$ and $\mathcal{R}_2$, the flows $\Phi_1$ and $\Phi_2$ are orbitally equivalent. 

Indeed, in the proof of Theorem \ref{t.generalbarbot}, for every periodic point $\clos{\gamma}$ in $\clos{\Gamma_1}$ (resp. $\clos{\Gamma_2}$) we defined a canonical basis $s^{M_1}_{\clos{\gamma}},t^{M_1}_{\clos{\gamma}}$ (resp.  $s^{M_2}_{\clos{\gamma}},t^{M_2}_{\clos{\gamma}}$) of its stabilizer group, where $s^{M_1}_{\clos{\gamma}}$ (resp.$s^{M_2}_{\clos{\gamma}}$) fixes all the stable (or unstable) leaves intersecting $\clos{\gamma}$ and $t^{M_1}_{\clos{\gamma}}$ (resp.$t^{M_2}_{\clos{\gamma}}$) acts as a translation by either $+1$ or $+ 2$ on the previous set depending on whether $\clos{\gamma}$ has respectfully negative or positive eigenvalues. By convention, we will assume that $s^{M_1}_{\clos{\gamma}}$ and $s^{M_2}_{\clos{\gamma}}$ act as expansions on every unstable leaf crossing $\clos{\gamma}$. We also proved that for every boundary periodic orbit $\clos{\gamma}\in \clos{\mathcal{P}_1}$ $$\alpha(s^{M_1}_{\clos{\gamma}})= s^{M_2}_{h(\clos{\gamma})}$$
 $$\alpha(t^{M_1}_{\clos{\gamma}})= t^{M_2}_{h(\clos{\gamma})}+ k \cdot s^{M_2}_{h(\clos{\gamma})}$$
where $k$ belongs in $\mathbb{Z}$ and only depends on the projection of the orbit  $\clos{\gamma}$ on $M_1$. We thus associated to every $\gamma\in \Gamma_1$ an integer $k(\gamma)$. Finally, we showed that $\Phi_2$ is orbitally equivalent to the flow obtained by performing a surgery on $\Phi_1$ of coefficient $-k(\gamma)$ along every $\gamma \in \Gamma_1$. Therefore, it suffices to show that for every $\gamma\in \Gamma_1$ we have $k(\gamma)=0$.

Fix $\clos{\gamma}\in \clos{\Gamma_1}$ and $\gamma$ its projection on $\mathcal{P}_1$. Consider a positive cycle around $\gamma$, say  $r_0,...,r_k=r_0$. Lift the previous cycle to a rectangle path $\clos{r_0},...,\clos{r_k}$ in $\clos{\mathcal{P}_1}$ intersecting $\clos{\gamma}$. By the construction of the canonical geometric type with cycles associated to $\mathcal{R}_1$,$\mathcal{R}_2$ (for our choice of orientations and rectangle representatives), every cycle in $\mathcal{P}_1$ corresponds to a cycle of the geometric type associated to $\mathcal{R}_1$ and $\mathcal{R}_2$; therefore it also corresponds to a cycle in $\mathcal{P}_2$. Therefore, $h(\clos{r_0}),...,h(\clos{r_k})$ projects to a closed rectangle path in $\mathcal{P}_2$. 

Since $r_0,...,r_k=r_0$ goes once around $\gamma$, if $\gamma$ has positive (resp. negative) eigenvalues, $t^{M_1}_{\clos{\gamma}}(\clos{r_0})=\clos{r_k}$ (resp. $2\cdot t^{M_1}_{\clos{\gamma}}(\clos{r_0})=\clos{r_k}$). Thus, by applying $h$, if $\gamma$ has positive (resp. negative) eigenvalues, $[t^{M_2}_{h(\clos{\gamma})}+ k(\gamma) \cdot s^{M_2}_{h(\clos{\gamma})}](h(\clos{r_0}))=h(\clos{r_k})$ \big(resp. $[2\cdot t^{M_2}_{h(\clos{\gamma})}+ 2k(\gamma) \cdot s^{M_2}_{h(\clos{\gamma})}](h(\clos{r_0}))=h(\clos{r_k})$\big). By our previous arguments, $h(\clos{r_0})$ and $h(\clos{r_k})$ project to the same rectangles in $\mathcal{P}_2$, which is only possible when $k(\gamma)=0$. 
\end{proof}

\end{document}